\documentclass[12pt]{amsart}
\usepackage[margin = 1in]{geometry}
\usepackage{amsfonts, amsmath,amscd, amssymb,amsthm, cite, latexsym, mathrsfs, mathtools,
slashed, stmaryrd, verbatim, wasysym }
\usepackage[usenames,dvipsnames]{color}
\usepackage[all]{xy}
\usepackage[pdftex]{graphicx}
\usepackage{hyperref}
\usepackage[applemac]{inputenc}
\usepackage[cyr]{aeguill}

\newtheorem{theorem}{Theorem}
\newtheorem{Theorem}{Theorem} 
\newtheorem{corollary}[theorem]{Corollary}
\newtheorem{definition}[theorem]{Definition}

\newtheorem{lemma}[theorem]{Lemma}
\newtheorem{proposition}[theorem]{Proposition}
\newtheorem{remark}[theorem]{Remark}
\newtheorem{assumption}[theorem]{Assumption}

\let\oldsqrt\sqrt
\def\sqrt{\mathpalette\DHLhksqrt}
\def\DHLhksqrt#1#2{%
\setbox0=\hbox{$#1\oldsqrt{#2\,}$}\dimen0=\ht0
\advance\dimen0-0.2\ht0
\setbox2=\hbox{\vrule height\ht0 depth -\dimen0}%
{\box0\lower0.4pt\box2}}

\newcommand\IH{\operatorname{IH}}
\newcommand\bm{\overline{m}}

\newcommand\tb{\operatorname{tb}}
\newcommand\tf{\operatorname{tf}}

\newcommand\ff{\mathrm{ff}}

\newcommand\bs{\mathrm{sb}}
\newcommand\sm{\mathrm{sm}}
\newcommand\ms{sm}

\newcommand\bff{\operatorname{bf}}

\newcommand\GL{\operatorname{GL}}

\newcommand\lf{\operatorname{lf}}
\newcommand\ebf{\operatorname{bf}}
\newcommand\mf{\mathrm{mf}}
\newcommand\lb{\operatorname{lb}}
\newcommand\rb{\operatorname{rb}}
\newcommand\bfo{\operatorname{bf}_0}
\newcommand\zf{\operatorname{zf}}
\newcommand\rf{\operatorname{rf}}

\newcommand\ran{\operatorname{ran}}
\newcommand\spann{\operatorname{span}}

\newcommand\sma{\operatorname{small}}
\newcommand\Span{\operatorname{span}}
\newcommand{\extu}{{\overline{\cup}}}

\newcommand{\tP}{\widetilde{P}}

\newcommand{\df}[1]{\mathfrak{#1}}

\newcommand{\bhs}[1]{\mathfrak B_{#1}}
\renewcommand{\tilde}{\widetilde}
\renewcommand{\bar}{\overline}
\renewcommand{\Re}{\operatorname{Re}}
\renewcommand{\hat}[1]{\widehat{#1}}

\newcommand{\wt}[1]{\widetilde{#1}}
\newcommand{\rest}[1]{\big\rvert_{#1}} 

\newcommand\lra{\longrightarrow}
\newcommand\xlra[1]{\xrightarrow{\phantom{x} #1 \phantom{x}}}

\newcommand\pa{\partial}

\newcommand\cf{cf\@. }

\newcommand\eps\varepsilon
\renewcommand\epsilon\varepsilon

\newcommand\eb{\operatorname{\eps, b}}

\newcommand\ew{\operatorname{\w,\eps}}
\newcommand\ee{\operatorname{e,s}}

\newcommand\w{\operatorname{w}}

\newcommand\dR{\operatorname{dR}}
\renewcommand\sc{\operatorname{sc}}

\newcommand\CI{{\mathcal{C}}^{\infty}}
\newcommand\CIc{{\mathcal{C}}^{\infty}_c}

\newcommand{\lrpar}[1]{\left( #1 \right)}
\newcommand{\lrspar}[1]{\left[ #1 \right]}
\newcommand\ang[1]{\langle #1 \rangle}

\renewcommand\det{\operatorname{det}}

\newcommand\diag{\operatorname{diag}}
\newcommand\Diff{\operatorname{Diff}}
\newcommand\dvol{\operatorname{dvol}}

\newcommand\End{\operatorname{End}}

\newcommand\even{\operatorname{ev}}

\DeclareMathOperator*{\FP}{\operatorname{FP}}

\newcommand\Hom{\operatorname{Hom}}
\newcommand\Id{\operatorname{Id}}

\newcommand\odd{\operatorname{odd}}

\newcommand\phg{\operatorname{phg}}
\newcommand\pt{\operatorname{pt}}

\renewcommand\Re{\operatorname{Re}}
\newcommand\Res{\operatorname{Res}}

\newcommand\spec{\operatorname{spec}}
\newcommand\Spec{\operatorname{Spec}}

\newcommand\Tr{\operatorname{Tr}}

\newcommand\tr{\operatorname{tr}}

\newcommand\pr{\operatorname{pr}}

\newcommand\Mand{\text{ and }}

\newcommand\Mat{\text{ at }}

\newcommand\Mst{\text{ s.t. }}

\newcommand\Mwith{\text{ with }}

\newcommand\paperintro%
        {%
         }
\newcommand\paperbody%
        {%
         }

\newcommand\bbC{\mathbb{C}}

\newcommand\bbN{\mathbb{N}}

\newcommand\bbR{\mathbb{R}}
\newcommand\bbS{\mathbb{S}}

\newcommand\bN{\mathbf{N}}
\newcommand\bR{\mathbf{R}}

\newcommand\cA{\mathcal{A}}
\newcommand\cB{\mathcal{B}}
\newcommand\cC{\mathcal{C}}
\newcommand\cD{\mathcal{D}}
\newcommand\cE{\mathcal{E}}
\newcommand\cF{\mathcal{F}}
\newcommand\cG{\mathcal{G}}
\newcommand\cH{\mathcal{H}}

\newcommand\cJ{\mathcal{J}}
\newcommand\cK{\mathcal{K}}

\newcommand\cQ{\mathcal{Q}}
\newcommand\cR{\mathcal{R}}
\newcommand\cS{\mathcal{S}}
\newcommand\cT{\mathcal{T}}
\newcommand\cU{\mathcal{U}}
\newcommand\cV{\mathcal{V}}

\newcommand\sA{\mathscr{A}}
\newcommand\sB{\mathscr{B}}

\newcommand\sE{\mathscr{E}}

\newcommand\sT{\mathscr{T}}

\DeclareMathAlphabet{\mathpzc}{OT1}{pzc}{m}{it}

\begin{document}

\title[Wedge surgeries]{A Cheeger-M\"uller theorem for manifolds with wedge singularities }
\author{Pierre Albin}
\author{Fr\'ed\'eric Rochon}
\author{David Sher}
\address{Department of Mathematics, University of Illinois at Urbana-Champaign}
\email{palbin@illinois.edu}
\address{D\'epartement de Math\'ematiques, UQ\`AM}
\email{rochon.frederic@uqam.ca }
\address{Department of Mathematical Sciences, DePaul University}
\email{dsher@depaul.edu}

\begin{abstract}
We study the spectrum and heat kernel of the Hodge Laplacian with coefficients in a flat bundle on a closed manifold degenerating to a manifold with wedge singularities.  Provided the Hodge Laplacians in the fibers of the wedge have an appropriate spectral gap, we give uniform constructions of the resolvent and heat kernel on suitable manifolds with corners.  When the wedge manifold and the base of the wedge are odd dimensional, this is used to obtain a Cheeger-M\"ueller theorem  relating analytic torsion with the Reidemeister torsion of the natural compactification by a manifold with boundary.     
\end{abstract}

\maketitle

\tableofcontents

\section*{Introduction}

Analytic torsion was introduced by Ray and Singer \cite{Ray-Singer} as an analytical counterpart to the Reidemeister torsion, a combinatorial invariant introduced by Reidemeister \cite{Reidemeister1935} and Franz \cite{Franz1935} to distinguish lens spaces that are homotopic but not homeomorphic.  Since the Reidemeister torsion can be defined in terms of certain combinatorial Laplacians, the idea of Ray and Singer was to replace these combinatorial Laplacians by the Hodge Laplacian acting on forms and to use the meromorphic continuation of the corresponding zeta functions to define their regularized determinants.  More precisely, if $M$ is a closed oriented odd dimensional manifold, $F\to M$ is a flat Euclidean vector bundle and $g$ is  a Riemannian metric on $M$, then the analytic torsion of $(M,g)$ is given by 
\begin{equation}
  T(M,g,F)= \prod_{q=0}^{\dim M} (\det \Delta_q)^{\frac{q(-1)^{q+1}}2}
\label{int.1}\end{equation}    
where $\Delta_q$ is the Hodge Laplacian acting on forms of degree $q$ taking values in $F$ with regularized determinant given by $\det \Delta_q= e^{-\zeta_q'(0)}$, the function $\zeta_q$ being the zeta function of $\Delta_q$, which for $\Re s\gg 0$ is given by
\begin{equation}
 \zeta_q(s)= \frac{1}{\gamma(s)} \int_0^{\infty} t^s \Tr(e^{-t\Delta_q}- \Pi_{\ker \Delta_q})  \frac{dt}{t}
\label{int.2}\end{equation}
and admits a meromorphic continuation to $\bbC$ with no pole at $s=0$.

In general, analytic torsion does depend on the metric $g$, but if we fix a basis $\{\mu^q_j\}$ of $\mathrm H^q(M;F)$, the cohomology with coefficients in $F$, and if $\{\omega^q\}$ is an orthonormal basis of the space $\ker \Delta_q= \cH^q(M;F)\cong \mathrm H^q(M;F)$ of harmonic forms of degree $q$ taking values in $F$, then it can be shown that the quantity 
\begin{equation}
   \bar T(M,\mu, F):= \frac{T(M,g,F)}{\left(\displaystyle\prod_{q=0}^{\dim M} [\mu^q | \omega^q]^{(-1)^q}\right)}
\label{int.3}\end{equation}
does not depend on the choice of the metric $g$, where $[\mu^q | \omega^q]=|\det W^q|$ with $W^q$ the matrix satisfying
$$
      \mu^q_i = \sum_j W^q_{ij} \omega_j^q.
$$
In particular, if $F$ is acyclic so that $\mathrm H^*(M;F)=\{0\}$, then analytic torsion is independent of the choice of metric.  This suggests that  $\bar T(M,\mu, F)$ is a topological quantity.  
Ray and Singer \cite{Ray-Singer} conjectured that it is in fact equal to the Reidemeister torsion, and this was subsequently established independently by Cheeger \cite{Cheeger1979} and M\"uller \cite{Muller1978}.  One can more generally consider a flat vector bundle $F$ with holonomy representation  \linebreak $\alpha: \pi_1(M) \to \GL(k,\bbR)$ not necessarily orthogonal.  When the holonomy representation is unimodular, M\"uller showed in \cite{Muller1993} that $\bar T(M,g,\mu)$ is still equal to the corresponding Reidemeister torsion.  For arbitrary holonomy representation, Bismut and Zhang \cite{Bismut-Zhang} were able more generally to compute explicitly the ratio between $ \bar T(M,\mu, F)$ and the Reidemeister torsion.  

The Cheeger-M\"uller theorem is a deep result that has important implications in various fields ranging from topology and number theory to mathematical physics.  In particular, it has been used recently by various authors \cite{BV,Calegari-Venkatesh, Muller:AsympRSTHyp3Mfds,Marshall-Muller, Muller-Pfaff:ATL2TorsionCmptLocSymSpaces, Muller-Pfaff:GrowthTorsionCohoArithGps, BMZ, Pfaff:ExpGrowthHomTorTowerCongSubgpsBianchi, Raimbault:Asymp, Raimbault:ARHtorsion} to study the homology of arithmetic groups.   

In some of these settings, one is naturally led to consider non-compact or singular manifolds, notably hyperbolic manifold with cusps.  More generally, since non-compact and singular spaces are ubiquitous in mathematics in physics, it is natural to ask if a Cheeger-M\"uller theorem holds on such a space.  On manifolds having cusp singularities, various Cheeger-M\"uller theorem were recently established \cite{Pfaff:GluingFormATHypMfdsCusps, ARS2, Vertman:CheegerMuller} as well as for fibered cusp metrics \cite{ARS1}.  The corresponding singular spaces are example of stratified spaces, a wide class of singular spaces that includes algebraic variety and quotients of smooth group actions.  More precisely, the singular space geometrically described by a fibered cusp metric is a stratified space of depth one.  Alternatively, one can put on such a space an incomplete edge metric, also called wedge metric.  

More precisely, let $N$ be the interior of $\bar N$, a compact oriented manifold with boundary endowed with a fiber bundle structure
\begin{equation}
\xymatrix{
   Z \ar[r] &  \pa \bar N \ar[d]^{\phi} \\ & Y,
}
\label{int.4}\end{equation}
where the base $Y$ is a compact oriented manifold. The corresponding stratified space, denoted $\widehat{N}$, is then obtained from $\bar N$ by collapsing the fibers of $\phi: \pa \bar N \to Y$ onto the base $Y$.    
Let 
\begin{equation}
c:[0,\delta)\times \pa \bar N\to \bar N
\label{int.4b}\end{equation} 
be a tubular neighborhood of $\pa \bar N$ and let $r\in \CI(\bar N)$ be a compatible boundary defining function in the sense that $r>0$ on $N$ and $c^*r: [0,\delta)\times \pa \bar N \to [0,\delta)$ is the projection on the first factor.  Then a \textbf{product-type} wedge metric on $N$ is a Riemannian metric $g_{\w}$ such that 
\begin{equation}
  c^*g_{\w}= dr^2  + r^2 \kappa + \phi^* g_Y,
\label{int.5}\end{equation}
where $g_Y$ is a Riemannian metric on $Y$ and $\kappa$ is a symmetric 2-tensor on $\pa\bar N$ which restricts to a Riemannian metric on each fiber of $\phi:\pa\bar N\to Y$.   More generally, we can consider  wedge metrics \textbf{even to order $\ell$}, which are wedge metrics asymptotic to $g_{\w}$ in suitable sense, see Definition~\ref{ewm.1} below.

For a choice of even wedge metric $g_{\w}$ on $N$ and a choice of flat vector bundle $F\to \bar N$ with unimodular holonomy representation $\alpha: \pi_1(\bar N)\to \GL(k,\bbR)$, we can define a Hodge Laplacian on $\Delta_{\w}$ acting on forms taking values in $F$ provided we also choose a Euclidean metric $g_F$ on $F$.  Notice that since $\alpha$ is not assumed to be orthogonal, this metric   is in general not compatible with the flat connection of $F$.  We will require $g_F$ to be \textbf{even}, that is, that it has an even expansion in $r$ with respect to a flat trivialization of $F$ in the tubular neighborhood \eqref{int.4b}.  Since the metric $g_{\w}$ is not complete, the operator $\Delta_{\w}$ is usually not essentially self-adjoint, so that there are many choices of self-adjoint extensions.  Among these, a natural choice is the Friedrich extension.  With this choice, the Hodge Laplacian $\Delta_{\w}$ becomes self-adjoint with discrete spectrum, so analytic torsion can be defined as on a compact manifold.  

 In \cite{Mazzeo-Vertman}, Mazzeo and Vertman computed the variation of analytic torsion under a change of even wedge metric for the trivial flat vector bundle, a delicate computation that required a fine understanding of the heat kernel of $\Delta_{\w}$.  There main result is that if $\dim \bar N$ and $\dim Y$ are odd and if $g_{\w}$ is even, asymptotically admissible and such that the indicial roots of the Hodge Laplacian are constant on $Y$, then $\bar T(N,g,\mu)$ does not depend on the choice of metric, suggesting there might be a Cheeger-M\"uller theorem for such metrics.  In fact, when $\phi$ is a bundle of spheres, their result leads to a Cheeger-M\"uller theorem, since then one can deform among wedge metrics to a smooth metric on $\widehat{N}$, which is then a compact smooth manifold, namely the blow-down of $\bar N$, on which the Cheeger-M\"uller theorem is known to hold.  
 
 When $\phi$ is a more general bundle and $\widehat N$ is not smooth,  a natural quantity that could be the topological analogue of  analytic torsion is the Dar $R$-torsion, a generalization of the Reidemeister torsion defined by Dar \cite{Dar} on stratified spaces in terms of the intersection cohomology  of Goresky-MacPherson \cite{GM1980}.  In good cases, one can also hope that analytic torsion is more simply related to the Reidemeister torsion of the manifold with boundary $\bar N$, for instance as in \cite{ARS1}, where the bundle $F$ is assumed to be strongly acyclic at infinity, that is,
 \begin{equation}
   \mathrm H^*(\pa \bar N/Y;F)=\{0\}.
 \label{int.5a}\end{equation}

 In the present paper,  we obtain the following Cheeger-M\"uller theorem for wedge metrics provided the flat vector bundle $F$ is strongly acyclic at infinity and the wedge metric $g_{\w}$ and bundle metrics $g_F$ are even and such that, for the induced de Rham operator $\eth_{\dR,Z_y}$ on each fiber $Z_y=\phi^{-1}(y)$ of the fiber bundle $\phi:H\to Y,$   
\begin{equation}
             \spec(\eth^2_{\dR,Z_y})\cap [0,4]=\emptyset.
\label{int.5b}\end{equation}

\begin{Theorem}
Let $\bar N$ be an odd dimensional oriented manifold with fibered boundary \linebreak $\phi: \pa \bar N\to Y$, where the base $Y$ is a closed odd dimensional oriented compact manifold.  Let $g_{\w}$ be a wedge metric even to order $\dim Y+1$ on the interior $N$ of $\bar N$ with respect to $\phi$ and a choice of boundary defining function $r\in \CI(\bar N)$.  Let $\alpha: \pi_1(\bar N)\to \GL(k,\bbR)$ be a unimodular representation with associated flat vector bundle $F$ such that $\mathrm H^*(\pa\bar N/Y;F)=0$.  Suppose $F$ is equipped with a smooth Euclidean metric $g_F$ on $\bar N$ having an even expansion in $r$ at $\pa \bar N$.  Suppose finally that the metrics $g_{\w}$ and $g_F$ are such that the corresponding de Rham operators on the fibers of $\phi: \pa \bar N\to Y$ satisfy condition \eqref{int.5b}.  Then we have the following equality
$$
   \bar T(N,g_{\w},F,\mu^*)=  \tau(\bar N, \pa \bar N,\alpha, \mu^*) \tau^{\frac12} (\pa \bar N,\alpha)
$$
for any choice of basis $\mu^q$ of $H^q(\bar N,\pa \bar N;F)$ for $q=0,\ldots, \dim \bar N$, where $\tau(\bar N, \pa \bar N,\alpha, \mu^*)$ and $\tau(\pa \bar N,\alpha)$ are the Reidemeister torsions associated to $(\bar N, \pa\bar N,\alpha,\mu^*)$ and $(\pa\bar N,\alpha)$ respectively.  
\label{int.6}\end{Theorem}

\begin{remark}
If $\alpha: \pi_1(\bar N)\to \operatorname{O}(k)$ is an orthogonal representation and $g_F$ is the induced metric compatible with the flat connection, then the formula of the previous corollary simplifies to
$$
  \bar T(N,g_{\w},F,\mu^*)=  \tau(\bar N, \pa \bar N,\alpha, \mu^*)
$$
since $\tau(\pa \bar N,\alpha)=1$ in this case. 
\label{int.7}\end{remark}

\begin{remark}
The strong acyclicity condition at infinity \eqref{int.5a} imposes some conditions on the bundle $F$.  In particular, $F$ cannot be the trivial bundle (so our result applies to a setting disjoint from the one considered in \cite{Mazzeo-Vertman}).  Nevertheless, the condition is often satisfied.  This is because on a compact manifold $Z$, the set of representations $\alpha:\pi_1(Z)\to \GL(k,\bbC)$ leading to an acyclic flat vector bundle is Zariski open, so as soon as there is one acyclic flat vector bundle, any other flat vector bundle is generically acyclic.   
\label{int.8a}\end{remark}

\begin{remark}
Condition \eqref{int.5b} is a technical condition that ensures, among other things, that a certain model operator is invertible, see in particular Proposition~\ref{ws.8} below.  If the bundle $F$ is strongly acyclic, it is very easy to obtain wedge metrics that satisfy this condition by appropriately  scaling the symmetric $2$-tensor $\kappa$ in \eqref{int.5}.  
\label{int.8}\end{remark}

\begin{remark}
This result is quite similar to the Cheeger-M\"uller theorem of \cite{ARS1}.  The main difference is that here, we assume $\dim Y$ is odd, while in \cite{ARS1}, we assume instead that $\dim Y$ is even. 
\end{remark}

Our strategy, following  Hassell, Mazzeo and Melrose \cite{mame1, hmm, Hassell} and strongly inspired by \cite{ARS1}, is to compute the limit of analytic torsion for a family of smooth metrics on a compact manifold degenerating to a wedge metric.  We will in fact consider the double $M= \bar N\cup_{\pa \bar N} \bar N$ of $\bar N$.  If we denote by $H\cong \pa \bar N$ the hypersurface that separates the two copies of $N$ in $M$ and let $c:(-\delta,\delta)_x\times H\to M$ be a choice of tubular neighborhood, then we consider on $M$ a family $g_{\ew}$ of smooth metrics such that 
\begin{equation}
   c^* g_{\ew}= dx^2 + (x^2+\epsilon^2)\kappa+ \phi^* g_Y, \quad \epsilon>0.
\label{int.9}\end{equation}
In this way, when $\epsilon=0$, the metric $g_{\w,0}$ restricts to give the metric $g_{\w}$ on each copy of $N$ in $M$.  When $\dim Y=0$, this corresponds to the family of metrics considered by McDonald in his thesis \cite{McDonald}, see Figure~\ref{fig_donutcroissant} for an illustration of such a metric.  
\begin{figure}
	\centering
	\scalebox{.85}{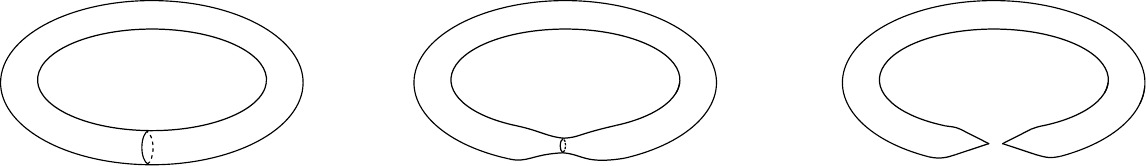}
	\caption{Wedge surgery metric when $\dim Y=0$.}
	\label{fig_donutcroissant}
\end{figure}

To compute the limit of analytic torsion as $\epsilon\to 0$, we construct uniformly the resolvent of the Hodge Laplacian of $g_{\ew}$ as $\epsilon\searrow 0$.  We do this on a suitable double space $X^2_{e, s}$, which is  a manifold with corners on which the behavior of the resolvent as $\epsilon\to 0$ is determined by model operators on two boundary hypersurfaces, one model being the Hodge Laplacian associated to $g_{\w}$ on each copy of $N$ and the other model being a family of suspended scattering Hodge Laplacians.  Condition~\eqref{int.5b} and \cite{GH1} imply that  this family of suspended scattering operators is invertible. This yields a uniform construction of the resolvent, see Theorem~\ref{ur.8} below, which implies in particular that there is a uniform spectral gap around zero for the family of Hodge Laplacian associated to $g_{\ew}$ as $\epsilon\searrow 0$.  
\begin{remark}
When the wedge Laplacian on $N$ has a trivial $L^2$-kernel, this spectral gap can be obtained more directly following the approach of \cite[Proposition~9]{SS16}.  In the general case however, the proof of our main result also requires Theorem~\ref{ur.8} to get good control on the behavior of the $L^2$-kernel as $\eps\searrow 0$, see in particular Corollary~\ref{small.2} below.  
\end{remark}
\begin{remark}
We do not pursue it here, but using \cite{GH2} and \cite{Guillarmou-Sher}, we expect that our methods can be adapted to construct uniformly the resolvent under weaker hypotheses than \eqref{int.5a} and \eqref{int.5b}, for instance, as in \cite{ARS1}, for situations where there are finitely many positive small eigenvalues. 
\end{remark}

We also need a uniform construction of the heat kernel as $\epsilon\searrow 0$.  This is achieved on a suitable heat space, see Theorem~\ref{uhk.1} below.  This allows us to obtain the asymptotics of the trace of heat kernel as $t$ or $\epsilon$ goes to zero.  It is only at that stage that we require the metrics $g_{\w}$ and $g_F$ to be even and the base $Y$ to be odd dimensional.  Indeed, these extra hypotheses crucially rule out the presence of certain undesired terms in the asymptotic of the trace.     Combining with our uniform construction of the resolvent, this allows us to show that the limit of analytic torsion as $\epsilon\searrow 0$ is the analytic torsion of $g_{\w}$ on the two copies of $N$ in $M$.  The main result then follows from a corresponding formula for the Reidemeister torsion, \cf \cite[Theorem~8.7]{ARS1}.

When $\dim Y$ is even, our strategy does not quite work since when we take the limit as $\epsilon\searrow 0$, we obtain an extra term that we are unable to compute explicitly.  This is consistent with the variation formula of \cite{Mazzeo-Vertman}, which suggests in this case that analytic torsion should depend on the geometry of the link of the wedge, not just on its topology.  Nevertheless, in agreement with this prediction, we can give a relative Cheeger-M\"uller theorem, see Corollary~\ref{cm.9} below, essentially a reformulation of the gluing formula of Lesch \cite[Theorem~1.2]{Lesch2013} that can also be seen as a particular case of \cite{Guillarmou-Sher} or of the recent announcement of Ludwig \cite{Ludwig} when $\dim Y=0$.  

Finally we remark that a similarly geometric construction of the heat kernel for some Laplace-type operators associated to wedge metrics has been carried out in \cite{Albin-GellRedman} for stratified spaces of arbitrary depth. It would be interesting to consider a similar degenerating family of metrics, and the resulting behavior of the heat kernels, in this case.

The paper is organized as follows.  In \S~\ref{w.1}, we review the theory of elliptic wedge operators for the de Rham and  Hodge operators. In \S~\ref{wso.0}, we introduce the notion of wedge surgery differential operators and develop a corresponding pseudodifferential calculus.  This is used in \S~\ref{ur.0} to construct uniformly the resolvent of the Hodge Laplacian under a wedge surgery provided \eqref{int.5a} holds.  In \S~\ref{wshs.1}, we introduce a suitable heat space and heat calculus which are used in \S~\ref{hk.0} to provide a uniform construction of the heat kernel of the Hodge Laplacian under a wedge surgery, again provided \eqref{int.5a} holds.  Finally, in \S~\ref{wt.0}, this allows us to describe the asymptotics of the trace of the heat kernel, which combined with the resolvent construction yields our main result in \S~\ref{cm.0}.

{\bf Acknowledgements.}
P. A. was supported by Simons Foundation grant \#317883 and NSF grant DMS-1711325.
F. R. was supported by a Canada Research Chair and NSERC.
The authors are happy to acknowledge useful conversations with Chris Kottke and Rafe Mazzeo.

\numberwithin{equation}{section}
\numberwithin{theorem}{section}

\section{Elliptic wedge operators} \label{w.1}

Let $N$ be the interior of a manifold with boundary $\bar N$ and assume that $\pa \bar N$ participates in a fiber bundle
\begin{equation*}
	Z - \pa \bar N \xlra{\phi} Y.
\end{equation*}
A one-form on $\bar N$ is a {\bf wedge one-form} if its restriction to $\pa \bar N$ is conormal to the fibers of $\phi,$ or equivalently, it is in
\begin{equation*}
	\cV_{\w}^* = \{ \omega \in \CI(\bar N; T^*\bar N) \; | \;
	\omega(\xi)=0 \quad  \forall \xi\in T(\pa\bar N/Y) \}.
\end{equation*}
The Serre-Swan theorem guarantees that there is a vector bundle, which we call the {\bf wedge cotangent bundle}, and denote ${}^{\w}T^*\bar N \lra \bar N,$ together with a bundle map $j: {}^{\w} T^*\bar N \lra T^*\bar N$ such that
\begin{equation*}
	j_*\CI(\bar N;{}^{\w} T^*\bar N) = \cV_{\w}^* \subseteq \CI(\bar N; T^*\bar N).
\end{equation*}
Note that $j$ is a vector bundle isomorphism over $N.$

Let $r$ be a boundary defining function on $\bar N.$ That is, $r$ is a smooth non-negative function on $\bar N$ that vanishes exactly at $\pa \bar N$ and never vanishes to more than first order. Let $\cC(\pa \bar N)$ be a collar neighborhood of $\pa \bar N$ in $\bar N$ of the form $[0,1)_r \times \pa \bar N.$ We extend $\phi$ trivially to $\cC(\pa \bar N)$ without reflecting this in the notation. We will often make use of partial coordinates in $\cC(\pa \bar N)$ of the form $r, y, z$ wherein $y$ denotes coordinates along $Y$ and $z$ denotes coordinates along $Z.$ In these coordinates, the sections of the wedge cotangent bundle are locally generated by
\begin{equation*}
	dr, \quad r \; dz, \quad dy.
\end{equation*}
Crucially, $r\; dz$ does not vanish at $\pa \bar N$ as a section of ${}^{\w}T^*\bar N$ (although it does vanish as a section of $T^*\bar N$).

We refer to the dual bundle to ${}^{\w} T^*\bar N$ as the {\bf wedge tangent bundle} and denote it ${}^{\w} T\bar N.$ In local coordinates, its sections are spanned by
\begin{equation*}
	\pa_r, \quad \tfrac1r \pa_z, \quad \pa_y,
\end{equation*}
and as before a key fact is that $\tfrac1r \pa_z$ is not singular at $\{r=0\}$ as a section of ${}^{\w} T\bar N.$

We point out that these bundles are rescaled versions of the edge tangent bundle and edge cotangent bundle of Mazzeo \cite{maz91}
\begin{equation*}
	{}^{\w} T^*\bar N = r {}^e (T^*\bar N), \quad
	{}^{\w} T\bar N = \tfrac1r ({}^e T\bar N).
\end{equation*}
We will often make use of geometric objects related to $b$-geometry \cite{MelroseAPS}. The $b$-tangent bundle is the edge tangent bundle with respect to the trivial boundary fiber bundle \linebreak $\pa \bar N \lra \pt.$ The $b$-vector fields are the vector fields that are tangent to the boundary.

\begin{definition} A {\bf wedge metric} is a bundle metric on the wedge tangent bundle. A wedge metric is of {\bf product-type} if in some collar neighborhood $\cU$ it takes the form
\begin{equation*}
	dr^2 + r^2 g_Z + \phi^*g_Y,
\end{equation*}
where $g_Z + \phi^*g_Y$ is a Riemannian submersion metric on $\pa \bar N$ (neither $g_Z$ nor $g_Y$ is allowed to depend on $r$).   Such a metric induces  a decomposition
$\left.{}^{\w}T^*\bar N\right|_{\cU} = {}^{\w}T^*_H\bar N\oplus {}^{\w}T^*_V\bar N$ in terms of horizontal and vertical forms with respect to the fibration $\cU\to Y$ induced by $\phi:\pa \bar N\to Y$.
An {\bf even wedge metric} is a wedge metric differing from such a choice of product-type metric by elements of 
$r^2\CI(\cU;{}^{\w}T_H^*\bar N\otimes {}^{\w}T_H^*\bar N)$ and $r^2\CI(\cU;{}^{\w}T_V^*\bar N\otimes {}^{\w}T_V^*\bar N)$  having only even powers of $r$ in their expansion at $\pa \bar N$.  Notice that the notion of eveness  is with respect to the wedge cotangent bundle, so for instance $r^2(dy\otimes dz)=r(dy\otimes rdz)$ is not an even term in the expansion.  More generally, we say that a wedge metric is \textbf{even to order $\ell$} if it differs from an even metric by an element of $r^{\ell}\CI(\bar N; {}^{\w}T^*\bar N\otimes {}^{\w}T^*\bar N)$.  Finally, we say that a wedge metric is \textbf{exact} if it differs from a product-type wedge metric by an element of $r^{2}\CI(\bar N; {}^{\w}T^*\bar N\otimes {}^{\w}T^*\bar N)$.
\label{ewm.1}\end{definition}

Sections of ${}^e T\bar N$ are known as {\bf edge vector fields} and are locally spanned by $r\pa_r,$ $r\pa_y,$ $\pa_z.$ The enveloping algebra of the edge vector fields consists of the {\bf edge differential operators},
\begin{equation*}
	P \in \Diff^m_e(N) \iff
	\text{in local coordinates, }
	P = \sum_{j+|\alpha|+|\beta|\leq m} 
		a_{j,\alpha, \beta}(r,y,z) (r\pa_r)^j (r\pa_y)^{\alpha}\pa_z^\beta,
\end{equation*}
where $\alpha$ and $\beta$ are multi-indices and each $a_{j,\alpha,\beta}$ is a smooth function. Edge differential operators between sections of vector bundles $E,$ $F,$ $\Diff^*_e(N;E,F),$ are defined in the same way but with $a_{j,\alpha,\beta}$ a local section of the homomorphism bundle between $E$ and $F.$

The (principal) symbol of an edge differential operator is the section over the edge cotangent bundle obtained, e.g., by replacing $r\pa_r,$ $r\pa_y,$ and $\pa_z$ by dual variables on ${}^{e}T^*\bar N$, see \cite{maz91}. Ellipticity refers to the invertibility of this symbol off of the zero section.

A differential operator $P$ on $\bar N$ is a {\bf wedge differential operator} of order $m$ if $r^mP$ is an edge differential operator of order $m,$
\begin{equation*}
	\Diff^m_{\w}(\bar N;E,F) = r^{-m}\Diff^m_{e}(\bar N;E,F).
\end{equation*}
There is a corresponding notion of symbol and ellipticity.

Beyond the symbol, edge and wedge differential operators have a second model operator at the boundary. This is known as the normal operator and consists, for each point $p \in Y,$ of the operator $N_y(r^m P)$ on the model space $\bbR^+_r \times T_pY \times Z_p$ obtained by freezing coefficients at $p$ (see for instance \cite{maz91, Krainer-Mendoza}).
$ $\\

Let $g_{\w}$ be an exact wedge metric on $\bar N.$
Let $E \lra N$ be a Euclidean vector bundle on $\bar N$ and let $L^2_{\w}(\bar N;E)$ be the space of sections of $E$ that are square-integrable with respect $g_{\w}$ and the Euclidean metric $g_E$ of $E$.
Suppose that $P$ is an elliptic wedge differential operator of order $m,$ formally self-adjoint with respect to the $L^2$-inner product of $L^2_{\w}(\bar N;E)$.  Finally suppose that, with respect to the warped product metric $h_{\w}=ds^2+ du^2 +s^2 g_{\phi^{-1}(y)}$, where $g_{\phi^{-1}(y)}$ is some choice of metric on the fiber $\phi^{-1}(y)$,  the normal operator 
\begin{equation}
N_y(r^mP)\; \mbox{is injective when acting on} \; s^{a}L^2_{h_{\w}}(\bbR^+_s \times \bbR^h_u\times \phi^{-1}(y))\quad \forall y\in Y, \; \forall a\in [0,m].
\label{hyp.1}\end{equation}
   Before proceeding further in studying the properties of such an operator $P$, let us immediately mention two important examples that the reader should keep in mind.   

\begin{proposition}
If $F\to \bar N$ is a flat vector bundle on $\bar N$ with a Euclidean metric $g_{F}$ not necessarily compatible with the flat connection of $F$, then the Hodge Laplacian $\Delta_{\w}$ in $r^{-2}\Diff^2_e(\bar N;E)$ and the de Rham operator $\eth_{\w}\in r^{-1}\Diff^1_e(\bar N;E)$, where $E= \Lambda^*({}^{\w}T^*\bar N)\otimes F$, are formally self-adjoint elliptic wedge operators.  More importantly, $N_y(r\eth_{\w})$ is injective when acting on $s^{a}L^2_{h_{\w}}(\bbR^+_s \times \bbR^h_u\times \phi^{-1}(y))$ for all $a\ge 0$ provided 
\begin{equation}
H^{\frac{v}2}(\phi^{-1}(y);F)=\{0\} \quad  \mbox{and} \quad \Spec(\eth^2_{\dR,Z_y})\cap(0,1)=\emptyset,
\label{cond.1}\end{equation}
where $v$ is the dimension of the fibers of $\phi: \pa\bar N\to Y$.  
Similarly, $N_y(r^2\Delta_{\w})$ is injective when acting on $s^{a}L^2_{h_{\w}}(\bbR^+_s \times \bbR^h_u\times \phi^{-1}(y))$ for all $a\ge 0$ provided
\begin{equation}
H^{\frac{v+j}2}(\phi^{-1}(y);F)=\{0\} \; \forall j\in\{-2,-1,0,1,2\}\quad  \mbox{and} \quad \Spec(\eth^2_{\dR,Z_y})\cap(0,4)=\emptyset.
\label{cond.2}\end{equation} 
\label{w.1a}\end{proposition}
\begin{proof}
Ellipticity and formal self-adjointness are clear.   The injectivity of $N_y(r\eth_{\w})$ is a direct consequence of the proof \cite[Lemma~5.5]{ALMP:Witt}.  Notice that technically speaking, the statement of  \cite[Lemma~5.5]{ALMP:Witt} is only $a\in (0,1)$, but the proof applies without change to any $a\ge 0$.  For the injectivity of $N_y(r^2\Delta_{\w})$,  since $\Delta_{\w}=\eth_{\w}^2= \eth_{\w}r^{-1}(r\eth_{\w})$, it suffices to have that $N_y(r\eth_{\w})$ is injective  on $s^{a}L^2_{h_{\w}}(\bbR^+_s \times \bbR^h_u\times \phi^{-1}(y))$ for all $a\ge -1$, which can be shown by applying the proof of \cite[Lemma~5.5]{ALMP:Witt} using \eqref{cond.2} instead of \eqref{cond.1}.

\end{proof}

The operator $P$ has two canonical closed extensions  from $\CIc(\bar N\setminus \pa \bar N;E)$  in $L^2_{\w}(\bar N;E)$, namely the maximal extension
\begin{equation*}
	\cD_{\max}(P) = \{ u \in L^2_{\w}(\bar N;E) : P u \in L^2_{\w}(\bar N;E) \}
\end{equation*}
and the minimal extension
$$
\cD_{\min}(P) = \{ u \in L^2_{\w}(\bar N;E) : \exists \; (u_n) \subseteq
	\CIc(\bar N\setminus \pa \bar N;E) \\
	\Mst u_n \xrightarrow{L^2_{\w}} u \Mand Pu_n \text{ is $L^2_{\w}$-Cauchy} \}.
$$
Clearly, $\cD_{\min}(P)\subset \cD_{\max}(P)$.  We will be particularly interested in finding conditions that ensure that $\cD_{\min}(P)= \cD_{\max}(P)$, that is, that $P$ is essentially self-adjoint.

Since $L^2_{\w}(\bar N;E)= r^{-\frac{v+1}2}L^2_b(\bar N;E)$, where $v$ is the dimension of the fibers of $\phi:\pa \bar N\to Y$ and $L^2_b(\bar N;E)$ is the space of square integrable sections with respect to the $b$-density \linebreak $\nu_b= r^{-v-1}dg_{\w}$, it is useful to consider the related wedge operator $\tP= r^{\frac{v+1}2}Pr^{-\frac{v+1}2}$. Instead of the Sobolev spaces defined in terms of the metric $g_{\w}$, it is more useful to consider weighted edge Sobolev spaces, more precisely we will work with
$$
           H^{m}_{\w}(\bar N;E):= r^{-\frac{n}2}H^m_e(\bar N;E)
$$
where $n=\dim \bar N$ and $H^m_e(\bar N;E)$ is the standard Sobolev space defined in terms of the edge metric $g_e= \frac{g_{\w}}{r^2}$.

Recall from \cite{maz91} that for each $y\in Y$, the indicial family of $r^m\tP$ at $y$ is the family of operators given by
\begin{equation}
   I_y(r^m\tP,\lambda):=  \left. \left(r^{-\lambda}(r^m\tP) r^{\lambda}\right)\right|_{\phi^{-1}(y)}, \quad \lambda\in \bbC.
\label{indf.1}\end{equation}
Let us denote by $\spec_{b,y}(r^m\tP)$ the sets of $\lambda\in \bbC$ for which $I_y(r^m\tP,\lambda)$ is not invertible.
\begin{lemma}
If \eqref{hyp.1} holds, then for any $a\in [0,m]$ such that $a\notin \Re(\spec_{b,y}(r^m\tP))$ for all $y\in Y$, the map 
\begin{equation}
P: r^{a} H^m_{\w}(\bar N;E)\to r^{a-m}L^2_{\w}(\bar N;E)
\label{w.2}\end{equation} is essentially injective.  Furthermore, if we know also that $m-a\notin \Re(\spec_{b,y}(r^m\tP))$ for all $y\in Y$, then the map is actually Fredholm.   \label{lem:w.1}\end{lemma}  
\begin{proof}
Since for $a\in [0,m]$,  $\bar N_y(r^mP)$ is injective when acting on $s^{a}L^2(\bbR^+_s \times \bbR^h_u\times \phi^{-1}(y))$ for all $y\in Y$, we know from the proof of \cite[Theorem~6.1]{maz91} that the map \eqref{w.2} is essentially injective if $a\notin \Re(\spec_{b,y}(r^m\tP))$ for all $y\in Y$, that is, its range is closed and its nullspace is finite dimensional.  On the other hand, since $P$ is essentially self-adjoint, the adjoint of the map \eqref{w.2} is
$$
P: r^{m-a} H^m_{\w}(\bar N;E)\to r^{-a}L^2_{\w}(\bar N;E),
$$    
which again by the proof of \cite[Theorem~6.1]{maz91} is essentially injective if $m-a\notin \Re(\spec_{b,y}(r^m\tP))$ for all $y\in Y$.  Thus, if $\{a, m-a\}\cap \Re(\spec_{b,y}(r^m\tP))=\emptyset$ for all $y\in Y$, the map \eqref{w.2} is essential injective with finite dimensional cokernel, that is, it is Fredholm. 
\end{proof}

This result gives some information about the maximal domain of $P$.

\begin{lemma}
If there exists $\epsilon\in (0,m)$ such that $(0,\epsilon]\cap \Re(\spec_{b,y}(r^m\tP))=\emptyset$ for all $y\in Y$ and if \eqref{hyp.1} holds, then $\cD_{\max}(P)\subset r^{\epsilon} H^m_{\w}(\bar N;E).$
\label{w.3}\end{lemma}
\begin{proof}
For $\epsilon$ as in the statement of the lemma, we know by Lemma~\ref{lem:w.1} that the map 
\begin{equation}
       P: r^{\epsilon} H^m_{\w}(\bar N;E)\to r^{\epsilon-m}L^2_{\w}(\bar N;E)
\label{w.4}\end{equation}
is essentially injective, so there is a parametrix
$$
      G: r^{\epsilon-m}L^2_{\w}(\bar N;E)\to r^{\epsilon} H^m_{\w}(\bar N;E)
$$
      such that  $GP=\Id-\Pi$ where $\Pi$ is some projection on the nullspace of \eqref{w.4}.  Now let $u$ be any element of $\mathcal{D}_{\max}(P)$, so that $u\in L^2_{\w}(\bar N;E)$ and $P u=f\in L^2_{\w}(\bar N;E)$.  We need to show that $u\in r^{\epsilon}H^m_{\w}(\bar N,E)$.  Applying $G$ to the equation $P u=f$ gives
$$
         (\Id-\Pi)u=Gf  \quad \Longrightarrow u= Gf + \Pi u.
$$  
Since $Gf\in r^{\epsilon}H^m_{\w}(\bar N,E)$, it remains to show that $\Pi u \in r^{\epsilon}H^m_{\w}(\bar N,E)$.   By \cite[Corollary~3.24]{maz91}, we know that $$ \Pi u\in H^{\infty}_{\w}(\bar N;E)= \bigcap_{\ell\ge 0} H^{\ell}_{\w}(\bar N;E).$$
Moreover, by the proof of \cite[Theorem~7.3]{maz91}, we know that in fact $\Pi u\in r^{\delta}H^{\infty}_{\w}(\bar N;E)$ for any $\delta>0$ such that $(0,\delta]\cap \Re(\spec_{b,y}(r^m\tP))=\emptyset$ for all $y\in Y$.  In particular, we can take  $\delta=\epsilon$ so that $\Pi u \in r^{\epsilon}H^m_{\w}(\bar N,E)$.

\end{proof}
  
This gives us a simple condition that ensures that $P$ is essentially self-adjoint.  

\begin{proposition}
The operator $P$ is essentially self-adjoint provided 
$$
(0,m)\cap \Re(\spec_{b,y}(r^m\tP))=\emptyset \quad \forall y\in Y.
$$  
\label{w.5}\end{proposition}
\begin{proof}
By Lemma~\ref{w.3}, we then know that 
$$
\mathcal{D}_{\max}(P)\subset \bigcap_{a\in (0,m)} r^a H^m_{\w}(\bar N;E).
$$
In particular, for any $u\in \mathcal{D}_{\max}(P)$, $u_n:= r^{\frac{1}{n}}u$ is a sequence in  $r^m H^m_{\w}(\bar N;E)$ such that 
$$
     u_n\to u \;\mbox{in} \; r^{m-\epsilon}H^m_{\w}(\bar N;E), \quad  P u_n\to P u \; \mbox{in} \;r^{-\epsilon}L^2_{\w}(\bar N;E),
$$
for all $\epsilon>0$.  In particular, $r^{\epsilon}P u_n\to r^{\epsilon}P u$ in $L^2_{\w}(\bar N;E)$.  Thus, for any $$v\in \mathcal{D}_{\max}(P)\subset \bigcap_{a\in (0,m)} r^a H^m_{\w}(\bar N;E),$$ we have that for $\epsilon\in (0,m)$, 
$$
      \langle P u_n, v\rangle_{L^2_{\w}}= \langle r^{\epsilon}P u_n, r^{-\epsilon} v\rangle_{L^2_{\w}}\to \langle r^{\epsilon}P u, r^{-\epsilon}v\rangle_{L^2_{\w}}=\langle P u, v\rangle_{L^2_{\w}}.
$$
We also have that 
$$
\langle P u_n, v\rangle_{L^2_{\w}}=\langle  u_n, P v\rangle_{L^2_{\w}}\to\langle u, P v\rangle_{L^2_{\w}},
$$
showing that $\langle P u, v\rangle_{L^2_{\w}}=\langle u, P v\rangle_{L^2_{\w}}$.  Since $v\in \mathcal{D}_{\max}(P)$ was arbitrary and since $P$ is formally self-adjoint on $L^2_{\w}(\bar N; E)$, this shows that $u\in \mathcal{D}_{\min}(P)$, that is, $\mathcal{D}_{\max}(P)\subset \mathcal{D}_{\min}(P)$.        
\end{proof}

There is also a simple condition that ensures that the minimal domain is a weighted Sobolev space.  

\begin{proposition}(\cite[Theorem~4.2]{GKM})
If $m\notin \Re(\spec_{b,y}(r^m\tP))$ for all $y\in Y$, then $\cD_{\min}(P)= r^mH^m_{\w}(\bar N;E)$.  
\label{w.6}\end{proposition}
\begin{proof}
The strategy is to adapt the proof of the criterion of \cite[Proposition~3.6(2)]{Gil-Mendoza} to the wedge setting.  Clearly, since $\mathcal{C}^{\infty}_c(\bar N\setminus \pa \bar N;E)$ is dense in $r^mH^m_{\w}(\bar N;E)$, we have that  $r^mH^m_{\w}(\bar N;E)\subset \mathcal{D}_{\min}(P)\subset \mathcal{D}_{\max}(P)$, so it suffices to show that 
$$
            r^m H^m_{\w}(\bar N;E)\cap \mathcal{D}_{\max}(P)= r^mH^m_{\w}(\bar N;E)
$$
is a closed extension.  To see this, it suffices to use Lemma~\ref{lem:w.1}, by which we know that  
\begin{equation}
     P:  r^m H^m_{\w}(\bar N;E)\to L^2_{\w}(\bar N;E)
\label{w.7}\end{equation}
is essentially injective.  Indeed, let $\{u_n\}\subset r^mH^m_{\w}(\bar N;E)$ be a sequence such that $u_n\to u$ and 
$P u_n\to f$ in $L^2_{\w}(\bar N;E)$.  Since the map \eqref{w.7} is essentially injective, there exists a parametrix $Q: L^2_{\w}(\bar N;E)\to r^mH^m_{\w}(\bar N;E)$
such that 
$$
     QP= \Id -\Pi: r^mH^m_{\w}(\bar N;E)\to r^mH^m_{\w}(\bar N;E)
$$
where $\Pi: r^mH^m_{\w}(\bar N;E)\to r^mH^m_{\w}(\bar N;E)$ is some projection on the kernel of the map \eqref{w.7}.  In this case, $P u_n\to f$ implies that 
$(\Id-\Pi)u_n\to Qf$ in $r^mH^m_{\w}(\bar N;E)$, so 
$$
    \Pi u_n= -(\Id-\Pi)u_n + u_n\to v= -Qf+u \in L^2_{\w}(\bar N;E).
$$
Since $\Pi u_n\in \ker_{r^mH^m_{\w}} P$ and $\ker_{r^mH^m_{\w}}P\subset r^mH^m_{\w}(\bar N;E)\subset L^2_{\w}(\bar N;E)$ is finite dimensional, the norms of $L^2_{\w}(\bhs{\sm};E_{\sm})$ and $r^mH^m_{\w}(\bar N;E)$ are equivalent on $ \ker_{r^mH^m_{\w}} P$, so that $\Pi u_n\to v$ in $r^mH^m_{\w}(\bar N;E)$.  Hence,
$$
    u_n= \Pi u_n + (\Id-\Pi)u_n\to v+Qf \; \mbox{in} \; r^mH^m_{\w}(\bar N;E),
$$
which shows that $u=v+Qf\in r^mH^m_{\w}(\bar N;E)$, so that $r^mH^m_{\w}(\bar N;E)$ is indeed a closed extension.  \end{proof}

\section{Wedge surgery and the de Rham operator} \label{wso.0}

Let $M$ be an oriented closed manifold and $H\subseteq M$ a co-oriented hypersurface in $M$. We will define a one-parameter family of smooth Riemannian metrics on $M$, which we call \textbf{wedge surgery metrics} or \textbf{$\w,\eps$-metrics}, that degenerate as $\eps\to 0$ to a wedge metric on $M\setminus H$. In particular, if $\bar N$ is a compact manifold with boundary, applying this construction with $M$ equal to the double of $\bar N$ yields a family of metrics degenerating to a wedge metric on each copy of $\bar N$.

Throughout, we assume that $H$ is oriented with trivial normal bundle, and that $H$ is the total space of a fibration:
\begin{equation}
\xymatrix{
       Z \ar[r] & H \ar[d]^{\phi} \\
       &  Y.
}
\label{fb.1}\end{equation}
Fix once and for all a connection for $\phi$ and pick a compatible submersion metric $\phi^*g_Y+g_{H/Y}$, where $g_{H/Y}$ restricts to a metric on each fiber. To define our surgery metrics, let $c: H\times (-\delta,\delta)\to \cT\subset M$ be a tubular neighborhood and $x$ a smooth function defined near $H$  such that $c^*x: H\times (-\delta,\delta)\to (-\delta,\delta)$ is the projection on the second factor,
so that $H=\{x=0\}$ and $dx$ is nonzero along $H$. A \textbf{product type $\w,\eps$-metric} is a family of metrics on $M$, parametrized by $\eps\in(0,1]$, which are smooth down to $\eps=0$ in $M\setminus\cT$ and which, on $\cT$, have the form
\[g_{\w,\eps}=dx^2+(x^2+\eps^2)g_{H/Y}+\phi^*g_Y.\]
We will see below that these are naturally interpreted as bundle metrics on a suitable vector bundle.

\subsection{The single surgery space}

An $\w,\eps$-metric is singular at $(x,\eps)=(0,0)$, that is, at $H\times\{\eps=0\}\subseteq M\times[0,1]$. To resolve this singularity, we define the \textbf{single surgery space}
\[X_s=[M\times [0,1];H\times\{0\}],\]
where the notation denotes radial blow up as in \cite[\S2.2]{mame1}.  
As illustrated in Figure \ref{fig:singlespace}, $X_s$ has two boundary hypersurfaces, which we denote $\bhs{\bs}\cong H\times [-\pi/2,\pi/2]$ and $\bhs{\sm}\cong [M;H]$.
\begin{figure}
	\centering
	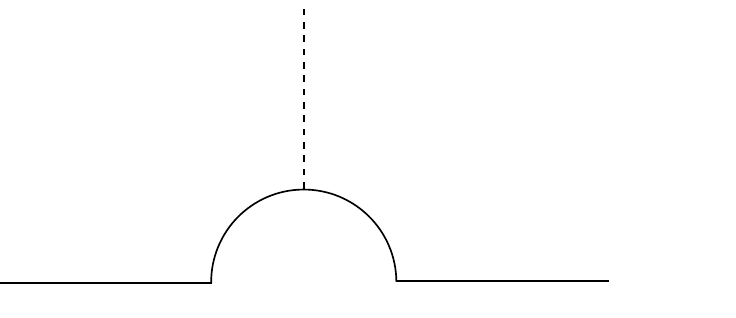
	\caption{The single surgery space $X_s.$}
	\label{fig:singlespace}
\end{figure}
Let $\beta_{s}:X_s\to M\times [0,1]_{\eps}$ be the blow-down map, and let $\pi_{\eps}:X_s\to [0,1]_{\eps}$ be the composition of $\beta_{s}$ with projection onto $[0,1]_{\eps}$.
As observed in \cite[\S~3.2]{mame1},  $\pi_{\eps}$ is a $b$-submersion, so there is a well-defined vector bundle 
\[^{\eps}TX_s:=\ker{}^{b}(\pi_{\epsilon})_*\subseteq {}^{b}TX_s,\]
where ${}^{b}(\pi_{\epsilon})_*$ is the $b$-differential of $\pi_{\epsilon}$.  
The space of \textbf{ edge surgery vector fields} is then defined by
\[\cV_{e,s}(X_s)=\{V\in C^{\infty}(X_s;^{\eps}TX_s)\; | \; V|_{\bhs{\bs}}\textrm{ is tangent to the fibers of }\phi:H\to Y\}.\]
 In local coordinates $(x,y,z,\epsilon)$,
\[\cV_{e,s}=\textrm{span}\{\rho\partial_x,\rho\partial_y,\partial_z\},\]
where $\rho=\sqrt{x^2+\eps^2}$ is a boundary defining function for $\bhs{\bs}$.  By the Serre-Swan theorem, there is a vector bundle $^{e,s}TX_s$, called the $e,s$ \textbf{tangent bundle}, whose sections are elements of $\cV_{e,s}$. The $e,s$ \textbf{cotangent bundle} is the dual $^{e,s}T^*X_s$ of $^{e,s}TX_s$, and it is spanned in local coordinates by
\[\frac{dx}{\rho}, \frac{dy}{\rho}, dz.\]
Note that the restriction of $^{e,s}T^*X_s$ to $\bhs{\sm}$ is the edge cotangent bundle $^{e}T^*\bhs{\sm}$ of Mazzeo \cite{maz91}.

The space $\cV_{e,s}(X_s)$ is closed under Lie the bracket, so is in fact Lie subalgebra of the Lie algebra of vector fields $\CI(X_s; {}^{\eps}TX_s)$.  It can also be characterized geometrically in terms of the $\w,\eps$ metric $g_{\w,\eps}$,
\begin{equation}   
  \cV_{e,s}(X_s)=\{ V\in \CI(X_s; {}^{\eps}TX_s) \; | \; \sup_{p\in X\setminus \bhs{\bs}} \frac{g_{\w,\eps}(V,V)}{\rho^2} <\infty\}.
\end{equation}
In other words, the metric $\frac{g_{\w,\eps}}{\rho^2}$ naturally extends to give a bundle metric for $^{e,s}TX_s$.  We can also associate directly a tangent bundle to $g_{\ew}$, namely the  $\ew$ \textbf{tangent bundle} $^{\ew}TX_s$ obtained by rescaling the $e, s$ tangent bundle,
\[^{\w,\eps}TX_s:=\frac{1}{\rho}\, {}^{e,s}TX_s.\]
Its dual is the $\ew$ \textbf{cotangent bundle}. In local coordinates, sections of these two bundles are spanned by
\[\{\partial_x,\partial_y,\frac{1}{\rho}\partial_z\}\textrm{ and }\{dx,dy,\rho dz\}\textrm{ respectively.}\]

A general \textbf{$\ew$-metric} is a bundle metric on ${}^{\ew}TX_s.$  A $\ew$-metric \textbf{product-type to order $\ell$} is a $\ew$-metric which differs from a product-type $\ew$-metric  by $\rho^\ell$ times a smooth section of ${}^{\ew}T^*X_s \otimes {}^{\ew}T^*X_s.$  An \textbf{exact} $\ew$-metric is a $\ew$-metric product-type to order 2.  \\

The main reason why we are not directly working with $^{\ew}TX_s$ is that, though its sections naturally correspond to vector fields, they are not a Lie subalgebra of $\CI(X_s;{}^{\eps}TX_s)$, and so we cannot directly associate a corresponding algebra of differential operators.  On the other hand, since $\cV_{e,s}(X_s)$ is a Lie algebra, we can define unambiguously its  algebras $\Diff^*_{\ee}(X_s)$  of differential operators as the universal enveloping algebra of $\cV_{e,s}(X_s)$ over $\CI(X_s)$, that is, $P\in \Diff^k_{e,s}(X_s)$ is generated by $\CI(X_s)$ and finite sums of products of at most $k$ vector fields in $\cV_{e,s}(X_s)$.  If $E\to X_s$ is a Euclidean vector bundle, then  we can define the spaces $\Diff^*_{\ee}(X_s;E)$ acting on sections of $E$ in the usual way.  One can then define the space $\Diff^k_{\ew}(X_s;E)$ of $\ew$ differential operators of order $k$ acting on sections of $E$ by
$$
     \Diff^k_{\ew}(X_s;E):= \rho^{-k}\Diff^k_{e,s}(X_s;E).
$$
 As we show below in \eqref{dR.1}, the de Rham operator for a wedge-surgery metric on $X_s$ with coefficients valued in a flat bundle $F\to M$ is an element of Diff$_{\ew}^1(X_s;E)$, where
\[E=\Lambda^*(^{\ew}T^*X_s)\otimes F.\]

There is a natural fiber bundle structure on $\bhs{\bs}$ induced from the fiber bundle structure on $H$, namely
\begin{equation}
\xymatrix{
       [-\frac{\pi}2,\frac{\pi}2]\times Z \ar[r] &  \bhs{\bs}\ar[d]^{\phi_+} \\
       &  Y,
}
\label{se.2a}\end{equation}    
where $\phi_+:= \phi \circ \beta_{s}$.   We can also define the normal bundle
\begin{equation}\label{eq:edgesurgeryconormalbundle}^{\ee}N\bhs{\bs}\to\bhs{\bs}\end{equation}
as the kernel of the natural map of $^{\ee}TX_s$ into $^{\eps}TX_s$ at the face $\bhs{\bs}$. In local coordinates, its sections are spanned over $C^{\infty}(\bhs{\bs})$ by $\rho\partial_y$, as $^{\eps}TX_s$ is spanned by $\{\rho\partial_x,\partial_y,\partial_z\}$.  \begin{lemma}
Multiplication by $\rho$ induces an isomorphism $\phi_+^*(TY)\cong {}^{\ee}\!N\bhs{\bs}$.
\label{se.3}\end{lemma}
\begin{proof}
Notice first that the map $^{\eps}TX_s\to  ^{\ee}\!TX_s$ which to $v\in ^{\eps}TX_s$ associates $\rho v\in {}^{\ee}TX_s$, induces, when restricted to $\bhs{\bs}$, an isomorphism
\begin{equation}
     T\bhs{\bs}/ (T(\bhs{\bs}/Y))\to {}^{\ee}N\bhs{\bs},
\label{se.4}\end{equation}
where $T(\bhs{\bs}/Y)$ is the vertical tangent bundle of the fiber bundle \eqref{se.2a}.  On the other hand, the bundle projection $\phi_+: \bhs{\bs}\to Y$ induces the following short exact sequence of vector bundles
\begin{equation}
0\to T(\bhs{\bs}/Y)\to T\bhs{\bs}\to \phi_+^*TY\to 0,
\label{se.5}\end{equation} 
so that 
\[
        \phi_+^* TY\cong T\bhs{\bs}/T(\bhs{\bs}/Y)\cong {}^{\ee}N\bhs{\bs}.
\]
\end{proof}
\begin{remark}
To make the multiplication by $\rho$ explicit in the isomorphism of Lemma~\ref{se.3}, we use $\rho TY$ to denote the natural vector bundle on $Y$ such that multiplication by $\rho$ induces a natural isomorphism $TY\to \rho TY$.  The dual of $\rho TY$ will be denoted $\rho^{-1}T^*Y$.  Similarly, we will denote by $\eps TY$ the vector bundle on $Y$ such that multiplication by $\eps$ induces a natural isomorphism $TY\to \eps TY$, with dual denoted by $\eps^{-1}T^*Y$. Notice that Lemma~\ref{se.3} can be reformulated as saying that there is a natural identification
\begin{equation}\label{forsuspension}{}^{e,s}N\bhs{\bs}\cong \bhs{\bs}\times_{Y} \rho TY.\end{equation}
\label{nota.1}\end{remark}
\subsection{The de Rham operator}

Let $F\lra M$ be a flat bundle over $M.$
It is convenient to consider the de Rham operator $d+\delta,$ acting on forms with coefficients in $F,$ as an operator on sections of the exterior powers of ${}^{\ew}T^*X_s,$ tensored with $F.$ For any fixed $\eps>0,$ the $\ew$-cotangent bundle is canonically equivalent to $T^*(M \times \{\eps\}),$ and so making this change does not affect the limit. The advantage is that, as an operator on ${\ew}$-differential forms, the de Rham operator of a product-type $\ew$-metric is an $\ew$-differential operator of order one,
\begin{equation*}
	\eth_{\ew} \in \Diff^1_{\ew}(X_s;E), \quad 
	E=\Lambda^*(^{\ew}T^*X_s)\otimes F.
\end{equation*}
Indeed, as in \cite[(4.18)]{Hunsicker-Mazzeo}, in a tubular neighborhood around $H$, the operator  $\eth_{\ew}$ has the form
\begin{equation}
	\eth_{\ew} = 
	\begin{pmatrix}
	\tfrac1\rho \eth_{\dR}^{H/Y} + \hat \eth_{\dR}^Y + \rho \bR & -\pa_x + (\bN_{H/Y} - v) \tfrac x{\rho^2}  \\
	\pa_x + \bN_{H/Y} \tfrac x{\rho^2} & -(\tfrac1\rho \eth_{\dR}^{H/Y} + \hat \eth_{\dR}^Y + \rho \bR) 
	\end{pmatrix},
\label{dR.1}\end{equation}
where $\eth_{\dR}^{H/Y}$ is the de Rham operator in the fibers of \eqref{fb.1},  $\hat \eth_{\dR}^Y$ is a first order horizontal operator involving the second fundamental form of \eqref{fb.1}, $\bN_{H/Y}$ is the number operator in the fibers of \eqref{fb.1} multiplying a form by its vertical degree and $\bR$ is a curvature term.

In this section we will describe the model operators of $\eth_{\ew}.$ The de Rham operator of an exact $\ew$ metric has the same models as that of a product-type $\ew$ metric so we will assume in this section that the metric is of product-type.\\

We are interested in $\eth_{\ew}$ as an unbounded operator on the space $L^2_{\ew}(X_s;E)$ of $L^2$-sections with respect to the density induced by $g_{\w,\eps}$. Equivalently we will work with 
\begin{equation*}
	D_{\ew} = \rho^{(v+1)/2}\eth_{\ew}\rho^{-(v+1)/2}
\end{equation*}
as an unbounded operator on $L^2_{\eb}(X_s;E),$ the space of $L^2$-sections with respect to the $b$-surgery density $\frac{1}{\rho^{v+1}}dg_{\w,\eps}$.  

The restriction of $\eth_{\ew}$ to $\bhs{\sm}(X_s)$ is a differential wedge operator of order one which we denote $\eth_{\w,0}.$ Near $\pa\bhs{\sm}(X_s)$ it has the form
\begin{equation}\label{eq:WedgeNormalOp}
	\eth_{\w,0} = 
	\begin{pmatrix}
	\tfrac1{|x|} \eth_{\dR}^{H/Y} + \hat \eth_{\dR}^Y + |x| \bR 
		& -\pa_x + (\bN_{H/Y} - v) \tfrac 1{|x|}  \\
	\pa_x + \bN_{H/Y} \tfrac 1{|x|} 
		& -(\tfrac1{|x|} \eth_{\dR}^{H/Y} + \hat \eth_{\dR}^Y + |x| \bR) 
	\end{pmatrix}.
\end{equation}
Note that on $\bhs{sm}$,  the subsets where $x$ takes small positive values and where $x$ takes small negative values are separated, so that the function $|x|$ is smooth.
 
On the other hand, $\eth_{\ew}$ does not restrict to $\bhs{\bs}(X_s),$ so we first multiply by $\rho$ before restricting. We define
\begin{equation}\label{eq:VerticalModelOp}
	\eth_v = \eth_{\ew}\rho\rest{\rho=0}
	=
	\begin{pmatrix}
	\eth_{\dR}^{H/Y} & (\bN_{H/Y} - v-1) \tfrac x{\rho}  \\
	(\bN_{H/Y}+1) \tfrac x{\rho} & -\eth_{\dR}^{H/Y} 
	\end{pmatrix}.
\end{equation}

We define $D_{\w}$ and $D_v$ similarly, starting with $D_{\ew}$ instead of $\eth_{\ew}.$ We will now take a closer look at these model operators.\\

\subsection{The model operator at $\bhs{\sm}$}

Let us consider the edge differential operator $\rho D_{\ew}\rest{\bhs{\sm}}$ given, near $\pa \bhs{\sm}(X_s),$ by
\begin{equation*}
	\rho D_{\dR}\rest{\bhs{\sm}} = |x|D_{\w} = 
	\begin{pmatrix}
	\eth_{\dR}^{H/Y}  + |x|\hat \eth_{\dR}^Y  + |x|^2 \bR 
		& -|x|\pa_{|x|} + (\bN_{H/Y} - \tfrac v2+\tfrac12)  \\
	|x|\pa_{|x|} + (\bN_{H/Y}-\tfrac v2-\tfrac12)
		& -(\eth_{\dR}^{H/Y}  + |x|\hat \eth_{\dR}^Y + |x|^2 \bR) 
	\end{pmatrix}.
\end{equation*}
From \cite{maz91}, this has two model operators: an indicial family and a normal operator.

Its indicial family at $y_0 \in Y$ is
\begin{equation*}
	I_b(|x|D_{\w})(y_0;\zeta) = 
	\begin{pmatrix}
	\eth_{\dR}^{H/Y}  
		& -\zeta + (\bN_{H/Y} - \tfrac v2+\tfrac12)  \\
	\zeta + (\bN_{H/Y}-\tfrac v2-\tfrac12) 
		& -\eth_{\dR}^{H/Y} 
	\end{pmatrix},
\end{equation*}
and indicial roots are values of $\zeta$ for which $I_b(|x|D_{\w})(y_0;\zeta)$ is not invertible. The next proposition is a slight improvement over the computation in \cite[\S3.1]{ALMP:Hodge}.

\begin{proposition}
The indicial roots of $|x|D_{\w}$ at $y_0 \in Y$ are
\begin{equation*}
\begin{gathered}
	\{ q-\tfrac v2+\tfrac12, -(q-\tfrac v2 -\tfrac12) : \mathrm H^q(Z_{y_0})\neq 0\} \\
	\bigcup
	\{ \ell \pm \sqrt{\lambda + (q-\tfrac v2 + \tfrac12)^2}: \ell \in \{ 0,1\}, \lambda \in \Spec((\delta d)_q(Z_{y_0}))\setminus \{ 0 \} \} \\
	\bigcup
	\{ \ell \pm \sqrt{\lambda + (q-\tfrac v2 - \tfrac12)^2}: \ell \in \{ 0,1\}, \lambda \in \Spec((d\delta)_q(Z_{y_0}))\setminus \{ 0 \} \}.
\end{gathered}
\end{equation*}
In particular, $\alpha$ is an indicial root if and only if $1-\alpha$ is.
\label{specb.1}\end{proposition}

\begin{proof}
For simplicity we write $\delta = \delta^{H/Y}, d=d^{H/Y}, \delta = (d^{H/Y})^*, \bN = \bN_{H/Y}.$
It suffices, since $I_b(|x|^{1/2}D_{\w} |x|^{1/2})(\zeta)$ is self-adjoint, to find those $\zeta$ for which $I_b(|x|D_{\w})(\zeta)$ is not injective.
Suppose $(a,b)$ is in the null space of $I_b(|x|D_{\w})(\zeta).$ Since $[\eth_{\dR},\bN] = (d^{H/Y})^*-d^{H/Y},$ it will be convenient to use the Kodaira decomposition of $a$ and $b$ with respect to $d^{H/Y},$ which we write $a = a_d + a_{\delta}+a_{\cH}$ and similarly for $b.$
Thus we find,
\begin{equation*}
\begin{gathered}
	\begin{pmatrix}
	\eth_{\dR}  
		& -\zeta + (\bN - \tfrac v2+\tfrac12)  \\
	\zeta + (\bN-\tfrac v2-\tfrac12) 
		& -\eth_{\dR} 
	\end{pmatrix}
	\begin{pmatrix} a \\ b \end{pmatrix} = 
	\begin{pmatrix} 0 \\ 0 \end{pmatrix} 
	\iff
	\begin{cases}
	\eth_{\dR}a = (\zeta +\tfrac v2-\tfrac12 -\bN)b \\
	\eth_{\dR}b = (\zeta -\tfrac v2-\tfrac12 + \bN)a
	\end{cases}
	\\
	\iff
	\begin{cases}
	(\zeta +\tfrac v2-\tfrac12 -\bN)b_{\cH}=0, \quad
	da_\delta = (\zeta +\tfrac v2-\tfrac12 -\bN)b_{d}\\
	\delta a_d = (\zeta +\tfrac v2-\tfrac12 -\bN)b_{\delta}, \quad
	(\zeta -\tfrac v2-\tfrac12 + \bN)a_{\cH}=0 \\
	d b_{\delta} = (\zeta -\tfrac v2-\tfrac12 + \bN)a_d, \quad 
	\delta b_d = (\zeta -\tfrac v2-\tfrac12 + \bN)a_{\delta}.
	\end{cases}
\end{gathered}
\end{equation*}
Thus the harmonic terms contribute
\begin{equation*}
	\{ 1/2 \pm |v/2-q| : \mathrm H^q(Z)\neq 0 \}.
\end{equation*}
For $a_{\delta}$ we have
\begin{multline*}
	\Delta a_{\delta}
	= \delta d a_{\delta} 
	= (\zeta +\tfrac v2-\tfrac32 -\bN)\delta b_{d}
	= (\zeta +\tfrac v2-\tfrac32 -\bN)(\zeta -\tfrac v2-\tfrac12 + \bN)a_{\delta} \\
	= \lrspar{ (\zeta-1)^2 - (\bN-\tfrac v2+\tfrac12)^2 }a_{\delta} 
	\implies 
	\{ \zeta = 1 \pm \sqrt{ \lambda + (q-\tfrac v2+\tfrac12)^2 } : \lambda \in \Spec((\delta d)_q^Z) \}.
\end{multline*}
For $a_d$ we have
\begin{multline*}
	\Delta a_d
	=d\delta a_d
	=(\zeta +\tfrac v2+\tfrac12 -\bN)db_{\delta}
	=(\zeta +\tfrac v2+\tfrac12 -\bN)(\zeta -\tfrac v2-\tfrac12 + \bN)a_d \\
	=\lrspar{ \zeta^2 - (\bN-\tfrac v2-\tfrac12)^2}a_d 
	\implies
	\{ \zeta = \pm \sqrt{\lambda + (q-\tfrac v2 - \tfrac12)^2} : \lambda \in \Spec((d\delta)_q^Z) \}.
\end{multline*}
For $b_\delta$ we have
\begin{multline*}
	\Delta b_{\delta} = \delta d b_{\delta} 
	= (\zeta -\tfrac v2 + \tfrac12 + \bN)\delta a_d
	= (\zeta -\tfrac v2 + \tfrac12 + \bN)(\zeta +\tfrac v2-\tfrac12 -\bN)b_{\delta} \\
	= \lrspar{ \zeta^2 - (\bN-\tfrac v2 + \tfrac12)^2 }b_{\delta} 
	\implies
	\{ \zeta = \pm \sqrt{\lambda + (q-\tfrac v2 + \tfrac12)^2}: \lambda \in \Spec((\delta d)_q^Z) \}.
\end{multline*}
Finally for $b_d$ we have
\begin{multline*}
	\Delta b_d = d \delta b_d
	= (\zeta -\tfrac v2-\tfrac32 + \bN)d a_{\delta}
	= (\zeta -\tfrac v2-\tfrac32 + \bN)(\zeta +\tfrac v2-\tfrac12 -\bN)b_{d} \\
	= \lrspar{ (\zeta-1)^2-(\bN-\tfrac v2 -\tfrac12)^2 }b_d 
	\implies
	\{ \zeta = 1 \pm \sqrt{\lambda +(q-\tfrac v2 -\tfrac12)^2} : \lambda \in \Spec((d\delta)_q^Z) \}.
\end{multline*}

These are necessary conditions, but it is easy to see that they all occur. Indeed,
if $\lambda \in \Spec(\Delta_q^Z)$ and $\Delta u = \lambda u$ with $u \neq 0$ then 
\begin{equation*}
	\Delta u = \Delta(u_d+u_\delta + u _\cH) = d\delta u_d + \delta d u_\delta = \lambda u 
	\implies u_{\cH} =0, \quad \Delta u_d = \lambda u_d, \quad \Delta u_{\delta} = \lambda u_{\delta}
\end{equation*}
and we must have either $u_d$ or $u_{\delta}$ non-zero (so $\Spec(\Delta_q) = \Spec((d\delta)_q)\cup\Spec((\delta d)_q)$).
If $u_d \neq 0,$ let 
\begin{equation*}
	\zeta_1 = \pm \sqrt{\lambda + (q-\tfrac v2 - \tfrac12)^2}, \quad
	\zeta_2 = 1 \pm \sqrt{\lambda +(q-\tfrac v2 -\tfrac12)^2}
\end{equation*}
and set
\begin{equation*}
	\gamma_1 = \frac{\delta u_d}{\zeta_1 +\tfrac v2+\tfrac12-q}, \quad
	\gamma_2 = \frac{\delta u_d}{\zeta_2 -\tfrac v2-\tfrac32+q}
\end{equation*}
then $(u_d, \gamma_1)$ is in the kernel of $I_b(|x|D_{\w})(\zeta_1)$ and $(\gamma_2, u_d)$ is in the kernel of $I_b(|x|D_{\w})(\zeta_2).$
If $u_{\delta}\neq 0,$ let
\begin{equation*}
	\zeta_3 = 1 \pm \sqrt{ \lambda + (q-\tfrac v2+\tfrac12)^2 }, \quad
	\zeta_4 = \pm \sqrt{\lambda + (q-\tfrac v2 + \tfrac12)^2}
\end{equation*}
and set
\begin{equation*}
	\gamma_3 = \frac{du_{\delta}}{\zeta_3+\tfrac v2 -\tfrac32 -q}, \quad
	\gamma_4 = \frac{du_{\delta}}{\zeta_4-\tfrac v2 +\tfrac12 +q}
\end{equation*}
then $(u_{\delta}, \gamma_3)$ is in the kernel of $I_b(|x|D_{\w})(\zeta_3)$ and $(\gamma_4, u_\delta)$ is in the kernel of $I_b(|x|D_{\w})(\zeta_4).$

Note that $d$ maps the $\lambda $ eigenspace of $(\delta d)_q$ isomorphically onto the $\lambda$-eigenspace of $(\delta d)_{q+1}$ and $\delta$ maps the $\lambda$ eigenspace of $(\delta d)_q$ isomorphically onto the $\lambda$ eigenspace of $(d\delta)_{q-1}.$

\end{proof}

Combining this computation with \cite[Theorem~4.7]{Hunsicker-Mazzeo} leads to the following criterion.

\begin{corollary}
The operator $\eth_{\w}:= \left.\eth_{\dR}\right|_{\bhs{\sm}}$ is essentially self-adjoint on $L^2(\bhs{\sm};E_{\sm})$ if
\eqref{cond.1} holds for each $y\in Y$.
\label{wss.1}\end{corollary}
\begin{proof}
By Proposition~\ref{w.1a} and Proposition~\ref{w.5}, the operator $\eth_{\w}$ is essentially self-adjoint if $\Re(\spec_{b,y}(xD_{\w}))\cap (0,1)=\emptyset$.  From Proposition~\ref{specb.1}, we see that this is the case if \eqref{cond.1} holds. 
\end{proof}

Using Proposition~\ref{specb.1}, we can also obtain a sufficient condition for the minimal extension to be given by a weighted Sobolev space. 

\begin{corollary}
The minimal extension of $\eth_{\w}$ is given by $|x|H^1_{\w}(\bhs{\sm};E_{\sm})$ provided for each $y\in Y$,  \eqref{cond.1} holds, $H^{\frac{v\pm 1}2}(\phi^{-1}(y);F)=\{0\}$   and   
$$
1\notin \spec(\left.(\eth_y^{H/Y})^2\right|_{\Omega^{\frac{v\pm 1}2}(\phi^{-1}(y);F)}).
$$

\label{wss.5}\end{corollary}
\begin{proof}
By Proposition~\ref{w.1a} and Proposition~\ref{w.6}, we know that the minimal extension will be given by $|x|H^1_{\w}(\bhs{\sm};E_{\sm})$ provided $1\notin \Re(\spec_{b,y}(|x|D_{\w}))$ for all $y\in Y$.  This condition can be reformulated as in the statement of the proposition using Proposition~\ref{specb.1}.

\end{proof}

We perform a similar analysis of the Hodge Laplacian.
First, let us write it in terms of 
$A = \tfrac{x}{\rho^2}$ and $A' = \pa_x(\tfrac{x}{\rho^2})$ as
\begin{multline*}
	\eth_{\w}^2 = 
	\lrpar{ (\tfrac1\rho \eth_{\dR}^{H/Y} + \hat \eth_{\dR}^Y + \rho \bR)^2  -\pa_x^2 - v A \pa_x}
	\begin{pmatrix} \Id & 0 \\ 0 & \Id \end{pmatrix} \\
	+
	\begin{pmatrix}
	 - \bN_{H/Y} A' + \bN_{H/Y}(\bN_{H/Y}-v)A^2
	& 
	A(-2\tfrac1\rho d^Z + \rho \bR)\\
	A(-2\tfrac1\rho \delta^Z + \rho \bR)
	&
	 (\bN_{H/Y}-v) A' + \bN_{H/Y}(\bN_{H/Y}-v)A^2
	\end{pmatrix}
\end{multline*}
or better
\begin{multline*}
	\rho^{v/2} \eth_{\w}^2 \rho^{-v/2} =
	\lrpar{ (\tfrac1\rho \eth_{\dR}^{H/Y} + \hat \eth_{\dR}^Y + \rho \bR)^2  -\pa_x^2 }
	\begin{pmatrix} \Id & 0 \\ 0 & \Id \end{pmatrix} \\
	+
	\begin{pmatrix}
	 - (\bN_{H/Y}-\tfrac v2) A' + (\bN_{H/Y}-\tfrac v2)^2A^2
	& 
	A(-2\tfrac1\rho d^Z + \rho \bR)\\
	A(-2\tfrac1\rho \delta^Z + \rho \bR)
	&
	(\bN_{H/Y}-\tfrac v2) A' + (\bN_{H/Y}-\tfrac v2)A^2
	\end{pmatrix}
\end{multline*}

\begin{proposition}
The indicial roots of $|x|^2D_{\w}^2$ at $y_0 \in Y$ are
\begin{equation*}
\begin{gathered}
	\{ q-\tfrac v2+\tfrac12, -(q-\tfrac v2 -\tfrac12), q-\tfrac v2+\tfrac32, -(q-\tfrac v2 - \tfrac32) : \mathrm H^q(Z_{y_0})\neq 0\} \\
	\bigcup
	\{ \ell \pm \sqrt{\lambda + (q-\tfrac v2 + \tfrac12)^2}: \ell \in \{ 0,1,2 \}, \lambda \in \Spec((\delta d)_q(Z_{y_0}))\setminus \{ 0 \} \} \\
	\bigcup
	\{ \ell \pm \sqrt{\lambda + (q-\tfrac v2 - \tfrac12)^2}: \ell \in \{ 0,1, 2 \}, \lambda \in \Spec((d\delta)_q(Z_{y_0}))\setminus \{ 0 \} \}.
\end{gathered}
\end{equation*}
In particular, notice that $\delta$ is an indicial root if and only if $2-\delta$ is.


\label{specb.2}\end{proposition}

\begin{proof}
From Proposition~\ref{specb.1}, it suffices to notice that $|x|^2D_{\w}^2= |x|(|x|D_{\w})|x|^{-1}(|x|D_{\w})$, so that $\spec_{b,y}(|x|^2D_{\w})= \spec_{b,y}(|x|D_{\w})\bigcup (\spec_{b,y}(|x|D_{\w})+1)$.

\end{proof}

\begin{corollary}
The wedge Hodge Laplacian $\Delta_{\w}:= \eth_{\w}^2$ 
as an operator on $L^2_{\w}(\bhs{\sm};E_{\sm})$ 
is essentially self-adjoint 
if \eqref{cond.2} holds for each $y\in Y$. 
\label{wss.6}\end{corollary}
\begin{proof}
By Proposition~\ref{w.1a} and Proposition~\ref{w.6}, the wedge Hodge Laplacian will be essentially self-adjoint on $L^2_{\w}(\bhs{\sm};E_{\sm})$ provided $\Re(\spec_{b,y}(|x|^2D_{\w}^2))\cap (0,2)=\emptyset$ for all $y\in Y$, so  the result follows from Proposition~\ref{specb.2}.

\end{proof}

As for the operator $\eth_{\w}$, there is a simple sufficient condition that ensures that the minimal extension of the wedge Hodge Laplacian is a weighted edge Sobolev space.
\begin{corollary}
If   for each $y \in Y$, \eqref{cond.2} holds, $\mathrm H^{\frac{v\pm 3}2}(\phi^{-1}(y);F)=\{0\}$ 
and  
$$4\notin \spec(\left.(\eth_y^{H/Y})^2\right|_{\Omega^{\frac{v\pm1}2}\phi^{-1}(y);F)}),
$$  
then 
$$
      \mathcal{D}_{\min}(\Delta_{\w})= |x|^2 H^2_{\w}(\bhs{\sm};E_{\sm}).
$$
\label{wss.7}\end{corollary}
\begin{proof}
By Proposition~\ref{specb.2}, we have in this case that $2\notin \Re(\spec_{b,y}(|x|^2D_{\w}^2))$ for all $y\in Y$,  so the result follows by applying Proposition~\ref{w.1a} and Proposition~\ref{w.6}.
\end{proof}

For a different point of view on Corollaries~\ref{wss.1}, \ref{wss.5}, \ref{wss.6} and \ref{wss.7} as well as a generalization to wedge metrics on stratified spaces of higher depth, we refer the reader to \cite{HLV}.

\subsection{The model operator at $\bhs{\bs}$}
Analogously to \eqref{eq:VerticalModelOp}, the vertical operator associated to $D_{\ew}$ is
\begin{equation}\label{eq:VerticalModelOp2}
	D_v = D_{\ew}\rho\rest{\rho=0}
	=
	\begin{pmatrix}
	\eth_{\dR}^{H/Y} & -\rho\pa_x +(\bN_{H/Y} - \tfrac v2-\tfrac12) \tfrac x{\rho}  \\
	\rho\pa_x+ (\bN_{H/Y}-\tfrac v2+\tfrac12) \tfrac x{\rho} & -\eth_{\dR}^{H/Y} 
	\end{pmatrix}.
\end{equation}

If we consider this operator in polar coordinates near $\bhs{\bs},$
\begin{equation*}
	\rho = \sqrt{x^2+\eps^2}, \quad \theta = \tan^{-1}\frac x\eps ,
\end{equation*}
in which we have $\beta^*(\pa_x) = \sin\theta \; \pa_\rho + \tfrac1\rho\cos\theta\; \pa_\theta,$ we get
\begin{equation*}
	D_v
	=
	\begin{pmatrix}
	\eth_{\dR}^{H/Y}
		& -\cos\theta\; \pa_{\theta} + (\bN_{H/Y} - \tfrac v2-\tfrac{1}2) \sin\theta  \\
	-\cos\theta\; \pa_{\theta} + (\bN_{H/Y}-\tfrac v2+\tfrac{1}2) \sin\theta 
		& -\eth_{\dR}^{H/Y}
	\end{pmatrix}.
\end{equation*}
In particular we note that 
\begin{equation*}
	D_v \in \Diff^1_b(([-\frac{\pi}2,\frac{\pi}2]_{\theta} \times H)/Y;E|_{\bhs{\bs}}).
\end{equation*}

It is also convenient to consider projective coordinates
\begin{equation*}
	X = \frac x\eps, \quad \eps
\end{equation*}
valid away from $\bhs{\sm}$ and in which $\eps$ is a bdf for $\bhs{\bs}.$
In these coordinates, $\beta^*(\pa_x) = \tfrac1\eps\pa_X,$ $\beta^*(\rho) = \eps\sqrt{1+X^2} = \eps \ang X,$ and hence
\begin{equation*}
	D_v
	=
	\begin{pmatrix}
	\eth_{\dR}^{H/Y}
		& -\ang{X}\pa_X + (\bN_{H/Y} - \tfrac v2-\tfrac{1}2) \tfrac{X}{\ang{X}}  \\
	\ang{X} \pa_X + (\bN_{H/Y}-\tfrac v2+\tfrac{1}2)  \tfrac{X}{\ang{X}} 
		& -\eth_{\dR}^{H/Y}
	\end{pmatrix}.
\end{equation*}
For the Hodge Laplacian, the corresponding model operator is 
\begin{equation}
 \Delta_v:= \left. \rho D^2_{\w,\eps}\rho\right|_{\bhs{\bs}}= \ang{X}D_v\ang{X}^{-1}D_v.
\label{ws.3}\end{equation}
Using the notation
\begin{equation}
   P(a):= \ang{X}\pa_X +a\frac{X}{\ang X} = \ang{X}^{-a}(\ang X \pa_X) \ang{X}^a,  \quad \Delta(a)= -P(-a)P(a),
\label{ws.4}\end{equation}
one computes that 
\begin{equation}
\Delta_v= \begin{pmatrix}
	(\eth_{\dR}^{H/Y})^2+\Delta(N_{H/Y}-\frac{v-1}2)
		&   -\frac{2X}{\ang X} d^{H/Y} \\
	-\frac{2X}{\ang X} \delta^{H/Y}		& (\eth_{\dR}^{H/Y})^2 + \Delta(-(N_{H/Y}-\frac{v+1}2))
	\end{pmatrix}.
\label{ws.5}\end{equation}
For each $y\in Y$, this is an operator on the cylinder  $\bbR_X\times \phi^{-1}(y)$ which is an elliptic $b$-operator of order $2$ in the sense of Melrose \cite{MelroseAPS}.  We can therefore easily characterized when such an operator is Fredholm.  

\begin{lemma}
When restricted to $\bbR\times \phi^{-1}(y)$ for $y\in Y$, the $b$-operator $\Delta_v$ is Fredholm provided 
$$
 \spec(\left.(\eth^{H/Y}_y)^2\right|_{\Omega^{\frac{v\pm1}2}(\phi^{-1}(y),F)})\cap [0,1]=\emptyset, \quad \mbox{if $v$ is odd}
$$
and 
$$
 \spec(\left.(\eth^{H/Y}_y)^2\right|_{\Omega^{\frac{v}2+j}(\phi^{-1}(y),F)})\cap (0,\frac34]=\emptyset \quad \mbox{for} \; j\in \{-1,0,1\}, \quad \mbox{if $v$ is even.}
 $$
\label{ws.6}\end{lemma}
\begin{proof}
By \cite[Theorem~5.40]{MelroseAPS}, the $b$-operator $\Delta_v$ will be Fredholm if and only if its indicial family $I(\Delta_v,\lambda)$ as defined in \eqref{indf.1} is invertible for all $\lambda\in i\bbR$.  Now, by \cite[(2.12)]{ARS2}, we know that 
$$
     I(P(a),\lambda)= \pm(a-\lambda) \quad \mbox{as} \; X\to \pm\infty,
$$
so that 
$$
    I(\Delta(a),\lambda)= -\lambda^2+a^2 \quad \mbox{as} \; X\to \pm \infty.
$$
Hence, as $X\to \pm\infty$, the indicial family of $\Delta_v$ is given by
\begin{equation}
I(\Delta_v,\lambda)= \begin{pmatrix}
	(\eth_{\dR}^{H/Y})^2-\lambda^2+(N_{H/Y}-\frac{v-1}2)^2
		&   \mp2 d^{H/Y} \\
	\mp2 \delta^{H/Y}		& (\eth_{\dR}^{H/Y})^2 -\lambda^2+(N_{H/Y}-\frac{v+1}2)^2
	\end{pmatrix}.
\label{ws.7}\end{equation} 
When we restrict it to the kernel of $(\eth_{\dR}^{H/Y})^2$, this gives
$$
 \begin{pmatrix}
	-\lambda^2+(N_{H/Y}-\frac{v-1}2)^2
		&  0 \\
	 0		&-\lambda^2+ (N_{H/Y}-\frac{v+1}2)^2
	\end{pmatrix},
$$
which is invertible for all $\lambda\in i\bbR$ if and only if $\eth^{H/Y}_{\dR}$ has no non-trivial harmonic forms in degrees $\frac{v\pm1}{2}$.  

Hence, the result will follows if we can show that the indicial family $I(\Delta_v,\lambda)$ is invertible for all $\lambda\in i\bbR$ when we restrict it to the orthogonal complement of the harmonic forms of $\eth^{H/Y}_{\dR}$, in fact to an eigenspace of $(\eth^{H/Y}_{\dR})^2$ associated to an eigenvalue $\nu^2$ with $\nu>0$.  Such a space is always even dimensional with building blocks of dimension 2 given by forms $\alpha$ and $\beta$ of degrees $k$ and $k+1$ such that
\begin{equation}
\begin{gathered}
      (\eth^{H/Y}_{\dR})^2\alpha=\nu^2\alpha, \quad (\eth^{H/Y}_{\dR})^2\beta=\nu^2\beta,\quad \delta^{H/Y} \alpha=0, \\
      d^{H/Y}\alpha=\nu\beta, \quad \delta^{H/Y}\beta=\nu\alpha, \quad d^{H/Y}\beta=0.
\end{gathered}      
\label{ws.6b}\end{equation}
When we restrict to this subspace using $\{\alpha,\beta\}$ as a basis, the indicial family $I(\Delta_v,\lambda)$ becomes
$$
 \begin{pmatrix}
	\nu^2- \lambda^2+(k-\frac{v-1}2)^2 &0 &0 &0 \\
	 0		& \nu^2- \lambda^2+ (k-\frac{v-3}2)^2 & \mp 2\nu & 0 \\
	 0 & \mp 2\nu & \nu^2-\lambda^2+(k-\frac{v+1}2)^2 & 0 \\
	 0 & 0& 0 & \nu^2-\lambda^2+ (k-\frac{v-1}2)^2
	\end{pmatrix}
$$
with determinant given by 
$$
  (\nu^2-\lambda^2+ (k-\frac{v-1}2)^2)^2 [(\nu^2-\lambda^2+ (k+1-\frac{v-1}2)^2)(\nu^2-\lambda^2+ (k-\frac{v+1}2)^2)-4\nu^2].
$$
Given that $\nu^2>0$ and $\lambda\in i\bbR$, this determinant will be zero if and only if $ \nu^4+ b\nu^2 +c=0$ with
$$
b:= -2\lambda^2+ (k-\frac{v-3}2)^2+ (k-\frac{v+1}2)^2-4, \quad c:= (-\lambda^2+ (k-\frac{v-3}2)^2)(-\lambda^2+ (k-\frac{v+1}2)^2)\ge 0,
$$
that is, if and only if 
$$
   \nu^2= -\frac{b}2 \pm \frac{\sqrt{b^2-4c}}2.
$$
Since $\nu^2>0$, a necessary condition for this equality to be true is that $b<0$,  which is the case if and only if $|k-\frac{v-1}2|<1$ and $-\lambda^2$ is sufficiently small.  Thus, if $v$ is odd, this means $k=\frac{v-1}2$, so that for all $\lambda\in i\bbR$, $b\ge -2$, $c\ge 1$ and 
$$
       \nu^2\le \frac{2+ \sqrt{4-4}}{2}=1.
$$
If instead $v$ is even, this means $k=\frac{v}2$ or $k=\frac{v}2-1$, and in either case, $b\ge -\frac32$, $c\ge\frac{9}{16}$, so that we must have in that case
$$
    \nu^2 \le \frac34 + \frac12\sqrt{\frac94-\frac94}=\frac34.
$$
Keeping in mind that $\beta$ is of degree $k+1$, we see that the conditions of the lemma are sufficient to ensure that $\Delta_v$ is Fredholm.
\end{proof}

The $L^2$-kernel of $\Delta_v$ is also easy to describe modulo some further assumption on \linebreak $\spec((\eth^{H/Y}_{\dR})^2)$.  
\begin{proposition}
For $y\in Y$, suppose that $(\phi^{-1}(y),F)$ has no harmonic forms in degree $\frac{v\pm1}2$ and that 
\begin{equation}
    \spec((\eth_y^{H/Y})^2)\cap (0,4]=\emptyset.
\label{ws.8b}\end{equation}
Then the $L^2$-kernel of the restriction $\Delta_{v,y}$ of $\Delta_v$ to $\bbR\times\phi^{-1}(y)$ is given by
\begin{equation}
     \ker_{L^2}\Delta_{v,y}= \Span \left\{ \begin{pmatrix} u\ang{X}^{\frac{v-1}2-k} \\ v\ang{X}^{\ell-\frac{v+1}2}\end{pmatrix} \; |  \; \begin{array}{ll} u\in \rho^k\cH^k(\phi^{-1}(y);F), & k>\frac{v-1}2,  \\ v\in\rho^\ell\cH^{\ell}(\phi^{-1}(y);F), &\ell<\frac{v+1}2 \end{array} \right\}.
\label{ws.8a}\end{equation}
\label{ws.8}\end{proposition}
\begin{proof}
First, if we restrict $\Delta_{v,y}$ to the harmonic forms on the factor $(\phi^{-1}(y),F)$, we obtain
$$
\begin{pmatrix}
	\Delta(N_{H/Y}-\frac{v-1}2)
		&   0  \\
	0	& \Delta(-(N_{H/Y}-\frac{v+1}2))
	\end{pmatrix}.
$$
On the other hand, by \cite[\S~2.1]{ARS2}, we know that 
$$
    \ker_{L^2}\Delta(a)= \left\{  \begin{array}{ll}  c\ang{X}^{-a}, \; c\in \bbC, &  a>0, \\ \{0\}, & a\le 0. \end{array}  \right.
$$
Thus, we see that the $L^2$-kernel of the restriction of $\Delta_{v,y}$ to harmonic forms of $\eth^{H/Y}_y$ gives all of \eqref{ws.8a} already.  To complete the proof, it suffices then to show that $\Delta_{v,y}$ is invertible when restricted to the orthogonal complement of harmonic forms on the factor $(\phi^{-1}(y),F)$.  As in the proof of Lemma~\ref{ws.6}, it suffices in fact to check this on any 2-dimensional eigenspace of $(\eth^{H/Y}_y)^2$ of eigenvalue $\nu^2$ with $\nu>0$ spanned by orthonormal eigenforms $\alpha$ and $\beta$ of degree $k$ and $k+1$ such that \eqref{ws.6b} holds. When we restrict $\Delta_{v,y}$ to such an eigenspace  using the basis $\{\alpha,\beta\}$, we obtain
\begin{equation}
 \begin{pmatrix}
	\nu^2+ \Delta(k-\frac{v-1}2) &0 &0 &0 \\
	 0		& \nu^2 + \Delta(k-\frac{v-3}2) &  -\frac{2X\nu}{\ang X } & 0 \\
	 0 & -\frac{2 X\nu}{\ang X} & \nu^2+\Delta(\frac{v+1}2-k) & 0 \\
	 0 & 0& 0 & \nu^2+ \Delta(\frac{v-1}2-k)
	\end{pmatrix}.
\label{ws.9}\end{equation}
Suppose then that $(u_1,u_2,u_3,u_4)$ is in the $L^2$ kernel of this operator on $\bbR$, which means that 
\begin{equation}
\begin{gathered}
     \Delta(k-\frac{v-1}2)u_1=-\nu^2u_1, \\
     \Delta(k-\frac{v-3}2)u_2=-\nu^2 u_2 + \frac{2X\nu}{\ang X}u_3, \\
     \Delta(\frac{v+1}2-k)u_3= -\nu^2u_3 + \frac{2X\nu}{\ang X}u_2, \\
     \Delta(\frac{v-1}2-k)u_4= -\nu^2u_4.
     \end{gathered}     
\end{equation}
Since the operator $\Delta(a)$ has positive spectrum for all $a\in \bbR$, we see from the first and last equation that $u_1=u_4=0$, while the two equations in the middle implies that 
\begin{equation}
\begin{gathered}
         \nu^2\| u_2\|^2_{L^2}\le 2\nu \ang{ u_2, \frac{X}{\ang X} u_3}_{L^2}\le \nu(\|u_2\|^2_{L^2}+\|u_3\|^2_{L^2} ), \\
          \nu^2\| u_3\|^2_{L^2}\le 2\nu \ang{ u_3, \frac{X}{\ang X} u_2}_{L^2} \le \nu(\|u_2\|^2_{L^2}+\|u_3\|^2_{L^2} ).
\end{gathered}
\end{equation}
Summing these two inequality yields
$$
  \nu (\|u_2\|^2_{L^2}+\|u_3\|^2_{L^2} )\le 2(\|u_2\|^2_{L^2}+\|u_3\|^2_{L^2} )  \quad \Longrightarrow \quad \nu\le 2 \quad \mbox{or} \quad \|u_2\|^2_{L^2}+\|u_3\|^2_{L^2} =0.
$$
Since we must have $\nu>2$ by assumption \eqref{ws.8b}, this means that $u_2=u_3=0$ and the $L^2$-kernel of the operator \eqref{ws.9} is trivial.  Since this operator is clearly essentially self-adjoint and Fredholm by Lemma~\ref{ws.6}, it must therefore be invertible as desired.  
\end{proof}

\section{The edge surgery calculus}

\subsection{The surgery double space}

To define our surgery double space, we start with the $b$-surgery double space $X^2_{b,s}$ of Mazzeo-Melrose \cite{mame1}, 
\begin{equation*}
	X^2_{b,s}=
	[M\times M\times [0,1]_{\epsilon}; H\times H\times\{0\}; H\times M\times\{0\};
	M\times H\times\{0\}], 
\end{equation*}
with blow-down map 
$$
\beta^2_{b,s}:X^2_{b,s} \lra M^2\times[0,1]_{\eps}.
$$
As in \cite{ARS1}, we denote by $\bhs{\mf}$, $\bhs{\lf}$, $\bhs{\rf}$, and $\bhs{\bff}$ the boundary hypersurfaces  given by the interior lifts under $(\beta^2_{b,s})^{-1}$ of $M\times M\times\{0\}$, $H\times M\times\{0\}$, $M\times H\times\{0\}$, and $H\times H\times\{0\}$ respectively. We also let $D_{b,s}\subseteq X^2_{b,s}$ be the interior lift of $\diag_M\times [0,1]_{\eps}$, where $\diag_M$ is the diagonal in $M^2$.  The face $\bhs{\bff}$ has a canonical decomposition
\begin{equation}\label{eq:BhsBf}
	\bhs{\bff}= H\times H\times [-\tfrac{\pi}2,\tfrac{\pi}2]^2_{ob},
\end{equation}
where $[-\frac{\pi}2,\frac{\pi}2]^2_{ob}$ is the overblown $b$-double space of the interval $[-\frac{\pi}2,\frac{\pi}2]$ defined in \cite[p.41]{mame1} by blowing up the four corners of  $ [-\frac{\pi}2,\frac{\pi}2]^2$.  Using this decomposition, the fibration $\phi: H\to Y$ induces a natural fibration $\phi_b: \bhs{\bff}\to Y\times Y$, where $\phi_b= (\phi\times \phi)\circ \pr$ with $\pr: H\times H\times [-\frac{\pi}2,\frac{\pi}2]^2_{ob}\to H\times H$ the projection on $H\times H$.  To obtain the edge surgery double space, we need to blow up $\phi_b^{-1}(D_Y)$, where $D_Y\subset Y\times Y$ is the diagonal,
\begin{equation*}
	X^2_{e,s} := [X^2_{b,s}; \phi_b^{-1}(D_Y)] \quad \mbox{with blow-down map} \quad \beta^2_{e,s}: X^2_{e,s}\to M^3\times [0,1]_{\eps},
\end{equation*}
see Figure~\ref{fig:wedgedoublespace} below.  
In other words, we need to blow up $\phi_b^{-1}(D_Y)= H\times_{\phi} H\times [-\frac{\pi}2,\frac{\pi}2]^2_{ob}$ in $\bhs{\bff}$,  where $H\times_{\phi} H\subset H\times H$ is the fibered diagonal.  In particular, this face comes naturally with a fibration structure
\begin{equation}
\xymatrix{
       Z^2\times [-\frac{\pi}2,\frac{\pi}2]^2_{ob} \ar[r] & \phi_b^{-1}(D_Y) \ar[d]^{\phi_b} \\
       &  D_Y.
}
\label{se.2}\end{equation}  

\begin{figure}
	\centering
	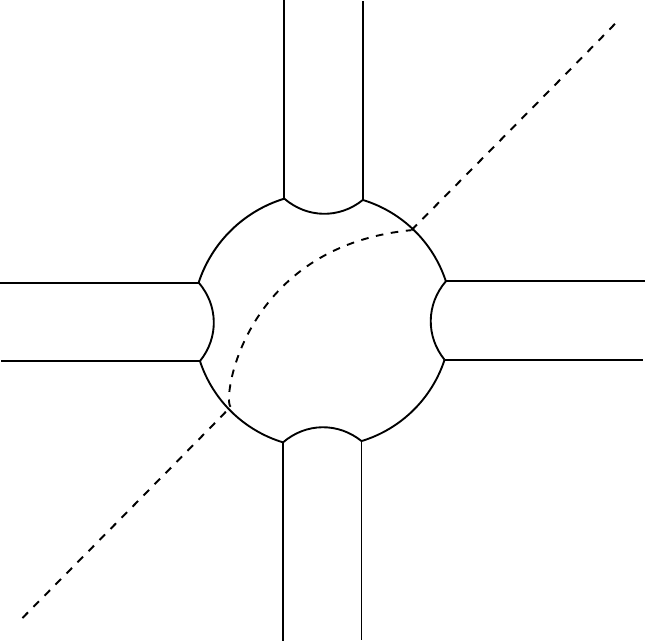
	\caption{The double surgery space $X^2_{e,s}$, obtained by a series of blow-ups from $M\times M\times [0,1]_{\eps}$. Before the blow-ups, coordinates on $M\times M\times [0,1]_{\eps}$ are $(x,x',\eps)$. In this picture, the $x$-axis points upward, the $x'$-axis points to the right, and the $\eps$-axis points up and out of the page.}
	\label{fig:wedgedoublespace}
\end{figure}

Note that to obtain the $\phi$-surgery double space used in \cite{ARS1}, we would instead have blown up the smaller $p$-submanifold $H\times_{\phi}H\times (D_{b,s}\cap\pa  [-\frac{\pi}2,\frac{\pi}2]^2_{ob})$. Denote by $\bhs{\ff}$ the new boundary hypersurface created by blowing up $\phi_b^{-1}(D_Y)$.  Abusing notation slightly, we keep the names of the other boundary hypersurfaces of $X^2_{e,s}$ unchanged under the blow-up from $X^2_{b,s}$. Let $D_{e,s}\subseteq X^2_{e,s}$ be the interior lift of $D_{b,s}\subset X^2_{b,s}$.

We now examine the geometry of $X^2_{e,s}$ more closely. One might expect $\bhs{\mf}$ to be identified with the edge double space $[M;H]_e^2$ of Mazzeo \cite{maz91}, but in fact it is identified with the `overblown' or `extended'  edge double space of Lauter \cite{Lauter} 
\[[M;H]_{e,ob}^2=[[M;H]^2; \{x=x'=0\},\{x=x'=0,y=y'\}].\]
As nested blow-ups commute, $[M;H]_{e,ob}^2$ is a blow-up of $[M;H]_e^2$.

Now consider $\bhs{\ff}$.  Clearly, the normal bundle of $\phi_b^{-1}(D_Y)$ in $\bhs{\bff}$ is identified with $\phi_b^*(N D_Y)$, where $ND_Y$ is the normal bundle of $D_Y$ in $Y\times Y$.  Now, using the projection on the left or on the right factor to identify $D_Y$ with $Y$ gives at the same time a natural identification of $ND_Y$ with $TY$.  Thus, using the fibration
$$
     \phi_b: \phi_b^{-1}(D_Y)\to D_Y\cong Y,
$$
we can identify the normal bundle of $\phi_b^{-1}(D_Y)$ in $\bhs{\bff}$ with $\phi_{b}^*(TY)$.  Thus, since $\bhs{\ff}$ is obtained by blowing up $\phi_b^{-1}(D_Y)$ in $\bhs{\bff}$, we see, making multiplication by $\rho$ explicit, that there is a canonical identification 
\begin{equation}
\bhs{\ff}\cong \overline{\phi_b^* (\rho TY)},
\label{se.6}\end{equation}
where $\overline{\phi_b^* (\rho TY)}$ is the radial compactification of the vector bundle $\phi_b^*( \rho TY)\to \phi_b^{-1}(D_Y)$.    Alternatively, 
\begin{equation}
   \bhs{\ff}\cong (\overline{\rho TY}\times_{Y} (H\times_{\phi} H))\times [-\frac{\pi}2,\frac{\pi}2]^2_{ob}.
\label{se.7}\end{equation}
As a result, observe that $\bhs{\ff}$ is a fiber bundle over $Y$, the fiber of which is given at $y\in Y$ by
\begin{equation}\label{se.7a}
\bhs{\ff,y}:=\overline{\rho T_yY}\times Z_y\times Z'_{y}\times [-\frac{\pi}2,\frac{\pi}2]^2_{ob}.
\end{equation}
\begin{remark}
This could be seen as the double space for $\rho T_yY$-suspended $b$-operators on $\phi^{-1}(y)\times [-\pi/2,\pi/2]$.  However, within the double space $X^2_{e,s}$ this is not quite the type of operators that will appear, since the commutator $[\rho\pa_x,\rho\pa_y]=x\pa_y$ does not vanish on $\bhs{\bs}$ as an $e,s$ vector field. In fact, looking at the restriction of $\cV_{e,s}(X_s)$ to $\bhs{\bs}$, we see that the operators described by $\bhs{\ff,y}$ consist of the subclass of edge operators on the non-compact manifold with boundary $\rho T_yY\times Z_y\times [-\frac{\pi}2,\frac{\pi}2]$, with respect to the boundary fibration $\rho T_yY\times Z_y\times \pa[-\frac{\pi}2,\frac{\pi}2]\to \rho T_yY$, which are invariant by translation in the factor $\rho T_yY$.  
\label{rem.1}\end{remark}

It is straightforward to find local coordinates on $X^2_{b,s}$ and $X^2_{e,s}$. For example, coordinates valid near the interior of $\bhs{\bff}\subseteq X^2_{b,s}$, away from the other boundary hypersurfaces, are
\[\big(\rho=\sqrt{x^2+\eps^2},\theta=\arctan\big(\frac{x}{\eps}\big),\theta'=\arctan\big(\frac{x'}{\eps}\big),y,y',z,z')\big).\]
To obtain $X^2_{e,s}$, we blow up the submanifold $\{\rho=0,y=y'\}$. After this blow-up, coordinates in the interior of $\bhs{\ff}$ are
\begin{equation}\label{coordsff}\big(\rho,\theta,\theta',\hat y=\frac{y-y'}{\rho},y',z,z'\big).\end{equation}
In fact, these coordinates are also valid near the interior of $\bhs{\ff}\cap\bhs{\rf}$ (where \linebreak $\bhs{\rf}=\{\theta'=\pm\pi/2\}$ and $\rho$ is still a defining function for $\bhs{\ff}$), though not near $\bhs{\mf}$ or $\bhs{\lf}$.  The coordinates \eqref{coordsff} in particular nicely illustrate \eqref{se.7}: $y'$ is the coordinate on the base of the fiber product, $\hat y$ is the coordinate on the fibers of $TY$, $(z,z')$ are the coordinates on the fibers of $H\times_{\phi} H$, and $(\theta,\theta')$ are coordinates on $[-\frac{\pi}2,\frac{\pi}2]^2_{ob}$. Note also that the roles of $y$ and $y'$ may be switched.

Near $\bhs{\ff}\cap\bhs{\mf}\cap D_{e,s}$, we can use the projective coordinates
\begin{equation}\label{coordsmf}\big(x',s=\frac{x}{x'},u=\frac{\eps}{x'},\tilde y=\frac{y-y'}{x'},y',z,z'\big).
\end{equation}

As we now check, there are natural projections from $X^2_{e,s}$ onto the single surgery space $X_s$.  
\begin{lemma}
The projections $\pi_R: M\times M \times [0,1]_{\eps}\to M\times [0,1]_{\eps}$ and $\pi_L:M\times M\times [0,1]_{\eps}\to M\times [0,1]_{\eps}$ off the first and second factors lift to b-fibrations $\pi_{e,R}: X^2_{e,s}\to X_s$ and $\pi_{e,L}: X^2_{e,s}\to X_s$.  \label{ws.13}\end{lemma}
\begin{proof}
From \cite{mame1}, we know that we have lifts to $b$-fibrations $\pi_{b,R}: X^2_{b,s}\to X_s$ and $\pi_{b,L}: X^2_{b,s}\to X_s$.  The result then follows easily  by applying \cite[Lemma~2.7]{hmm}.
\end{proof}
The exponent matrices of these $b$-fibrations can be readily computed and contain only $0$ and $1$ entries.  Indeed, for $\pi_{e,R}$, the non-zero entries are for
\begin{equation}
\{(\bhs{\rf},\bhs{\bs}),(\bhs{\ebf},\bhs{\bs}),(\bhs{\ff},\bhs{\bs}),(\bhs{\lf},\bhs{\ms}),(\bhs{\mf},\bhs{\ms})\},
\label{ws.15}\end{equation}
while for $\pi_{e,L}$, the non-zero entries are for
\begin{equation}
\{(\bhs{\lf},\bhs{\bs}),(\bhs{\ebf},\bhs{\bs}),(\bhs{\ff},\bhs{\bs}),(\bhs{\rf},\bhs{\ms}),(\bhs{\mf},\bhs{\ms})\}.
\label{ws.16}\end{equation}

These $b$-fibrations can be used to lift vector fields from the left or from the right.

\begin{proposition}\label{prop:liftedvectorfields} The lift of $\cV_{\ee}(X_s)$ from the left or from the right to $X^2_{e,s}$ is transversal to the lifted diagonal $D_{e,s}$.  \end{proposition}

\begin{proof} By symmetry, it suffices to prove the proposition for lifts from the left.  We calculate in local coordinates using \eqref{coordsff} and \eqref{coordsmf}. First consider a region near $\bhs{\ff}$, away from $\bhs{\lf}$, $\bhs{\mf}$ and $\bhs{\ebf}$. Here the coordinates \eqref{coordsff} are valid. Computing the lifts of each generator, we obtain
\[\pi_{e,L}^*(\rho\partial_x)=\rho\sin\theta\partial_{\rho}+\cos\theta\partial_{\theta};\ \pi_{e,L}^*(\rho\partial_y)=\partial_{\hat y};\ \pi_{e,L}^*(\partial_z)=\partial_z.\]
The transversality for $\cV_{\ee}(X_s)$ follows there from the fact that $D_{e,s}=\{\theta=\theta',\hat y=0,z=z'\}$.  It also holds when $\rho=0$, since on the diagonal in this region, $\theta\neq\pm\pi/2$. 

Now consider a region near $\bhs{\mf}\cap\bhs{\ff}\cap D_{e,s}$, where we have \eqref{coordsmf}. As before,
\[\pi_{e,L}^*(\rho\partial_x)=\sqrt{s^2+u^2}\partial_s;\ \pi_{e,L}^*(\rho\partial_y)=\sqrt{s^2+u^2}\partial_{\tilde y};\ \pi_{e,L}^*(\frac{\eps}{\rho}\partial_z)=\partial_z.\]
In these coordinates, $D_{e,s}=\{s=1,\tilde y=0,z=z'\}$. Since $s=1$ on the diagonal, the result follows.  
\end{proof}

\begin{corollary}\label{prop:canoniso} The normal bundle $N D_{e,s}$ is  canonically isomorphic to $^{e,s}TX_s$.
\end{corollary}

\subsection{Densities}
We follow \cite[\S3.2]{ARS1}. Let $W$ be a manifold with corners. We write $\nu(W)$ for a nonvanishing smooth density on $W$, and $\nu_b(W)$ for a nonvanishing smooth $b$-density on $W$. Note that $\nu_b(W)=f^{-1}\nu(W)$ with $f$ a product of boundary defining functions for each boundary hypersurface of $W$. The sections $\nu(W)$ and $\nu_b(W)$ span the density and $b$-density bundles $\Omega(W)$ and $\Omega_b(W)$ respectively over $C^{\infty}(W)$.

Next we define density bundles on $X_s$. Let $g_{\w,\eps}$ be our wedge surgery metric, then let $\nu_{\ew}=|dg_{\w,\eps}|$. Let $\Omega_{\ew}(X_s)$ be the $\ew$-density bundle, defined so that $\nu_{\ew}$ is a nonvanishing section. Note that in local coordinates,
\[\nu_{\ew}=\rho^v dx\, dy\, dz,\]
where $v$ is the dimension of the fibers of $\phi: H\to Y$.  
The corresponding edge object is $\nu_{\ee}=\rho^{-n}\nu_{\ew}$ and we have
\[\nu_{\ee}=\rho^{-h-1} dx\, dy\, dz,\]
where $h=\dim Y$ and $n=\dim M$.  
Similarly, we let $\Omega_{\ee}(X_s)=\rho^{-n}\Omega_{\ew}(X_s)$. 

Finally, in many situations, it will actually be more convenient to work with the $b$-surgery density bundle $\Omega_{b,s}(X_s)=\rho^{-v-1}\Omega_{\ew}(X_s)$, which has a smooth nonvanishing section $\nu_{b,s}$ given in local coordinates by
\[\nu_{b,s}=\rho^{-1} dx\, dy\, dz.\]

\subsection{Operator calculi}
We follow \cite{ARS1} and let our operators have Schwartz kernels with densities lifted from the right. Let $E\to X_s$ be a vector bundle and let $E^*\to X_s$ be the dual bundle. We define both $b$-surgery and edge-surgery calculi. 

Begin with the small calculi. The small $b$-surgery calculus of \cite{mame1} is the union over $m\in\mathbb R$ of the spaces
\[\Psi^m_{b,s}(X_s;E)=\dot C^{\infty}_{\bff,\mf}I^m(X^2_{b,s},D_{b,s};( (\pi_{b,L})^*E\otimes (\pi_{b,R})^*(E^*\otimes\Omega_{b,s}(X_s)))).\]
The kernels of such operators are polyhomogeneous on $X_{b,s}$, smooth down to $\bhs{\bff}$ and $\bhs{\mf}$, and vanish to infinite order at other faces, with a (one-step polyhomogeneous) conormal singularity of order $m$ at $D_{b,s}$. Similarly, the small edge surgery calculus is defined to be the union over $m\in \bbR$ of
\[\Psi^m_{e,s}(X_s;E)=\dot C^{\infty}_{\ff,\mf}I^m(X^2_{e,s},D_{e,s};( (\pi_{e,L})^*E\otimes (\pi_{e,R})^*(E^*\otimes\Omega_{\ee}(X_s)))).\]

To define wider classes of pseudodifferential operators, we need to recall the notation from \cite{me1,mame1}.
Let $L$ be a manifold with corners, $R$ a total boundary defining function, $\mu$ a positive section of the density bundle $\Omega(L),$ and $\df s$ a multiweight. We denote the conormal functions with multiweight $\df s-$ by
\begin{equation*}
	\sA^{\df s}_-(L;\mu) = \bigcap_{\df s' < \df s} R^{\df s'}H_b^{\infty}(L;\mu).
\end{equation*}

Given an index family $\cE$ and two multiweights $\df s \leq \df w,$ the conormal functions with multiweight $\df s-$ that have partial asymptotic expansion at each boundary hypersurface $H$ with exponents in $\cE(H)$ and remainder term a conormal function with multiweight $\df w-$ are denoted
\begin{equation*}
	\sB^{\cE/\df w}_{\phg}\sA^{\df s}_-(L).
\end{equation*}
Thus, for example, 
\begin{equation*}
\begin{gathered}
	\sA^{\cE}_{\phg}(L) 
	= \sB^{\cE/\infty}_{\phg}\sA^{\df s}_-(L)
	\quad \Mand \quad
	\sA^{\df w}_-(L)  = \sB^{\emptyset/\df w}_{\phg}\sA^{\df s}_-(L) 
\end{gathered}
\end{equation*}
as long as $\df s< \inf \cE$ in the former and $\df s \leq \df w$ in the latter.
We point out that this set of distributions is a $\CI(L)$-module and so extends trivially to sections of a vector bundle, $E \lra L,$
\begin{equation*}
	\sB^{\cE/\df w}_{\phg}\sA^{\df s}_-(L;E)
	= 
	\sB^{\cE/\df w}_{\phg}\sA^{\df s}_-(L) \otimes_{\CI(L)} \CI(L;E).
\end{equation*}

Now, given  index families $\cE$, $\cF$  for $X_{b,s}^2$ and $X_{e,s}^2$, we define 
\[\Psi^{m,\cE}_{b,s}(X_s;E)= \Psi^{m}_{b,s}(X_s;E) + \sA_{\phg}^{\cE}(X^2_{b,s};( (\pi_{b,L})^*E\otimes (\pi_{b,R})^*(E^*\otimes\Omega_{b,s}(X_s))))   ;\]
\[ \Psi^{m,\cF}_{e,s}(X_{s};E)=\Psi^m_{e,s}(X_s;E)+ \sA_{\phg}^{\cF}(X^2_{e,s};( (\pi_{e,L})^*E\otimes (\pi_{e,R})^*(E^*\otimes\Omega_{\ee}(X_s)))) .\]
The large $b$-surgery and  edge-surgery calculi are defined as the unions over all $m$ and all index families $\cE$ and $\cF$ of the corresponding spaces above. 

We also need families of pseudodifferential operators which are conormal but which are polyhomogeneous only up to a certain order. These definitions are modelled on \cite[(78),(79)]{mame1}, with some modifications for our setting. We first set
\[\Psi^{-\infty,\tau,res}_{e,s}(X_s;E)=(\rho_{\ebf}\rho_{\rf})^h\sA^{\tau}_-(X_{e,s}^2;( (\pi_{e,L})^*E\otimes (\pi_{e,R})^*(E^*\otimes\Omega_{\ee}(X_s))).\]
These residual operators are conormal, with order $\tau$ at all boundary hypersurfaces except at $\bhs{\ebf}$ and $\bhs{\rf}$, where they have order $h+\tau$.
 We then set 
\[\Psi^{-\infty,\tau}_{e,s}(X_{s};E):= \sB^{\cE/\df s}_{\phg}\rho_{\ebf}^h\rho_{\rf}^h\sA^{0}_-(X_{e,s}^2;( (\pi_{e,L})^*E\otimes (\pi_{e,R})^*(E^*\otimes\Omega_{\ee}(X_s))))\]
with $\cE$ given by $\bbN_0$ at $\bhs{\ff}$, $\bhs{\mf}$ and by $\emptyset$ elsewhere, and with multi-weight $\df s$ given by $h+\tau$ at $\bhs{\ebf}$, $\bhs{\rf}$ and by $\tau$ elsewhere. These operators have kernels which are
\begin{itemize}
\item partially smooth up to $\bhs{\mf}$ and $\bhs{\ff}$  to order $\tau$ (with conormal remainder);
\item conormal of order $\tau$ at $\bhs{\lf}$; and
\item conormal of order $h+\tau$ at $\bhs{\rf}$ and $\bhs{\ebf}$.
\end{itemize}
Finally, we  define
\[\Psi^{m,\tau}_{e,s}(X_s;E):=\Psi^{m}_{e,s}(X_s;E)+\Psi^{-\infty,\tau}_{e,s}(X_s,E).\]

A symbol map for the $b$-surgery calculus is defined in \cite{mame1}, and there is also a symbol map for the edge surgery calculus.  Indeed for $A\in\Psi^{m,\cF}_{e,s}(X_s;E)$, let its Schwartz kernel be $K_A$. The kernel $K_A$ has a conormal singularity at the lifted diagonal $D_{e,s}$, and as such it has a symbol $\sigma(K_A)$ which is an element of $S^m(N^*D_{e,s};\textrm{End}(E))$. Let the edge surgery symbol of $A$, denoted $^{e,s}\sigma(A)$, be $\sigma(K_A)$. Using Corollary~\ref{prop:canoniso}, $^{e,s}\sigma(A)$ may be viewed as an element of $S^m(^{e,s}T^*X_s;\textrm{End}(E))$. Within the small calculus, this yields the exact sequence
\[0\xrightarrow{}\Psi^{m-1}_{e,s}(X_s;E) \xrightarrow{}\Psi^m_{e,s}(X_s;E) \xrightarrow{^{e,s}\sigma} S^m(^{e,s}T^*X_s;\textrm{End}(E))\xrightarrow{} 0.\]

\subsection{Mapping properties and composition}

The triple edge surgery space $X_{e,s}^3$  is constructed as follows.

We start with $M^3\times [0,1]_{\eps}$ and denote its boundary hypersurface $M^3\times \{0\}$ by $\bhs{Z}$.  The other boundary hypersurface at $\eps=1$ will play no role in our construction and can safely be ignored.  Let us denote by $x,x'$ and $x''$ the functions on each factor of $M$ corresponding to the coordinate normal to $H$ for a choice of tubular neighborhood of $H$.  We will now perform a series of blow-ups.  Each of them will introduce new boundary hypersurfaces and possibly modify the boundary hypersurfaces present before the blow-up.  However, to lighten the notation, for the boundary hypersurfaces already present, we will use the same notation for the boundary hypersurface before and after the blow-up. Let us first recall the construction of the surgery triple space of \cite{mame1}:
\begin{enumerate}
\item Blow up $H\times H\times H\times \{0\}$ and call the new face $\bhs{T}$;
\item Blow up $H\times H\times M\times \{0\}$, $H\times M\times H\times \{0\}$ and $M\times H\times H\times \{0\}$, denoting respectively by $\bhs{F}$, $\bhs{C}$ and $\bhs{S}$ the new boundary hypersurfaces;
\item Blow up  $M\times M\times H\times \{0\}$, $M\times H\times M\times \{0\}$ and  $H\times M\times M\times \{0\}$, calling the new faces $\bhs{N_1}$, $\bhs{N_2}$ and $\bhs{N_3}$, which yields the  surgery triple space $X^3_{b,s}$ of \cite{mame1}.
\end{enumerate}
Recall from \cite{mame1} that $X^3_{b,s}$ is such that the maps $M^3\times [0,1]\to M^2\times [0,1]$ obtained by projecting off one of the copies of $M$ all lift to $b$-fibrations
\begin{equation}
 X^3_{b,s}\to X^2_{b,s}.
\label{bf.1}\end{equation}
To construct the edge surgery triple space, on which the maps \eqref{bf.1} shall lift to $b$-fibrations onto $X^2_{e,s}$, it suffices then by \cite[Lemma~2.5]{hmm} to blow up the lifts of $\phi_b^{-1}(D_Y)$ with respect to the three $b$-fibrations of \eqref{bf.1}.  Since these lifts intersect on the face $\bhs{T}$, we can first blow up their intersection to keep the triple space symmetric with respect to the three $b$-fibrations.  This suggests to construct the edge surgery triple space by performing the following  three other rounds of blow-ups:
\begin{enumerate}\setcounter{enumi}{3}
\item Blow up $\bhs{T}\cap \{y=y'=y''\}$, creating a new boundary hypersurfaces $\bhs{TT}$ analogous to $\bhs{\phi_{TT}}$ in \cite{ARS1}, but with the important difference that $\bhs{TT}$ intersects all other boundary hypersurfaces while $\bhs{\phi_{TT}}$ only intersects $\bhs{T}$ and $\bhs{Z}$;

\item Blow up the interior lifts of $\bhs{T}\cap \{y=y'\}$, $\bhs{T}\cap \{y=y''\}$, and $\bhs{T}\cap \{y'=y''\}$, which are disjoint $p$-submanifolds by the previous blow-up, to get three new faces $\bhs{FT}$, $\bhs{CT}$, and $\bhs{ST}$ respectively;

\item Blow up the interior lifts of $\bhs{F}\cap\{y=y'\}$, $\bhs{C}\cap\{y=y''\}$, and $\bhs{S}\cap\{y'=y''\}$, denoting respectively the new boundary hypersurfaces $\bhs{FD}$, $\bhs{CD}$, and $\bhs{SD}$, to obtain the edge surgery triple space $X^3_{e,s}$.

\end{enumerate}
Let us denote by $\beta^3_{e,s}: X^3_{e,s}\to M^3\times [0,1]_{\eps}$  the corresponding blow-down map.  As the next lemma shows, it is then straightforward to check that on $X^3_{e,s}$, a projection off a factor lifts to a $b$-fibration.  

\begin{lemma}\label{lem:bfib}  The projections off the first, second, and third factor in $M\times M\times M\times [0,1]_{\eps}$ extend by continuity to well-defined $b$-fibrations $\pi_{e,S}$, $\pi_{e,C}$, and $\pi_{e,F}$ respectively from $X_{e,s}^3$ to $X_{e,s}^2$.  
\end{lemma}
\begin{proof}
This is quite similar to the proof of the corresponding statement for the triple surgery space of \cite{ARS1} and amounts to using \cite[Lemma~2.5 and Lemma~2.7]{hmm} as well as the commutativity of nested blow-ups.  Moreover, by symmetry, it suffices to establish our results for $\pi_{e,F}$.  By \cite{mame1}, we can also use the fact that the projections off the third factor of $M^3\times [0,1]_{\eps}$ extend by continuity to a $b$-fibration $\pi_{b,F}:X^3_{b,s} \to X^2_{b,s}$.  

Now, to lift $\pi_{b,F}$ to a $b$-fibration onto $X^2_{e,s}$, we know according to \cite[Lemma~2.5]{hmm} that it suffices to blow-up the pre-images of $\phi_b^{-1}(D_Y)\subset X^2_{b,s}$ with respect to $\pi_{b,F}$, namely $\bhs{T}\cap\{y=y'\}$ and $\bhs{F}\cap \{y=y'\}$.  In other words, there is a lift $\tilde{\pi}_{e,F}: \widetilde{X}^3_{e,s}\to X^2_{e,s}$ of $\pi_{b,s}$ with 
$$
    \widetilde{X}^3_{e,s}:= [X^2_{b,s}; \bhs{T}\cap\{y=y'\}; \bhs{F}\cap\{y=y'\}].
$$  
Since $\bhs{T}\cap\{y=y'=y''\}$ is a $p$-submanifold of $\bhs{T}\cap\{y=y'\}$, one can use the commutativity of nested blow-ups \cite[Lemma~2.1]{hmm} to see that $\widetilde{X}^3_{e,s}$ is a blow-down of $X^3_{e,s}$,
$$
     X^3_{e,s}= [\widetilde{X}^3_{e,s}; \bhs{T}\cap\{y=y'=y''\}, \bhs{T}\cap\{y=y''\}, \bhs{T}\cap\{y'=y''\}, \bhs{C}\cap\{y=y''\}, \bhs{S}\cap\{y'=y''\}].
$$
If $\beta_F: X^3_{e,s}\to \widetilde{X}^3_{e,s}$ is the blow-down map, we see therefore that the desired extension is $\pi_{e,F}:= \tilde{\pi}_{e,F}\circ \beta_F$, which by \cite[Lemma~2.7]{hmm} is a $b$-fibration.
\end{proof}

The exponent matrices of the $b$-fibrations of Lemma~\ref{lem:bfib} only have entries equal to 1 or zero.  
Using \cite[(101)]{mame1} for the corresponding $b$-fibrations of the $b$-surgery triple space, we compute that for $\pi_{e,F}$, the non-zero entries are given by
\begin{multline}\label{expmatrixf}\left\{(\bhs{TT},\bhs{\ff}), (\bhs{FT},\bhs{\ff}), (\bhs{FD},\bhs{\ff}), (\bhs{CT},\bhs{\ebf}), (\bhs{ST},\bhs{\ebf}), \right.\\ \left.(\bhs T,\bhs{\ebf}), (\bhs F,\bhs{\ebf}),(\bhs{SD},\bhs{\rf}),(\bhs S,\bhs{\rf}),(\bhs{N_2},\bhs{\rf}),\right.\\ \left.(\bhs{CD},\bhs{\lf}),(\bhs C,\bhs{\lf}),(\bhs{N_3},\bhs{\lf}),(\bhs{N_1},\bhs{\mf}),(\bhs Z,\bhs{\mf})\right\},\end{multline}
while for $\pi_{e,C}$ they are given by
\begin{multline}\label{expmatrixc}\left\{(\bhs{TT},\bhs{\ff}), (\bhs{CT},\bhs{\ff}), (\bhs{CD},\bhs{\ff}), (\bhs{FT},\bhs{\ebf}), (\bhs{ST},\bhs{\ebf}), \right.\\ \left.(\bhs T,\bhs{\ebf}), (\bhs C,\bhs{\ebf}),(\bhs{SD},\bhs{\rf}),(\bhs S,\bhs{\rf}),(\bhs{N_1},\bhs{\rf}),\right.\\ \left.(\bhs{FD},\bhs{\lf}),(\bhs F,\bhs{\lf}),(\bhs{N_3},\bhs{\lf}),(\bhs{N_2},\bhs{\mf}),(\bhs Z,\bhs{\mf})\right\},\end{multline}
and for $\pi_{e,S}$ they are given by
\begin{multline}\label{expmatrixs}\left\{(\bhs{TT},\bhs{\ff}), (\bhs{ST},\bhs{\ff}), (\bhs{SD},\bhs{\ff}), (\bhs{FT},\bhs{\ebf}), (\bhs{CT},\bhs{\ebf}), \right.\\ \left.(\bhs T,\bhs{\ebf}), (\bhs S,\bhs{\ebf}),(\bhs{CD},\bhs{\rf}),(\bhs C,\bhs{\rf}),(\bhs{N_1},\bhs{\rf}),\right.\\ \left.(\bhs{FD},\bhs{\lf}),(\bhs F,\bhs{\lf}),(\bhs{N_2},\bhs{\lf}),(\bhs{N_3},\bhs{\mf}),(\bhs Z,\bhs{\mf})\right\}.\end{multline}

Now we examine densities and their behavior under blow-up. The following is the analogue of Corollary C.6 in \cite{ARS1}.
\begin{proposition}\label{prop:densityblowup} The canonical density bundles are transformed by blow-ups as follows:
\begin{equation}\label{singlespacedenstrans}
(\beta_s)^*\nu(M\times[0,1]_{\eps})=\rho_{\bs}\nu(X_s);
\end{equation}
\begin{equation}\label{doublespacedenstrans}
(\beta_{e,s}^2)^*\nu(M^2\times[0,1]_{\eps})=\rho_{\lf}\rho_{\rf}\rho^2_{\ebf}\rho_{\ff}^{h+2}\nu(X_{e,s}^2);
\end{equation}
\begin{multline}\label{triplespacedenstrans}
(\beta_{e,s}^3)^*\nu(M\times[0,1]_{\eps})=(\rho_{N_1}\rho_{N_2}\rho_{N_3})(\rho_F\rho_C\rho_S)^2(\rho_T)^{3}(\rho_{FD}\rho_{CD}\rho_{SD})^{h+2}\\
(\rho_{FT}\rho_{CT}\rho_{ST})^{h+3}\rho_{TT}^{2h+3}\nu(X_{e,s}^3).
\end{multline}
\end{proposition}
\begin{proof} Each statement is proved by repeated applications of Proposition C.5 in \cite{ARS1}, which is a proposition originally due to Melrose. 
\end{proof}

\begin{theorem}\label{mappingprops} Let $f\in\mathcal A^{\mathcal F}(X_s;E)$ and let $A\in\Psi^{m,\mathcal E}_{e,s}(X_s;E).$ Then $g=Af\in\mathcal A^{\mathcal G}(X_s;E),$ with
\[G_{\ms}=(E_{\mf}+F_{\ms})\extu(E_{\rf}+F_{\bs}-h);\ \ G_{\bs}=(E_{\lf}+F_{\ms})\extu(E_{\ff}+F_{\bs})\extu(E_{\ebf}+F_{\bs}-h).\]
\end{theorem}
Note that since all our operators act on compact manifolds for positive $\eps$, no condition on the index set is needed for this action to be defined.
\begin{proof}
We prove the theorem when $m=-\infty$; the extension to arbitrary $m$ is standard, as in \cite{ARS1}. The kernel $K(A)$ of $A$ is an element of $\mathcal A(X_s)^2$ with index family $\cE$. We have
\begin{equation}\label{mapping}g=(\pi_{e,L})_*(K(A)\kappa_{e}(\pi_{e,R})^*f),\end{equation}
where $\kappa_{e}$ is a smooth nonvanishing section of $(\pi_{e,R})^*\Omega_{e,s}(X_s)$. As in \cite{ARS1}, for convenience in using the pull-back and push-forward theorems of Melrose, we transform (\ref{mapping}) into an equation where the left-hand side is a multiple of $\nu(X_s)$ and the interior of the push-forward on the right is a multiple of $\nu(X_s^2).$ First multiply each side by $\nu_{e,s}|d\epsilon|,$ where $\nu_{e,s}\in \Omega_{e,s}(X_s)$ is a non-vanishing section, and rewrite
\[g\nu_{e,s}|d\epsilon|=(\pi_{e,L})_*(K(A)(\pi_{e,R})^*f((\pi_{e,L})^*\nu_{e,s}(\pi_{e,R})^*\nu_{e,s}|d\epsilon|)).\]
Now, the left-hand side is $g\rho^{-(h+1)}\beta_s^*\nu(M\times[0,1]_{\epsilon})$, while the right-hand side is
\[(\pi_{e,L})_*((\rho\rho')^{-(h+1)}K(A)(\pi_{e,R})^*f((\pi_{e,L})^*\beta_s^*\nu(M\times[0,1]_{\epsilon})(\pi_{e,R})^*\beta_s^*\nu(M\times[0,1]_{\epsilon})|d\epsilon|^{-1}).\]
Because $\nu(M\times[0,1]_{\epsilon})=|dx\ dy\ dz\ d\epsilon|$,
\[(\pi_{e,L})^*\beta_s^*\nu(M\times[0,1]_{\epsilon})(\pi_{e,R})^*\beta_s^*\nu(M\times[0,1]_{\epsilon})=(\beta_{e,s}^2)^*\nu(M^2\times[0,1]_{\epsilon})|d\epsilon|.\]
Combining this with the definitions of the boundary defining functions, as in \cite{ARS1}, gives
\[g\rho_{\bs}^{-(h+1)}\beta_s^*\nu(M\times[0,1]_{\epsilon})=(\pi_{e,L})_*((\rho_{\lf}\rho_{\rf}\rho^2_{\ebf}
\rho^2_{\ff})^{-(h+1)}K(A)(\pi_{e,R})^*f((\beta_{e,s}^2)^*\nu(M^2\times[0,1]_{\epsilon})).\]
Applying Proposition \ref{prop:densityblowup} yields
\begin{equation}\label{mappingtrans} g\rho_{\bs}^{-h}\nu(X_s)=(\pi_{e,L})_*(K(A)(\pi_{e,R})^*f(\rho_{\lf}\rho_{\rf}\rho_{\ebf}^2\rho_{\ff})^{-h}\nu(X_{e,s}^2)).\end{equation}

Melrose's pullback theorem shows that $K(A)(\pi_{e,R})^*f(\rho_{\lf}\rho_{\rf}\rho_{\ebf}^2\rho_{\ff})^{-h}$ is polyhomogeneous on $X_s^2$ with index family
\[E_{\ff}+F_{\bs}-h\textrm{ at }\bhs{\ff},\ E_{\ebf}+F_{\bs}-2h\textrm{ at }\bhs{\ebf},\ E_{\lf}+E_{\ms}-h\textrm{ at }\bhs{\lf},\]
\[E_{\rf}+F_{\bs}-h\textrm{ at }\bhs{\rf},\ E_{\mf}+F_{\ms}\textrm{ at }\bhs{\mf}.\]
Then the pushforward theorem shows that $g\rho^{-h}_{\bs}$ is polyhomogeneous on $X_s$ with index family 
\[(E_{\ff}+F_{\bs}-h)\extu (E_{\ebf}+F_{\bs}-2h)\extu(E_{\lf}+F_{\ms}-h)\textrm{ at }\bhs{\bs},\]
\[(E_{\mf}+F_{\ms})\extu (E_{\rf}+F_{\bs}-h)\textrm{ at }\bhs{\ms}.\]
Adding $h$ to the index set at $\bhs{\bs}$ completes the proof.  
\end{proof}

We also have a composition formula.
\begin{theorem}\label{composition}[Composition] Let $A\in\Psi^{m,\mathcal E}_{e,s}(X_s;E)$ and $B\in\Psi^{m',\mathcal F}_{e,s}(X_s;E).$ Then  \linebreak $C=A\circ B\in\Psi^{m+m',\mathcal G}_{e,s}(X_s;E),$ where \[G_{\ff}=(E_{\ff}+F_{\ff})\extu(E_{\ebf}+F_{\ebf}-h)\extu(E_{\lf}+F_{\rf});\] 
\[G_{\ebf}=(E_{\ebf}+F_{\ff})\extu(E_{\ff}+F_{\ebf})\extu(E_{\ebf}+F_{\ebf}-h)\extu (E_{\lf}+F_{\rf});\] 
\[G_{\lf}=(E_{\ff}+F_{\lf})\extu(E_{\ebf}+F_{\lf}-h)\extu(E_{\lf}+F_{\mf});\]
\[G_{\rf}=(E_{\rf}+F_{\ff})\extu(E_{\rf}+F_{\ebf}-h)\extu(E_{\mf}+F_{\rf});\]
\[G_{\mf}=(E_{\mf}+F_{\mf})\extu(E_{\rf}+F_{\lf}-h).\]
\end{theorem}
\begin{proof} Assume for the moment that $m=m'=-\infty$. The extension to arbitrary $m$ and $m'$ is standard and follows the arguments in \cite{ARS1}; we therefore omit the details. As in \cite{mame1} and \cite{ARS1}, we have
\begin{equation}K(C)\kappa_{e}=(\pi_{e,C})_*((\pi_{e,F})^*(K(A)\kappa_{e})(\pi_{e,S})^*(K(B)\kappa_{e})),\label{ws.21}\end{equation}
where $\kappa_e= \pi_{e,R}^*\nu_{e,s}$ with $\nu_{e,s}$ a non-vanishing section of $\Omega_{e,s}(X_s)$.  
We follow the same approach as in the previous proof, namely rewrite everything in terms of $\nu(X_{e,s}^i)$, $i=2,3$, and apply the pullback and pushforward theorems.  Denoting coordinates on the second factor with a prime and on the third factor with a double-prime, we see multiplying both sides of \eqref{ws.21} by $(\pi^*_{e,L}\nu_{e,s})|d\eps|$ that 
\begin{multline}\label{thiseq}K(C)(\rho\rho'')^{-(h+1)}(\beta_{e,s}^2)^*\nu(M^2\times[0,1]_{\epsilon})=(\pi_{e,C})_*((\rho\rho'\rho'')^{-(h+1)}(\pi_{e,F})^*(K(A))(\pi_{e,S})^*(K(B))\\ \cdot (\beta_{e,s}^3)^*\nu(M^3\times[0,1]_{\epsilon})).\end{multline}
The left-hand side of (\ref{thiseq}), using Proposition \ref{prop:densityblowup}, is
\[K(C)(\rho_{\lf}\rho_{\rf}\rho_{\ebf}^2\rho_{\ff})^{-h}\nu(X_{e,s}^2).\]
Now let
\[\nu_{RHS}=(\rho\rho'\rho'')^{-(h+1)}(\beta_{e,s}^3)^*\nu(M^3\times[0,1]_{\epsilon}).\]
As in \cite{ARS1}, we use the pullback theorem to compute that 
\begin{multline}\nu_{RHS}=(\rho_{TT}\rho_{FT}\rho_{CT}\rho_{ST}\rho_{T})^{-3(h+1)}(\rho_{FD}\rho_{CD}\rho_{SD}\rho_F\rho_C\rho_S)^{-2(h+1)}\\ \cdot(\rho_{N_1}\rho_{N_2}\rho_{N_3})^{-(h+1)}(\beta^3_{e,s})^*\nu(M^3\times[0,1]_{\epsilon}).\end{multline}
Using Proposition \ref{prop:densityblowup}, we compute that 
\[\nu_{RHS}=(\rho_T)^{-3h}(\rho_{FT}\rho_{CT}\rho_{ST}\rho_F\rho_C\rho_S)^{-2h}(\rho_{TT}\rho_{FD}\rho_{CD}\rho_{SD}\rho_{N_1}\rho_{N_2}\rho_{N_3})^{-h}\nu(X_{e,s}^3).\]

Now use the pullback and pushforward theorems. From the pullback theorem, $$K(C)(\rho_{\lf}\rho_{\rf}\rho_{\ebf}^2\rho_{\ff})^{-h}\nu(X_{e,s}^2)$$ is the pushforward by $\pi_{e,C}$ of a section of $\Omega(X_{e,s}^3)$ with index family
\begin{equation}
\begin{gathered}E_{\ff}+F_{\ff}-h \textrm{ at $\bhs{TT},$ } E_{\ff}+F_{\ebf}-2h\textrm{ at $\bhs{FT},$ } E_{\ebf}+F_{\ebf}-2h\textrm{ at $\bhs{CT},$ }\\
E_{\ebf}+F_{\ff}-2h\textrm{ at $\bhs{ST},$ } E_{\ff}+F_{\lf}-h\textrm{ at $\bhs{FD},$ } E_{\lf}+F_{\rf}-h\textrm{ at $\bhs{CD},$ }\\
E_{\rf}+F_{\ff}-h\textrm{ at $\bhs{SD},$ } E_{\ebf}+F_{\ebf}-3h\textrm{ at $\bhs T,$ } E_{\ebf}+F_{\lf}-2h\textrm{ at $\bhs F,$ }\\
E_{\lf}+F_{\rf}-2h\textrm{ at $\bhs C,$ } E_{\rf}+F_{\ebf}-2h\textrm{ at $\bhs S,$ } E_{\mf}+F_{\rf}-h\textrm{ at $\bhs{N_1},$ }\\
E_{\rf}+F_{\lf}-h\textrm{ at $\bhs{N_2},$ } E_{\lf}+F_{\mf}-h\textrm{ at $\bhs{N_3},$ } E_{\mf}+F_{\mf}\textrm{ at $\bhs Z.$}
\end{gathered}
\label{comp.1}\end{equation}
Applying the pushforward theorem, the full density $K(C)(\rho_{\lf}\rho_{\rf}\rho_{\ebf}^2\rho_{\ff})^{-h}\nu(X_{e,s}^2)$ has index sets
\[\tilde G_{\ff}=(E_{\ff}+F_{\ff}-h)\extu(E_{\ebf}+F_{\ebf}-2h)\extu(E_{\lf}+F_{\rf}-h);\] 
\[\tilde G_{\ebf}=(E_{\ff}+F_{\ebf}-2h)\extu(E_{\ebf}+F_{\ff}-2h)\extu(E_{\ebf}+F_{\ebf}-3h)\]\[\extu (E_{\lf}+F_{\rf}-2h);\] 
\[\tilde G_{\lf}=(E_{\ff}+F_{\lf}-h)\extu(E_{\ebf}+F_{\lf}-2h)\extu(E_{\lf}+F_{\mf}-h);\]
\[\tilde G_{\rf}=(E_{\rf}+F_{\ff}-h)\extu(E_{\rf}+F_{\ebf}-2h)\extu(E_{\mf}+F_{\rf}-h);\]
\[\tilde G_{\mf}=(E_{\mf}+F_{\mf})\extu(E_{\rf}+F_{\lf}-h).\]
Multiplying by $(\rho_{\lf}\rho_{\rf}\rho_{\ebf}^2\rho_{\ff})^{h}$ completes the proof.\end{proof}

In the resolvent construction, we need a couple of composition results for edge surgery operators in the calculus with bounds. These are proved precisely as with the composition formula above, \cf \cite[Section 4]{mame1}.  
\begin{theorem}\label{cwb.1}
If $A, B\in\Psi^{-\infty,\tau,res}_{e,s}(X_s;E)$ and $C,D\in \Psi^{m,\tau}_{e,s}(X_s;E)$  then
\[A\circ B, B\circ A\in\Psi^{-\infty,2\tau,res}_{e,s}(X_s;E);\] 
\[ A\circ C, C\circ A\in\Psi^{-\infty,\tau,res}_{e,s}(X_s;E);\]
\[ C\circ D, D\circ C\in\Psi^{-\infty,\tau}_{e,s}(X_s;E).\]
\end{theorem}
This is proved precisely as with Theorem \ref{composition}, using the techniques in \cf \cite[Section 4]{mame1}. The formulas for the index sets are the same, with the caveat that the outputs are no longer polyhomogeneous, just conormal to (at least) the given order.

\subsection{Normal operators}

Given $A\in \Psi^{m,\cE}_{e,s}(X_s;E)$ with $\inf E_{\ff}\ge 0$ and $\inf E_{\mf}\ge 0$, we can define two normal operators by restricting the Schwartz kernel $K_A$ of $A$ to $\bhs{\mf}$ and $\bhs{\ff}$,
\begin{equation}
 N_{\mf}(A):= \left.  K_A\right|_{\bhs{\mf}}, \quad  N_{\ff}(A):= \left. K_A\right|_{\bhs{\ff}}.
\label{no.1}\end{equation}
Since $\bhs{\mf}$ is naturally identified with the overblown edge double space, $N_{\mf}(A)$ is naturally an edge operator of \cite{maz91}.  In fact, the operation of evaluating at $\bhs{\mf}$ behaves well with respect to composition.

\begin{proposition}
For $A\in \Psi^{m,\cE}_{e,s}(X_s;E)$, $B\in \Psi^{m',\cF}_{e,s}(X_s,E)$ with index families $\cE$ and $\cF$ such that 
$\inf E_{\mf}\ge 0$, $\inf F_{\mf}\ge 0$ and $\inf(E_{\rf}+F_{\lf})>h$, we have that 
\begin{equation}
 N_{\mf}(A\circ B)= N_{\mf}(A)\circ N_{\mf}(B)
\label{no.2b}\end{equation}
where the composition on the right is as edge operators.  
\label{no.2}\end{proposition}
\begin{proof}
The conditions that $\inf E_{\mf}\ge 0$ and $\inf F_{\mf}\ge 0$ are there to ensure that $N_{\mf}(A)$ and $N_{\mf}(B)$ are well-defined.  Similarly, in light of Theorem~\ref{composition}, the condition $\inf(E_{\rf}+E_{\lf})>h$ ensures that $N_{\mf}(A\circ B)$ is well-defined.  At the same time, this condition is necessary to define the composition $N_{\mf}(A)\circ N_{\mf}(B)$ as edge operators.  Now, as is clear from \eqref{comp.1}, under the pushforward by $\pi_{e,C}$, the term $N_{\mf}(A\circ B)$ comes exclusively from the face $\bhs{Z}$.  Since this face  is in fact canonically identified with the triple space of the overblown edge double space, the identity \eqref{no.2b} follows.  

\end{proof}

There is a corresponding result for the restriction at $\bhs{\ff}$.

\begin{proposition}
For $A\in \Psi^{m,\cE}_{e,s}(X_s;E)$ and $B\in \Psi^{m',\cF}_{e,s}(X_s;E)$ with index families $\cE$ and $\cF$ such that 
$\inf E_{\ff}\ge 0$, $\inf F_{\ff}\ge 0$, $\inf(E_{\ebf}+F_{\ebf})>h$ and $\inf(E_{\lf}+F_{\rf})>0$, we have that
\begin{equation}
 N_{\ff}(A\circ B)= N_{\ff}(A)\circ N_{\ff}(B),
\label{no.3b}\end{equation}
where the composition on the right is performed using $\bhs{TT}$ as a triple space.  
\label{no.3}\end{proposition}
\begin{proof}
The conditions that $\inf E_{\ff}\ge 0$ and $\inf F_{\ff}\ge 0$ are there to ensure that $N_{\ff}(A)$ and $N_{\ff}(B)$ are well-defined.  The other conditions ensure that $N_{\ff}(A\circ B)$ is well-defined and that under the pushforward by $\pi_{e,C}$, it comes exclusively from the face $\bhs{TT}$ in \eqref{comp.1}, from which the identity \eqref{no.3b} follows.  
\end{proof}

One important step in the uniform construction of the resolvent will consist in inverting the model operator at $\bhs{\ff}$.  As already explained, $\bhs{\ff}$ is naturally a fiber bundle with fiber $\bhs{\ff,y}$ above $y\in Y$ given by \eqref{se.7a}.  To invert the model operator at $\bhs{\ff}$, it thus suffices to invert it on $\bhs{\ff,y_0}$ for each $y_0\in Y$.  Now, on the interior of $\bhs{\ff,y_0}$, using the coordinates $X=\frac{x}\eps$, $X'=\frac{x'}\eps$, $\check{y}= \frac{y_0-y'}\eps$, we will see that the model operator to invert is of the form
\begin{equation}
   \ang{X} \left(  \widetilde{\Delta}_{\sc}+ \Delta_h\right)\ang{X}, \quad \mbox{with} \quad \widetilde{\Delta}_{\sc}= \ang{X}^{\frac{v+1}2}\Delta_{\sc} \ang{X}^{-\frac{v+1}2},
\label{no.4}\end{equation}
where $\Delta_{\sc}$ is the Hodge Laplacian on $(Z_{y_0}\times \bbR_{X};F)$ associated to the scattering metric 
\begin{equation}
     g_{\sc,y_0}:= dX^2 + \ang{X}^2 g_{Z_{y_0}}
\label{no.5}\end{equation}
with $g_{Z_{y_0}}$ the restriction of $g_{H/Y}$ to $Z_{y_0}$, and where, using the notation of Remark~\ref{nota.1}, $\Delta_h$ is the Hodge Laplacian on $(\eps T_{y_0}Y,F)$ with Euclidean coordinate $\check{y}$ and Euclidean metric $\frac{g_Y}{\eps^2}$.  Here, the power of $\ang{X}^{\frac{v+1}2}$ is there to work directly with $b$-densities and unweighted $b$-Sobolev spaces.  

In terms of the decomposition $\bhs{\ff,y_0}=\overline{\rho T_{y_0}Y}\times Z_{y_0}\times [-\frac{\pi}2,\frac{\pi}2]$ of Remark~\ref{rem.1}, notice that $\ang{X}^{-1}$ is a boundary defining function for $[-\frac{\pi}2,\frac{\pi}2]$.  Hence, recalling the boundary fibration $\rho T_{y_0}Y \times Z_{y_0}\times \pa [-\frac{\pi}2,\frac{\pi}2]\to \rho T_{y_0}Y$ of Remark~\ref{rem.1},we see, keeping in mind that $\rho=\ang{X}\epsilon$ and that $\Delta_h$ is defined on $\epsilon T_{y_0}Y= \frac{\rho}{\ang{X}}T_{y_0}Y$, that \eqref{no.4} is indeed an edge operator in the sense of Remark~\ref{rem.1}.

We can try, at least partially, to invert directly this operators using the edge calculus.  For this, we need to consider  the space of operators
\begin{multline}
  \Psi^{m,\cE}_{\ff,y_0}(W\times \rho T_{y_0}Y;E):=  \dot{\cC}^{\infty}_{\mf\cap\ff}I^m(\bhs{\ff,y_0}, \bhs{\ff,y_0}\cap D_{e,s}; (\pi_{,b,L})^*E\otimes (\pi_{,b,R})^*(E^*\otimes\Omega_{b,s}(X_s)))) \\
  +\sA_{\phg}^{\cE}(\bhs{\ff,y_0}; (\pi_{,b,L})^*E\otimes (\pi_{,b,R})^*(E^*\otimes\Omega_{b,s}(X_s))))\label{no.13}\end{multline}   
for $\cE$ a family index for $\bhs{\ff,y_0}$, where $W$ is the manifold with boundary $Z_{y_0}\times [-\frac{\pi}2,\frac{\pi}2]$. This space is such that the normal operator $N_{\ff}(A)$ of an operator $A\in \Psi^{m,\cE}_{e,s}(X_s;E)$ takes value in such a space with index family naturally induced by the restriction of the  index family $\cE$ on $X^2_{e,s}$ to $\bhs{\ff}$.  Replacing $E_{\ff}$ and $F_{\ff}$ by zero in Theorem~\ref{composition} yields the following composition result. 
\begin{corollary} if $A\in \Psi^{m,\cE}_{\ff,y_0}(W\times \rho T_{y_0}Y;E)$ and $B\in \Psi^{m',\cF}_{\ff,y_0}(W\times \rho T_{y_0}Y;E)$ with index families $\cE$ and $\cF$ such that 
$\inf(E_{\lf}+E_{\rf})>0$ and $\inf (E_{\ebf}+F_{\ebf})>h$, then $A\circ B\in \Psi^{m+m',\cG}_{\ff,y_0}(W\times \rho T_{y_0}Y;E)$ with index family $\cG$ such that
\[G_{\ebf}=E_{\ebf}\extu F_{\ebf}\extu(E_{\ebf}+F_{\ebf}-h)\extu (E_{\lf}+F_{\rf});\] 
\[G_{\lf}=F_{\lf}\extu(E_{\ebf}+F_{\lf}-h)\extu(E_{\lf}+F_{\mf});\]
\[G_{\rf}=E_{\rf}\extu (E_{\rf}+F_{\ebf}-h)\extu(E_{\mf}+F_{\rf});\]
\[G_{\mf}=(E_{\mf}+F_{\mf})\extu (E_{\rf}+F_{\lf}-h).\]
\label{eo.1}\end{corollary} 
Using these properties and the edge calculus, we can partially invert the operator \eqref{no.4} as follows.
\begin{lemma}
There exists $Q\in \Psi^{-2,\cQ}_{\ff,y_0}(W\times \rho T_{y_0}Y;E)$ and $R\in\Psi^{-\infty,\cR}_{\ff,y_0}(W\times \rho T_{y_0}Y;E)$ such that 
$$
      \ang{X} (\widetilde{\Delta}_{\sc} +\Delta_{h})\ang{X} Q= \Id -R,
$$
where the index families $\cQ$ and $\cR$ are such that 
$$
   \inf \cQ_{\mf}\ge 0, \quad \inf \cQ_{\rf}>h, \quad \inf\cQ_{\lf}>0, \quad \inf \cQ_{\ebf}> h;
$$
$$
 \cR_{\mf}=\emptyset, \quad \inf \cR_{\rf}>h, \quad \inf \cR_{\lf}>0, \quad \inf R_{\ebf} > h.
$$

\label{eo.2}\end{lemma}
\begin{proof}
Since the operator $\ang{X}(\widetilde{\Delta}_{\sc}+ \Delta_h)\ang{X}$ is elliptic as an edge operator, we can perform this construction to have an error term of order $-\infty$ vanishing rapidly at all boundary faces except $\bhs{\mf}\cap\bhs{\ff}$.  On the other hand, at this face, by compatibility of $N_{\ff}(\rho D^2_{e,s}\rho)$ and $N_{\mf}(\rho D^2_{e,s}\rho)$ at $\bhs{\mf}\cap\bhs{\ff}$, the restriction of $\ang{X}(\widetilde{\Delta}_{\sc}+\Delta_h)\ang{X}$ to $\bhs{\mf}$ corresponds to the normal operator $N_{y_0}(r\widetilde{\Delta}_{\w}r)$ where $\widetilde{\Delta}_{\w}= r^{\frac{v+1}2}\Delta_{\w}r^{-\frac{v+1}2}$ with $\Delta_{\w}$ the wedge Hodge-Laplacian on $\bhs{\sm}$. By Proposition~\ref{w.1a}, $N_{y_0}(r\widetilde{\Delta}_{\w}r)$ is invertible  on $L^2_{g_b}(\bbR^+_s\times \bbR^h_u\times Z_y;E)$ for the metric $g_b= \frac{ds^2}{s^2}+ du^2+ g_{Z_{y_0}}$.  Hence, we know by \cite[Proposition~5.19]{maz91} that 
$$
   N_y(r\widetilde{\Delta}_{\w}r)^{-1}\in \Psi^{-2,\cH}_e(\bbR^+\times \bbR^h\times Z_{y_0};E)
$$
with index family $\cH$ such that $\cH_{\bff}=\bbN_0$, $\inf\cH_{\lf}>0$ and $\inf\cH_{\rf}>h$ when the Schwartz kernel is written in terms of right edge densities.  Using this inverse, we can construct $Q$ and $R$ as claimed, with the exception that the index set $\cR_{\mf}$ is only such that 
$$
     \cR_{\mf}=\bbN \;  \Longrightarrow \; \inf \cR_{\mf}\ge 1.  
$$
Using Corollary~\ref{eo.1}, the error term $R$ does iterate away at $\bhs{\mf}$, but this is at the cost of possibly deteriorating the properties of the index set of $R$ and $\bhs{\ebf}$.  Instead, using composition of the edge calculus $\Psi^{*,*}_e(\bbR^+\times \bbR^h\times Z_{y_0};E)$ at $\bhs{\mf}\cap\bhs{\ff}$, we can add terms to the parametrix $Q$ to recursively annihilate the terms in the expansion of $R$ at $\bhs{\mf}$ while keeping the index set at $\bhs{\ebf}$ under control.  Hence, taking a Borel sum of these corrections gives the result.  
\end{proof}

Unfortunately, Lemma~\ref{eo.2} does not tell us if the operator \eqref{no.4} is invertible.  Moreover, even if we knew that the inverse existed, because of the composition rules at $\bhs{\ebf}$ in Corollary~\ref{eo.1}, one cannot proceed as in \cite[(4.24) and (4.25)]{maz91} and use the parametrix of Lemma~\ref{eo.2} to get control of the inverse at $\bhs{\ebf}$.       To continue our discussion, we need to  take a different point of view in trying to invert \eqref{no.4}.  Namely, forgetting about the factors of $\ang{X}$ on both sides, what we need to invert is the $\eps T_{y_0}Y$-suspended elliptic scattering operator
\begin{equation}
  \widetilde{\Delta}_{\sc}+ \Delta_h.
\label{no.6}\end{equation}
To invert it, it is therefore natural to first take the Fourier transform in $\eps T_{y_0}Y$.  If $\check\xi \in \eps^{-1}T^*_{y_0}Y$ is the dual variable, then it takes the form
\begin{equation}
    \widetilde{\Delta}_{\sc}+ |\check\xi|^2_{y_0},
 \label{no.7}\end{equation} 
where $|\cdot|_{y_0}$ is the norm induced by the Euclidean metric $\frac{g_Y}{\eps^2}$.  In fact, assuming without loss of generality that we have chosen normal coordinates $y$ at $y_0$, we will assume that this is the usual norm,  $|\check\xi|^2_{y_0}=\check\xi^2$.  Now, for $\check\xi\ne 0$, the operator \eqref{no.7} is fully elliptic and is well-known to be invertible with inverse in the space $\Psi^{-2}_{\sc}(Z_{y_0}\times \overline{\bbR};E_{y_0})$ of scattering operators of order $-2$ for an appropriate vector bundle $E_{y_0}$ above $Z_{y_0}\times \overline{\bbR}$.  However, for $\check\xi=0$, it is not invertible as a scattering operator, in fact it is not even Fredholm.  Nevertheless, it is possible to invert it as a weighted b-operators provided the $b$-operator 
\begin{equation}
   \Delta_{v,y_0}:=  \ang{X} (\ang{X}^{\frac{v+1}2}\Delta_{\sc}\ang{X}^{-\frac{v+1}2})\ang{X},
\label{no.8}\end{equation}
which is the operator in Proposition~\ref{ws.8},  is invertible as a $b$-operator acting on (unweighted) $b$-Sobolev spaces.   
Needless to say, this way of inverting at $\check\xi=0$ is quite different from the way of inverting at $\check\xi\ne 0$.  In particular, the inverse is defined with respect to actions on different Sobolev spaces, namely $b$-Sobolev spaces instead of scattering Sobolev spaces.  Nevertheless, thanks to the work of Guillarmou and Hassell \cite{GH1}, see also \cite{Kottke} for a related work, these two ways of inverting can be pieced together on a suitable manifold with corners.  If we set $k:=|\xi|\in [0,\infty)$, recall that this manifold with corners is constructed as follows.  We start with $[0,\infty)\times W^2$ where $W:=Z_{y_0}\times \overline{\bbR}$, and we blow up the corner $C_3:=\{0\}\times (\pa W)^2$, and then the lifts of codimension 2 corners 
$$
C_{2,L}:=\{0\}\times (\pa W)\times W, \quad C_{2,R}:=\{0\}\times W\times \pa W \quad \mbox{and} \quad C_{2,C}:= [0,\infty)\times \pa W\times \pa W,
$$
yielding the space 
\begin{equation}
  W^2_{k,b}:= [[0,\infty)\times W^2; C_3, C_{2,L}, C_{2,R}, C_{2,C}]
\label{no.9}\end{equation}
with blow-down map $\beta_{b,k}: W^2_{k,b}\to [0,\infty)\times W^2$.  Let us denote by $\zf$, $\lb$ and $\rb$ the lifts to $W^2_{k,b}$ of the boundary hypersurfaces 
$\{0\}\times W^2$, $[0,\infty)\times \pa W\times W$ and $[0,\infty)\times W\times \pa W$ of $[0,\infty)\times W^2$.  Let us also denote by $\bfo$, $\lb_0$, $\rb_0$ and $\bff$ the new boundary hypersurfaces in $W^2_{k,b}$ created by the blow-ups of $C_3, C_{2,L}, C_{2,R}$ and $C_{2,C}$.  Denote also by $D_{k,b}$ the lift of the diagonal 
$$
    \{ (k,w,w)\in [0,\infty)\times W^2 \; | \; w\in W\} \subset [0,\infty)\times W^2
$$
to $W^2_{k,b}$.  The manifolds with corners of Guillarmou-Hassell \cite{GH1} is then obtained by blowing up the $p$-submanifold $D_{k,b}\cap \bff$,
\begin{equation}
  W_{k,\sc}:= [W_{k,b}; D_{k,b}\cap \bff] \quad \mbox{with blow-down map} \; \beta_{k,\sc}: W^2_{k,\sc}\to [0,\infty)\times W^2.
\label{no.10}\end{equation}
Let us denote by $\sc$ the new boundary hypersurface created by this blow-up, and use the same notation as on $W_{k,b}$ to denote the hypersurfaces of $W_{k,\sc}$ coming from the lift of boundary hypersurfaces on $W_{k,b}$.  Let us also denote by $D_{k,\sc}$ the lift of $D_{k,b}$ to $W_{k,\sc}$.  

For $\cE$ an index family of $W^2_{k,\sc}$, we can consider the space of operators
\begin{multline}
\Psi^{m,\cE}_{k}(W;E):= \dot{\cC}^{\infty}_{\sc,\bfo,\zf}I^m(W^2_{k,\sc}, D_{k,\sc}; \pi_{k,\sc,L}^*E \otimes \pi^*_{k,\sc,R}(E^*\otimes \Omega_b(W)) \\
+  \sA_{\phg}^{\cE}(W^2_{k,\sc}; \pi_{k,\sc,L}^*E \otimes \pi^*_{k,\sc,R}(E^*\otimes \Omega_b(W))\label{no.11}\end{multline}
where $\Omega_b(W)$ is the bundle of $b$-densities on $W$, while  $\pi_{k,\sc,L}= \pi_L\circ \beta_{k,\sc}$ and $\pi_{k,\sc,R}= \pi_R\circ \beta_{k,\sc}$ with $\pi_L: [0,\infty)\times W^2\to W$ and $\pi_R: [0,\infty)\times W^2\to W$ the projections on the second and last factors respectively.  With this convention, the composition rules of \cite[(2.15)]{GH1} becomes the following.
\begin{lemma}[\cite{GH1}]
If $A\in \Psi^{m,\cE}_{k}(W;E)$ and $B\in \Psi^{m',\cF}_{k}(W;E)$ with index families $\cE$ and $\cF$ such that there index sets are empty at $\bff$, $\lb$ and $\rb$ and such that $\inf\cE_{\sc}\ge -v-1$ and $\inf\cF_{\sc}\ge -v-1$, then  $A\circ B\in \Psi^{m+m',\cG}_{k}(W;E)$ with
\begin{equation}
\begin{aligned} 
\cG_{\sc}&= \cE_{\sc}+ \cF_{\sc}+v+1; \\
\cG_{\zf}&= (\cE_{\zf}+ \cF_{\zf})\extu (\cE_{\rb_0}+ \cF_{\lb_0}); \\
 \cG_{\bfo}&= (\cE_{\lb_0}+ \cF_{\rb_0})\extu (\cE_{\bfo}+ \cF_{\bfo}); \\
 \cG_{\lb_0}&= (\cE_{\lb_0}+\cF_{\zf}) \extu (\cE_{\bfo}+\cF_{\lb_0}); \\
 \cG_{\rb_0}&= (\cE_{\zf}+ \cF_{\rb_0})\extu (\cE_{\rb_0}+ \cF_{\bfo}); \\
  \cG_{\bff}&=\cG_{\lb}=\cG_{\rb}=\emptyset. 
\end{aligned}
\label{no.11c}\end{equation}
\label{no.11b}\end{lemma}

In terms of this pseudodifferential calculus, the inverse of $(\widetilde{\Delta}_{\sc}+k^2)$ admits the following description.

\begin{theorem}[\cite{GH1} and \cite{Guillarmou-Sher}]
Suppose that the $b$-operator $\Delta_{v,y_0}$ in \eqref{no.8} is invertible as an operator
\begin{equation}
     \Delta_v: H^2_b(W;E)\to L^2_b(W;E).  
\label{no.12c}\end{equation}
Then $(\widetilde{\Delta}_{\sc}+k^2)^{-1}\in \Psi^{-2,\cE}_k(W;E)$ with $\cE$ an index family such that $E_{\bff}=E_{\lb}=E_{\rb}=\emptyset$ and with
\begin{equation}
\inf E_{\zf}\ge 0, \quad \inf E_{\bfo}\ge -2, \quad \inf E_{\lb_0}\ge \nu_0-1, \quad \inf E_{\rb_0}\ge\nu_0-1, \quad \inf E_{\sc}\ge -v-1,
\label{no.12b}\end{equation}
where $\displaystyle \nu_0:= \min_{y\in Y}\min_{\nu} \{ \nu\ge 0  \; | \;  \nu\in \Re\Spec_b(\Delta_{v,y})\}>0$.
\label{no.12}\end{theorem}
\begin{proof}
Notice first that in \cite{GH1}, the result is stated for the scalar Laplacian with a potential, but what is important for their construction to work is that $\widetilde{\Delta}_{\sc}+k^2$ is invertible as a scattering operator for $k>0$ and that \eqref{no.12c} is invertible\footnote{By the discussions near \cite[(3.1)]{GH1} and \cite[(2)]{Guillarmou-Sher}, the invertibility of \eqref{no.12c} corresponds to $\Delta_{\sc}$ having neither zero modes nor a zero-resonance.}.  Indeed, these assumptions imply that the model operators for $\widetilde{\Delta}_{\sc}+k^2$ at $\sc$, $\zf$ and $\bfo$ are invertible.  We can thus first construct an approximate inverse $G(k)\in \Psi^{-2,\cF}_k(W;E)$ for an index set with $F_{\bff}=F_{\lb}=F_{\rb}=\emptyset$ and 
$$
\inf F_{\zf}\ge 0, \quad \inf F_{\bfo}\ge -2, \quad \inf F_{\lb_0}\ge \nu_0-1, \inf F_{\rb_0}\ge \nu_0-1, \quad \inf E_{\sc}\ge -v-1.
$$
Furthermore, $\left. G(k)\right|_{\zf}$ is given by $\ang X\Delta_{v,y_0}^{-1} \ang{X'}$.  This approximate inverse inverts $\widetilde{\Delta}_{\sc}+k^2$ at $\sc$, $\zf$ and $\bfo$, so that using the composition formula of \cite{GH1} and taking into account the different convention used in \eqref{no.11} and \cite{GH1}, we have that
$(\widetilde{\Delta}_{\sc}+k^2)G(k)= \Id +R(k)$ with $R(k)\in \Psi^{-\infty,\cR}_{k}(W;E)$ where $\cR$ is an index family such that $R_{\bff}=R_{\lb}=R_{\rb}=\emptyset$ and 
$$
      \inf R_{\zf}>0, \quad \inf R_{\bfo}>0, \quad \inf R_{\lb_0}\ge \nu_0+1, \quad \inf R_{\rb_0}\ge \nu_0-1, \quad \inf R_{\sc}>-v-1.  
$$
As explained in \cite[p.879-880]{GH1}, it follows that $\Id +R(k)$ is invertible for small $k$ with inverse of the form $\Id+S(k)$ with $S(k)\in \Psi^{-\infty,\cS}_k(W;E)$ and $\cS$ an index family satisfying the same properties as $\cR$.  Hence, using again the composition formula \eqref{no.11c}, for small $k$,
$$
 (\widetilde{\Delta}_{\sc}+k^2)^{-1}= G(k)(\Id+S(k))
$$
has the desired properties.  For large $k$, it also has the desired properties, since $\Delta_{\sc}$ has positive spectrum, and therefore $\widetilde{\Delta}_{\sc}+k^2$ is invertible as a scattering operator for all $k>0$.  
\end{proof}
\begin{remark}
Using Lemma~\ref{eo.2}, We can initially choose the approximate inverse $G(k)$ in the previous proof in such a way that the problem is solved to all orders at the faces $\sc$ and $\bfo$.  In this case, using the composition rules \eqref{no.11c} to iterate away the error term $R$, we see that $(\widetilde{\Delta}_{\sc}+k^2)^{-1}$ and the approximate inverse $G(k)$ have the same expansion at $\sc$, and thus at the corner $\sc\cap \bfo$.   
\label{elr.1}\end{remark}

To obtain a nice description of the inverse of \eqref{no.4}, it suffices then to substitute $k=|\check\xi|$ in $(\widetilde{\Delta}_{\sc}+k^2)^{-1}$, to take the inverse Fourier transform in $\check \xi$ and to compose on the left by multiplication by $\ang{X}^{-2}$.  
\begin{theorem}
The inverse of \eqref{no.4} is an element of $\Psi^{-2,\cE}_{\ff,y_0}(W\times \rho T_{y_0}Y;E)$ with $\cE$ an index  family  such that 
$$
   \inf E_{\mf}\ge 0, \quad   \inf E_{\ebf}\ge h+1, \quad \inf E_{\lf}>0, \quad \inf E_{\rf}>h.  
$$
\label{no.14}\end{theorem}
\begin{proof}
The inverse is given by $\ang{X}^{-1}(\widetilde{\Delta}_{\sc}+ \Delta_h)^{-1}\ang{X}^{-1}$, and by Theorem~\ref{no.12}, 
\begin{equation}
 (\widetilde{\Delta}_{\sc}+\Delta_h)^{-1}= \left(\frac{1}{(2\pi)^h} \int e^{i\check{y}\cdot \check\xi} (\widetilde{\Delta}_{\sc}+ |\check\xi|^2)^{-1} d\check\xi\right) \frac{dy'}{\eps^h}.
\label{no.15}\end{equation}
Let us first compute this integral near $\sc\subset W^2_{k,\sc}$ with $k= |\check\xi|$.  There, suitable coordinates that can be used all the way down to $\bfo$ are 
\begin{equation}
k=|\check{\xi}|, \quad \chi'= \frac{1}{k|X'|}, \quad \check{S}= \frac{\log\left(\frac{|X|}{|X'|}\right)}{\chi'}= k|X'|\log\left(\frac{|X|}{|X'|}\right), \quad \check{Z}= \frac{z-z'}{\chi'}, \quad z'.
\label{no.16}\end{equation}
In terms of these variables, $(\widetilde{\Delta}_{\sc}+k^2)^{-1}$ is given near $\sc$ by  
\begin{equation}
\left( \frac{1}{(2\pi)^{v+1}} \int e^{i\check{S}\cdot \check{\sigma}} e^{i\check{Z}\cdot \check{\zeta}}  \check{q}(\chi', z', k, \check{\sigma},\check{\zeta}) d\check{\sigma} d\check{\zeta}\right) (k|X'|)^{v+1}\frac{d|X'|dz'}{|X'|}
\label{no.17}\end{equation}
with $\check{q}$ such that $k^2\check{q}\in \cS^{-2}(\cU\times [0,\infty); \bbR_{\check{\sigma}}\times \bbR^v_{\check{\sigma}})$ is a smooth symbol of order $-2$ with respect to  the variables $(\check{\sigma},\check{\zeta})$, where $\cU$ is an open set where the coordinates $(\chi',z')$ are valid.  In fact, on $\sc$, that is, when $\chi'=0$, we have explicitly that 
\begin{equation}
   \check{q}(0,z',k,\check{\sigma},\check{\zeta})=\frac{1}{k^2(1+\check{\sigma}^2+ p_2(z',\check{\zeta}))},
\label{no.18}\end{equation}
where $p_2(z',\zeta)= a_{ij}(z')\zeta^i \zeta^j$ is the principal symbol of the Laplacian on $(Z_{y_0}, g_{Z_{y_0}})$.  In other words, this is just the inverse of the symbol of $\widetilde{\Delta}_{\sc}+k^2$, which near $\sc\cap \bfo$ takes the   form 
\begin{equation}
       k^2\left(\check{\sigma}^2+  p_2(z',\check{\zeta}) +1 + \chi'p_1(\chi',z',k,\check{\sigma},\check{\zeta})+ (\chi')^2 p_0(\chi',z',k,\check{\sigma},\check{\zeta})\right),
\label{no.18b}\end{equation}
where $p_0$ and $p_1$ are smooth symbols homogeneous of degree $0$ and $1$ in $(\check{\sigma},\check{\zeta})$. 
This suggests to make the change of variable 
\begin{multline}
  |X'|^{-1}= k\chi',  \quad   S=\frac{\check S}{k|X'|}= \log\left(\frac{|X|}{|X'|} \right), \quad Y=\frac{\check y}{|X'|}=\frac{y_0-y'}{\eps|X'|},  \\ z-z'= \frac{\check Z}{k|X'|},  \quad \xi=\check\xi|X'|, \quad
   \sigma= k|X'|\check\sigma \quad \zeta= k|X'|\check\zeta.
\label{no.18c}\end{multline}
This coordinates are in fact well-adapted to $\bhs{\ff,y_0}$ and can be used to describe the corresponding edge operators invariant by translation in  the $T_{y_0}Y$ factor.  In terms of these coordinates, the symbol \eqref{no.18b} becomes
\begin{equation}
|X'|^{-2}\left({\sigma}^2+ p_2(z',\zeta)  +|\xi|^2 + p_1(\chi',z',k,\sigma,\zeta)+  p_0(\chi',z',k,\sigma,\zeta) \right).\label{no.18d}
\end{equation} 

so that \eqref{no.15} becomes 
\begin{equation}
\left( \frac{1}{(2\pi)^{n}} \int e^{iS\cdot \sigma} e^{i (z-z')\cdot \zeta}e^{iY\cdot\xi} q(|X'|^{-1},z',|\xi|,\sigma,\zeta) d\sigma d\zeta d\xi \right)  \frac{dy'}{(\eps |X'|)^h} \frac{dx'dz'}{x'}
\label{no.19}\end{equation}
where now
\begin{equation}
     q(|X'|^{-1},z',\xi,\sigma,\zeta)= \check{q}\left(\frac{1}{k|X'|},z',\frac{|{\xi}|}{|X'|},\frac{\sigma}{k|X'|},\frac{\zeta}{k|X'|}\right)
\label{no.20}\end{equation}
is such that $|X'|^{-2}q$ is  a  symbol of order $-2$ in $\sigma, \zeta$ and $ \xi$, depending smoothly on $|X'|^{-1}$ and $z'$.  Smoothness in $|X'|^{-1}$ is more delicate and relies on Remark~\ref{elr.1}.  Indeed, thanks to this remark, the symbol $|X'|^{-2}\check{q}$ has the same expansion at $\sc\cap \bfo$  as the one induced by the symbol of the approximate inverse of Lemma~\ref{eo.2}.  In terms of the coordinates \eqref{no.18c}, this means that $|X'|^{-2}q$ has a smooth expansion in $|X'|^{-1}$ as claimed.  

Notice that this discussion takes place  only near $\sc$, that is, for $|\xi|> \delta^{-1}>0$ for some $\delta>0$ small.  Thus, in this region, multiplying \eqref{no.19} by $\ang{X}^{-2}$ gives an operator of the desired form on $\bhs{\ff}$, but with rapid decay at $\bhs{\ebf}$, $\bhs{\lf}$ and $\bhs{\rf}$.  

Elsewhere, but near the lifted diagonal $D_{k,\sc}$, we can proceed essentially in the same way, that is, by taking the Fourier transform of the Schwartz kernel in the direction normal to $D_{k,\sc}$,  to get an operator of the desired form, again in this case decaying rapidly at $\ebf$, $\lf$ and $\rf$.  

With this understood and after multiplying by $\ang{X}^{-1}\ang{X'}^{-1}$, what is left to understand is the contribution near $\zf$ coming from an operator in $\hat{a}(k)\in\Psi^{-\infty,\cF}_{k,\sc}(W;E)$ with
$F_{\bff}=F_{\rb}=F_{\lb}=F_{\sc}=\emptyset$ and 
$$
   \inf F_{\zf}\ge 0, \quad \inf F_{\bfo}\ge 0, \quad \inf F_{\rb_0}>0, \quad \inf F_{\lb_0}>0, 
$$
when we take the inverse Fourier transform
\begin{equation}
a(\check y)= \left(\frac{1}{(2\pi)^h}\int e^{i\check y\cdot \check{\xi}} \hat{a}(|\check{\xi}|)d\check\xi\right) \frac{dy'}{\eps^h}.
\label{no.21}\end{equation}
Now, away from $\lb_0$ and possibly near $\bfo$ and $\rb_0$, we can use the coordinates 
$$
\xi= \ang{X'}\check{\xi}, \quad Y=\frac{\check{y}}{\ang{X'}}, \quad \ang{X}^{-1}, \quad s= \frac{\ang{X}}{\ang{X'}}, \quad z,z', 
$$
so that 
\begin{equation}
     a(Y)= \left(\frac{1}{(2\pi)^h} \int e^{iY\cdot \xi} \hat{a} \; d\xi \right) \frac{dy'}{(\rho')^h}.
\label{no.22}\end{equation}

Since $\hat{a}$ depends smoothly on $|\xi|$ for $|\xi|>0$, does not depend on $\frac{\xi}{|\xi|}$, has a polyhomogeneous expansion as $|\xi|\searrow 0$ and vanishes rapidly when $|\xi|\to\infty$, we see that $a$ will be smooth in $Y$, even at $Y=0$, and independent of $\frac{Y}{|Y|}$.  Moreover, the polyhomogeneous expansion of $\hat{a}$ in $|\xi|$ at $|\xi|=0$ will correspond to a polyhomogenous expansion in $\frac{1}{|Y|}$ as $|Y|\to \infty$ of $|Y|^{h+1} a$, so that in particular $a$ decays like $|Y|^{-h-1}$ as $|Y|\to \infty$ and gives the claimed behavior at $\bhs{\ebf}$.  Indeed, a term of order $\ell$ at $\zf$ will correspond  under the inverse Fourier transform to a term of order $|Y|^{-h-\ell}$ in the expansion as $|Y|\to \infty$.  However, since the term of order zero at $\zf$ is automatically smooth in $\xi$, its inverse Fourier transform will decay rapidly at infinity, so the dominant term in the expansion at $|Y|\to \infty$ decays at least like $|Y|^{-h-1}$ as claimed.   
      Then writing $a$ in terms of an edge surgery density gives the claimed behavior at the faces $\bhs{\mf}$ and $\bhs{\rf}$.  Away from $\rb_0$, but possibly near $\bfo$ or $\lb_0$, we use instead the coordinates 
 $$
\xi= \ang{X}\check{\xi}, \quad Y=\frac{\check{y}}{\ang{X}}, \quad \ang{X'}^{-1}, \quad s= \frac{\ang{X'}}{\ang{X}},  \quad z,z',
$$
and apply a similar argument to see that $a$ gives an operator of the desired type on $\bhs{\ff}$.  
\end{proof}

\section{Uniform construction of the resolvent under a wedge surgery}   \label{ur.0}

In this section, under suitable hypotheses, we will provide a uniform construction of the resolvent of the Hodge Laplacian under a wedge surgery.   Thus, let $H\subset M$ be a two-sided hypersurface in $M$ and $c: H\times(-\delta,\delta)\to X$ a tubular neighborhood,  $\phi:H\to Y$ a fiber bundle with base $Y$ a compact manifold.  Let $g_{\ew}$ be a choice of exact $\ew$-metric with respect these choices of $H$, $\phi$ and $c$.  Let also $F\to M$ be a flat vector bundle equipped with a bundle metric not necessarily compatible with the flat connection. For $E= \Lambda^*({}^{\ew}T^*X_s)\otimes F$, let $\eth_{\ew}\in \Diff^1_{\ew}(X_s;E)$ be the corresponding $\ew$-de Rham operator.  This operator is formally self-adjoint when acting on $\CI_c(X_s\setminus \bhs{\bs};E)\subset L^2_{\ew}(X_s;E).$  To work with $b$-densities, we will consider the related operator 
$$
    D_{\ew}= \rho^{\frac{v+1}2}\eth_{\ew}\rho^{-\frac{v+1}2},
$$
which is formally self-adjoint when acting on $\CI_c(X_s\setminus \bhs{\bs};E)\subset L^2_{b,s}(X_s;E).$

Let $r$ be a boundary defining function for $\bhs{\sm}$ which is equal to $|\rho|$ near $\pa \bhs{\sm}$.  Notice that $\rho^2\eth^2_{\ew}\in \Diff^2_{e,s}(X_s;E)$.  We will  make the following assumption on $\eth_{\ew}$.

\begin{assumption}
In the terms of the decomposition \eqref{dR.1}, we will assume that 
$$
\spec((\eth_{y}^{H/Y})^2)\cap [0,4]=\emptyset \quad \forall  \; y\in Y.
$$
\label{ur.1}\end{assumption}  

With this assumption, we know by Corollary~\ref{wss.6} and Corollary~\ref{wss.7} that $\eth_{\w}^2$ is essentially self-adjoint with unique self-adjoint extension  given by $\cD_{\min}(\eth^2_{\w})=r^2 H^2_{\w}(\bhs{\sm};E)$.  Furthermore, by Lemma~\ref{lem:w.1} and Propostion~\ref{specb.2}, the map
\begin{equation}
  \eth^2_{\w}: r H^2_{\w}(\bhs{\sm}; E_{\sm})\to r^{-1}L^2_{\w}(\bhs{\sm};E_{\sm})
\label{ur.3}\end{equation}
is Fredholm.

On the face $\bhs{\bs}$, we can define another model operator in terms of the fiber bundle \eqref{se.2a}.
\begin{definition}
The \textbf{vertical family} of $D^2_{\ew}$ is the family of $b$-operators \linebreak $\Delta_v\in \Diff^2_b(\bhs{\bs}/Y;E)$ acting fiberwise on the fiber bundle \eqref{se.2a} obtained by restricting the action of $\rho D^2_{\ew}\rho$ to the boundary face $\bhs{\bs}$.
\label{ur.5}\end{definition}
By Lemma~\ref{ws.6} and Proposition~\ref{ws.8}, we see that Assumption~\ref{ur.1} implies that for each $y\in Y$, the restriction $\Delta_{v,y}$ of $\Delta_v$ to the fiber $ [-\frac{\pi}2,\frac{\pi}2]\times Z_y:= (\phi_+)^{-1}(y)$ induces an isomorphism
\begin{equation}
        \Delta_{v,y}: H^2_{b}([-\frac{\pi}2,\frac{\pi}2]\times Z_y;E)\to L^2_b([-\frac{\pi}2,\frac{\pi}2]\times Z_y;E).
\label{inv.1}\end{equation}

We are now ready to state the main theorem of this section.

\begin{theorem}  Let $\eth_{\ew}$ be the wedge surgery de Rham operator associated to a choice of exact $\ew$-metric and a choice of flat vector bundle $F\to M$ with bundle metric.  Suppose that $\eth_{\ew}$ satisfy Assumption~\ref{ur.1}. 
 Then for any bounded open set $V\subset \bbC$ such that 
$$
     \overline{V}\cap \spec(N_{\mf}(\eth_{\ew}))\subset \{0\},
$$
there exists $\tau>0$ and a holomorphic family $V\ni \lambda\mapsto f(\cdot,\lambda) \in\CI([0,1]_{\epsilon})+\cA^\tau([0,1])$  such that the resolvent $(D_{\ew}^2-\lambda)^{-1}$ extends from $V\cap(\bbC\setminus\bbR)$ to a meromorphic family 
$$
        (D_{\ew}^2-\lambda)^{-1}= \Res_H(\lambda)+ \Res_M(\lambda)
$$
on $V$, with only simple poles,  where $\lambda\to \rho^{-1}\Res_H(\lambda)\rho^{-1}\in  \Psi^{-2,\tau}_{\ee}(X_s;E)$  is a holomorphic family, while $\Res_M(\lambda)$ is a meromorphic family uniformly of finite rank such that $$
\lambda\to f(\epsilon,\lambda)\rho^{-1}\Res_M(\lambda)\rho^{-1}\in \Psi^{-\infty,\tau}_{\ee}(X_s;E)
$$ 
is a holomorphic family uniformly  and  such that 
$$
N_{\mf}(\Res_M(\lambda))= -\frac{\Pi_{\ker_{L^2} N_{\mf}(D^2_{\ew})}}{\lambda},  \quad N_{\ff}(\rho^{-1}\Res_M(\lambda)\rho^{-1})=0.
$$  
\label{ur.8}\end{theorem}

Before proving this theorem, notice that it allows us to define the small eigenvalues of the family $D_{\ew}$ as the eigenvalues that approach zero as $\epsilon$ tends to zero.  In fact, taking a contour integral $\Gamma$  going anti-clockwise around the origin  and sufficiently small so that its interior contains no element of the spectrum $D_{\w}$ beside zero, we can for $\epsilon\ge 0$ sufficiently small define the projection onto the eigenspace associated to small eigenvalues by 
\begin{equation}
     \Pi_{\sma}=  \frac{i}{2\pi} \int_{\Gamma} (D_{\ew}^2-\lambda)^{-1}d\lambda= \frac{i}{2\pi}\int_{\Gamma}\Res_M(\lambda)d\lambda.
\label{small.1}\end{equation}
Since the family  $\Gamma\ni\lambda\mapsto \rho^{-1}\Res_M(\lambda)\rho^{-1}\in \Psi^{-\infty,\tau}_{\ee}(X_s;E)$ is smooth, we automatically obtain the following.
\begin{corollary}
The projection $\Pi_{\sma}$ is an element of $ \rho(\Psi^{-\infty,\tau}_{\ee}(X_s;E))\rho$.
\label{small.2}\end{corollary}

Coming back to Theorem~\ref{ur.8}, its proof will involve few steps.

\subsection*{Step 0: Symbolic inversion}

\begin{proposition}
There exist holomorphic families $\bbC\ni \lambda\mapsto Q_0(\lambda)\in \Psi^{-2}_{e,s}(X_s;E)$ and  $\bbC\ni \lambda\mapsto R_0(\lambda)\in \Psi^{-\infty}_{e,s}(X_s;E)$ such that 
$$
   \rho(D^2_{\ew}-\lambda)\rho Q_0(\lambda)= \Id -R_0(\lambda).
$$
\label{ur.9}\end{proposition}
\begin{proof}
Since $\rho (D^2_{\ew}-\lambda)\rho$ is an elliptic edge surgery operator, there exists $Q_0'\in \Psi^{-2}_{e,s}(X_s;E)$ with principal symbol ${}^{e,s}\sigma_{-2}(Q_0')= ({}^{e,s}\sigma_2(\rho D^2_{\ew}\rho-\lambda\rho^2))^{-1}= ({}^{e,s}\sigma_2(\rho D^2_{\ew}\rho))^{-1}$ such that 
$$
       \rho (D^2_{\ew}-\lambda)\rho Q_0'= \Id-R_0'(\lambda),
$$
where $R_0'(\lambda)\in \Psi^{-1}_{e,s}(X_s;E)$ is a holomorphic in $\lambda$.  Adding $Q_0''(\lambda):= Q_0'R_0'(\lambda)$, we obtain
$$
\rho (D^2_{\ew}-\lambda)\rho(Q_0'(\lambda)+ Q_0''(\lambda))= \Id-R_0''(\lambda)
$$  
with $R_0''(\lambda)\in \Psi^{-2}_{e,s}(X_s;E)$ holomorphic in $\lambda$.  Proceeding by induction, we find more generally holomorphic families $Q_0^{(k)}(\lambda):= Q_0'R_0^{(k-1)}(\lambda)\in \Psi^{-k}_{e,s}(X_s;E)$ and $R_0^{(k)}\in \Psi^{-k}_{e,s}(X_s;E)$ such that 
$$
(D^2_{\ew}-\lambda) \rho^2\left(\sum_{j=1}^k Q_0^{(j)}(\lambda)\right)= \Id-R_0^{(k)}(\lambda).
$$
Taking an asymptotic sum over the $Q_0^{(k)}(\lambda)$ gives the desired $Q_0(\lambda)$.  This can be done in such a way that $Q_0(\lambda)$ is holomorphic in $\lambda$.  
\end{proof}

\subsection*{Step 1: Removing the error term at $\bhs{\ff}$}

Since the operator \eqref{inv.1} is invertible, we know by Theorem~\ref{no.14} that $N_{\ff}(\rho(D^2_{\ew}-\lambda)\rho)= N_{\ff}(\rho D^2_{\ew}\rho)$ is invertible, which can be used to improve the error term.  
\begin{proposition}
There exists $\tau>0$ and holomorphic families $\bbC\ni \lambda\mapsto Q_1(\lambda) \in \Psi^{-2,\tau}_{\ee}(X_s;E)$,  $\bbC\ni \lambda\mapsto R_1(\lambda) \in \Psi^{-\infty,\tau}_{\ee}(X_s;E)$ such that 
\begin{equation}
                      \rho (D_{\ew}^2-\lambda)\rho Q_1(\lambda)= \Id-R_1(\lambda) \quad \mbox{and}  \quad N_{\ff}(R_1(\lambda))=0.
\label{ur.11}\end{equation}
\label{ur.10}\end{proposition}
\begin{proof}
By Theorem~\ref{no.14} and Corollary~\ref{eo.1} applied for each $y_0\in Y$, we see that there exists $\tau>0$ and $Q_1'(\lambda)\in \Psi^{-\infty,\tau}_{\ee}(X_s;E)$ such that 
$$
N_{\ff}(Q_1'(\lambda))= N_{\ff}(\rho D_{\ew}^2\rho)^{-1}N_{\ff}(R_0(\lambda)).
$$
Hence, using  Theorem~\ref{cwb.1} and Proposition~\ref{no.3}, we see that it suffices to take
$$
Q_1(\lambda)=Q_0(\lambda)+Q_1'(\lambda).
$$ 
\end{proof}

\subsection*{Step 2: Removing the error term at $\bhs{\mf}$}

Restricting \eqref{ur.11} to $\bhs{\mf}$ gives us a right parametrix $N_{\mf}(Q_1(\lambda))$ for $r(D_{\w}^2-\lambda)r$,
\begin{equation}
    r(D_{\w}^2-\lambda)r N_{\mf}( Q_1(\lambda))= \Id - N_{\mf}(R_1(\lambda)).
\label{ur.12}\end{equation}
On the other hand, since the operator $r(D_{\w}^2-\lambda)r$ is formally self-adjoint with respect to the $L^2_b$-inner product, we obtain a left parametrix $N_{\mf}(Q_1^*(\lambda))$ by taking the adjoint of \eqref{ur.12}, 
\begin{equation}
  N_{\mf}(Q_1^*(\lambda))r (D_{\w}^2-\lambda)r= \Id - N_{\mf}(R^*_1(\lambda)).
  \label{ur.13}\end{equation}
On the other hand, we know by Lemma~\ref{lem:w.1} that the operator 
$$
r(D_{\w}^2-\lambda )r: r^{-\frac{h}2}H^2_{e}(\bhs{\sm};E)\to r^{-\frac{h}2}L^2_e(\bhs{\sm};E)= L^2_b(\bhs{\sm};E)
$$
is Fredholm.  For $\lambda\in V\setminus \{0\}$, it is in fact invertible.  Moreover, since $r^2D^2_{\w}$ has no indicial root in $[0,1]$ for all $y\in Y$ by Assumption~\ref{ur.1} and Proposition~\ref{specb.2}, we know by  Lemma~\ref{lem:w.1} and \cite[Theorem~6.1]{maz91} that $r^2D^2_{\w}$ and $r D^2_{\w}r$ have isomorphic $L^2_b$-kernel with isomorphism given by multiplication by $r^{-1}$.   Hence, we can define a projection on the $L^2_b$-kernel of $rD^2_{\w}r$ by 
$$
\Pi_{\ker rD^2_{\w}r}:=  r \circ \left(\Pi_{\ker D^2_{\w}}\right) \circ r^{-1}
$$
with $\Pi_{\ker D_{\w}^2}$ the orthogonal projection onto the $L_b^2$-kernel of $D^2_{\w}$.  Hence, for $\lambda\in V$, taking 
$$
G(\lambda):= (\Id-\Pi_{\ker rD^2_{\w}r})r^{-1}((D^2_{\w}-\lambda))^{-1}r^{-1} (\Id-\Pi_{\ker rD^2_{\w}r})
$$
gives us a holomorphic family of operators $G(\lambda): L^2_b(\bhs{\sm};E) \to r^{\frac{h}2}H^2_{e}(\bhs{\sm};E)$    such that   
\begin{equation}
 r(D_{\w}^2-\lambda)r G(\lambda)=  \Id-\Pi_{\ker rD^2_{\w}r} \quad \mbox{and} \quad G(\lambda)r(D^2_{\w}-\lambda)r= \Id+ \Pi_{\ker rD^2_{\w}r} ,
\label{ur.14}\end{equation}
Proceeding as in \cite[(4.24), (4,25)]{maz91}, we deduce from  \eqref{ur.12}, \eqref{ur.13} and \eqref{ur.14}, that 
\begin{multline}
  G(\lambda)= N_{\mf}(Q_1(\lambda))+ N_{\mf}(Q_1^*(\lambda)R_1(\lambda))+ N_{\mf}(R_1^{*}(\lambda))G(\lambda)N_{\mf}(R_1(\lambda)) \\
  -N_{\mf}(Q_1^*(\lambda))  \Pi_{\ker rD^2_{\w}r} N_{\mf}(R_1(\lambda))-\Pi_{\ker rD^2_{\w}r}N_{\mf}(Q_1(\lambda)).
  \end{multline}
  Since $N_{\mf}(R_1(\lambda))$ and $N_{\mf}(R_1^*(\lambda))$  are very residual in the sense \cite[p.20]{maz91}, we thus see that $G(\lambda)\in \Psi^{-2,\tau}_{e}(\bhs{\sm};E)$.   This leads to the following.
  
\begin{proposition}
There is a holomorphic family $V\ni \lambda \mapsto Q_2(\lambda)\in \Psi^{-2,\tau}_{\ee}(X_s;E)$ such that 
\begin{equation}
      \rho(D^2_{\ew}-\lambda)\rho Q_2(\lambda)= \Id-R_2(\lambda)
\label{ur.16}\end{equation}
with $V\ni\lambda\mapsto R_2(\lambda)\in \Psi^{-\infty,\tau}_{\ee}(X_s;E)$ a holomorphic family such that $N_{\mf}(R_2(\lambda))=\Pi_{r\ker D^2_{\w}r}$ and  $N_{\ff}(R_2)=0$.  
\label{ur.15}\end{proposition}
\begin{proof}
Since $\Pi_{rD^2_{\w}r}rD^2_{\w}r= r\Pi_{D^2_{\w}}D^2_{\w}r=0$, we see from \eqref{ur.11} restricted to $\bhs{\mf}$ that 
\begin{equation*}
\begin{aligned}
    \Id-N_{\mf}(R_1(\lambda))&= (\Id-\Pi_{\ker rD^2_{\w}r})(\Id-N_{\mf}(R_1(\lambda)))-\lambda \Pi_{\ker rD^2_{\w}r}N_{\mf}(\rho^2 Q_{1}(\lambda))\\
    &=(\Id-\Pi_{\ker rD^2_{\w}r}) -(\Id-\Pi_{\ker rD^2_{\w}r})N_{\mf}(R_1(\lambda))-\lambda \Pi_{\ker rD^2_{\w}r}N_{\mf}(\rho^2 Q_{1}(\lambda)).   
\end{aligned}
\end{equation*}
Thus, since 
$$
      (r(D^2_{\w}-\lambda)r)^{-1}= G(\lambda) -r^{-1}\circ\left(\frac{\Pi_{\ker D^2_{\w}}}{\lambda}\right)\circ r^{-1},
$$
it suffices to take $Q_2(\lambda)=Q_1(\lambda)+Q_2'(\lambda)$ with $V\ni\lambda\mapsto Q_2'(\lambda)\in \Psi^{-\infty,\tau}_{\ee}(X_s;E)$ a choice of holomorphic family such that 
\begin{equation}
\begin{aligned}
    N_{\mf}(Q_2'(\lambda)) &:= -G(\lambda)N_{\mf}(R_1(\lambda)) -((r(D_{\w}-\lambda)r)^{-1}\lambda\Pi_{\ker rD^2_{\w}r}N_{\mf}(\rho^2 Q_1(\lambda)), \\
     &=-G(\lambda)N_{\mf}(R_1(\lambda)) -r^{-1}\Pi_{\ker D^2_{\w}} r^{-1}\Pi_{\ker rD^2_{\w}r}N_{\mf}(\rho^2 Q_1(\lambda)),  \\
     &=-G(\lambda)N_{\mf}(R_1(\lambda)) -r^{-2}\Pi_{\ker rD^2_{\w}r}N_{\mf}(\rho^2 Q_1(\lambda)), 
\end{aligned}    
\end{equation}
\end{proof}

Now, let $\phi_1^0,\ldots,\phi_N^0$ be an orthonormal basis of the $L^2_b$-kernel of $D^2_{\w}$, so that 
$$
      \Pi_{\ker D^2_{\w}}= \sum_{j=1}^N \phi_j^0\cdot \overline{\phi_j^0} \left. \frac{d g_{\ew}}{\rho^{v+1}}\right|_{\bhs{\sm}}.
$$ 
By Assumption~\ref{ur.1} and Proposition~\ref{specb.2}, we know that $r^2 D^2_{\w}$ has no indicial root in $[0,2]$ for all 
$y\in Y$.  Since $r^2D_{\w}^2 \phi_j=0$, this means by Lemma~\ref{lem:w.1} \cite[Theorem~{6.1}]{maz91} that $\phi_j\in r^2\cA^{\tau}(\bhs{\sm};E)$ for some $\tau>0$.  
Thus, we can extend them to sections 
$$
\phi_1,\ldots,\phi_N \in  \sB^{\cE/\df w}_{\phg}\sA^{ 0}_-(X_s;E)$$
such that $\left.  \phi_j\right|_{\bhs{\sm}}=\phi_j^0$, where $\cE_{\sm}=\bbN_0$, $\cE_{\bs}=\emptyset$ and $\df w_{\sm}=\infty$, $\df w_{\bs}=2+\tau$ for some $\tau>0$.  In particular, $(D^2_{\ew})\phi_j\in \sA^{\tau}_{-}(X_s;E)$ for some $\tau>0$.  

For later purposes, we need also to require, after taking $\tau>0$ smaller if needed, that $\epsilon^{-\tau}D^2_{\ew}\phi_j$ are linearly independent for $\epsilon\ge 0$ small.  This can be arranged for instance by picking $\eta_1,\ldots \eta_N\in \CI_c(X_s\setminus \bhs{\bs};E)$ with disjoint supports such that $ D^2_{\ew} \eta_1,\ldots  D^2_{\ew} \eta_N$ are linearly independent and by replacing $\phi_1,\ldots, \phi_N$ by $\phi_1+\epsilon^{\tau}\eta_1,\ldots, \phi_N+\epsilon^{\tau}\eta_N$ for $\tau>0$ sufficiently small. 

With this choice of extensions, we can then extend the projection $\Pi_{\ker D^2_{\w}}$ by 
$$
         \Pi= \sum_{j=1}^N \phi_j\cdot \overline{\phi_j}  \frac{d g_{\ew}}{\rho^{v+1}}  \in \Psi^{-\infty,\tau}_{\ee}(X_s;E).
$$
Since   $\Pi_{r D^2_{\w}r}= r\left(\Pi_{D^2_{\w}}\right) r^{-1}$,  we can then rewrite equation \eqref{ur.16} as
\begin{equation}
      \rho(D^2_{\ew}-\lambda)\rho Q_2(\lambda)=\Id-\rho\Pi \rho^{-1}-S_2(\lambda) \quad \mbox{with} \; S_2(\lambda)= R_2(\lambda)-\rho\Pi\rho^{-1}\in \Psi^{-\infty,\tau,res}_{\ee}(X_s;E).
\label{ur.17}\end{equation}
\begin{proposition}
There exists a holomorphic family of bounded operators 
$$S_3(\lambda)\in \Psi^{-\infty,\tau,res}_{\ee}(X_s;E)  \quad \mbox{such that} \quad  
        (\Id-S_2(\lambda))^{-1}= \Id-S_{3}(\lambda).
$$
\end{proposition}
\begin{proof} This proof is very similar to \cite[Lemma 4.21]{ARS1}. First we take the formal sum of the Neumann series
\[\widetilde S_3(\lambda)=\sum_{j=1}^{\infty}(S_2(\lambda))^j.\]
By Theorem \ref{cwb.1}, $(S_2(\lambda))^j\in\Psi^{-\infty,j\tau,res}_{\ee}(X_s;E)$ for each $j\in\mathbb N$. 
The fact that the kernel of $S_2(\lambda)$ satisfies a universal $L^{\infty}$ bound of $C\eps^{\tau}$ means the series is summable. Since the b-derivatives satisfy bounds as well, they may be taken term by term, and we conclude that $\widetilde S_3(\lambda)\in\Psi^{-\infty,\tau,res}_{\ee}(X_s;E)$, with
\[(\Id-S_2(\lambda))^{-1}(\Id+\widetilde S_3(\lambda))=\Id+T(\lambda),\]
with $T(\lambda)\in\Psi^{-\infty,\infty,res}_{\ee}(X_s;E)$, which is just the space of conormal distributions vanishing to infinite order at all boundary hypersurfaces of $X_s$. 

Now proceed exactly as in \cite[Lemma~{4.21}]{ARS1}.  By compactness, there exists $\epsilon_0$ such that operator $T(\lambda)$ then has norm bounded by 1/2 for all $\lambda$ and all $\eps<\eps_0$. Thus $\Id+T(\lambda)$ is itself invertible and has inverse of the form $\Id+T_1(\lambda)$, with $T_1(\lambda)\in\Psi^{-\infty,\infty,res}_{\ee}(X_s;E)$ as well. Setting
\[S_3(\lambda)=-\widetilde S_3(\lambda)-T_1(\lambda)-\widetilde S_3(\lambda)T_1(\lambda),\]
and using Theorem \ref{cwb.1} again completes the proof.
\end{proof}

Now we set
$$
     Q_3(\lambda)=  Q_2(\lambda)(\Id-S_3(\lambda))-\frac{\rho^{-1}\circ \Pi\circ \rho^{-1}}{\lambda},
$$
and observe that by Theorem \ref{cwb.1}, $Q_3(\lambda)\in\Psi^{-2,\tau}_{\ee}(X_s;E)$. Note that $Q_3(\lambda)$ is a meromorphic family with possibly a simple pole at $\lambda=0$. Furthermore,
we obtain that
\begin{equation}
    \rho (D_{\ew}^2-\lambda)\rho Q_3(\lambda)= \Id- R_3(\lambda) \quad \mbox{with} \; R_3(\lambda)= -\rho\Pi \rho^{-1} S_3(\lambda)+ \frac{\rho D^2_{\ew} \Pi\rho^{-1}}{\lambda},
\label{ur.18}\end{equation}
where by definition, $R_3(\lambda)\in \Psi^{-\infty,\tau,res}_{\ee}(X_s;E)$ and is a meromorphic family with possibly a simple pole at $\lambda=0$.

\subsection*{Step 3: Analytic Fredholm theory}

In this last step, we remove completely the error term using analytic Fredholm theory.  First, notice that the sections 
$$
\phi_1, \quad
\epsilon^{-\tau}D^2_{\ew}\phi_1, \quad 
\ldots,\quad
\phi_N, \quad 
\epsilon^{-\tau}D^2_{\ew}\phi_N,
$$
 are linearly independent for $\epsilon> 0$ sufficiently small.  Let $\Pi_1\in \Psi^{-\infty,\tau}_{\ee}(X_s;E)$ be the projection onto the range of these sections.  In terms of the decomposition
$$
   L^2_b(X_s;E)= \ran (\Id-\rho\Pi_1\rho^{-1})\oplus \ran(\rho \Pi_1\rho^{-1}),
$$
we have that 
$$
\Id-R_3(\lambda)= \left( \begin{array}{cc} \Id & 0 \\ C(\lambda) & D(\lambda)  \end{array} \right)
$$
with $C(\lambda)= \rho\Pi_1\rho^{-1}(\Id-R_3(\lambda))(\Id-\rho\Pi_1\rho^{-1})$ and $D(\lambda)=\rho\Pi_1\rho^{-1}(\Id-R_3(\lambda))\rho\Pi_1\rho^{-1}$.  From equation \eqref{ur.18}, $C(\lambda)$ is holomorphic while $D(\lambda)$ is meromorphic with possibly only a simple pole at $\lambda=0$.  Since the Fredholm determinant of $\Id-R_3(\lambda)$ is clearly equal to the determinant of $D(\lambda)$, we see that $(\Id-R_3(\lambda))$ is invertible if and only if $D(\lambda)$ is, and in this case,
\begin{equation}
 (\Id-R_3(\lambda))^{-1}= \left( \begin{array}{cc} \Id & 0 \\ -D(\lambda)^{-1}C(\lambda) & D(\lambda)^{-1}  \end{array} \right).
 \label{ur.19}\end{equation}  
Now, since $R_3(\lambda)\in \Psi^{-\infty,\tau,res}_{\ee}(X_s;E)$, we know  that for a fixed $\lambda_0\in V\setminus \{0\}$, we can find $\epsilon_0>0$ such that $R_{3}(\lambda_0)$ has a norm smaller than $1/2$ as an operator acting on $L^2_b(X_s;E)$, so that  $\Id-R_{3}(\lambda_0)$ will be invertible for all $\eps\in [0,\epsilon_0]$.  By analytic Fredholm theory, this means that for each $\epsilon\in [0,\epsilon_0]$, $\Id-R_3(\lambda)$ and $D(\lambda)$ are invertible in $\lambda\in V$ except for a finite number of points depending on $\epsilon$.  Since in the decomposition 
$$
\ran(\rho\Pi_1\rho^{-1})= \ran(\rho\Pi\rho^{-1})\oplus \rho\spann\langle\epsilon^{-\tau}D_{\ew}\phi_1,\ldots,\epsilon^{-\tau}D_{\ew}\phi_1\rangle,
$$ 
$$
  D(\lambda)=  \left( \begin{array}{cc} E(\lambda) & F(\lambda) \\ \frac{G(\lambda)}{\lambda} & H(\lambda)  \end{array} \right)
$$
with $E,F,G,H$ holomorphic in $\lambda$ and $G\ne 0$ when $\lambda=0$, we see that 
$$
  \det D(\lambda)= \frac{f(\epsilon,\lambda)}{\lambda^N}  \; \Longrightarrow \det D(\lambda)^{-1}= \frac{\lambda^N}{f(\epsilon,\lambda)}
$$
where $f(\epsilon,\lambda)= \lambda^N\det D(\lambda)$ is holomorphic in $\lambda$ and $f(\epsilon,\lambda)\in \CI([0,\epsilon_0])+\epsilon^{\tau}\cA([0,\epsilon_0])$.  Thus, using the formula for $D(\lambda)$ in terms of the cofactor matrix, we see that 
$$
f(\epsilon,\lambda) D(\lambda)^{-1}\in \Psi^{-\infty,\tau}_{\ee}(X_s;E).
$$ 
From \eqref{ur.19}, $(\Id-R_3(\lambda))^{-1}= \Id-R_{4}(\lambda)$ with
$$
     f(\epsilon,\lambda)R_4(\lambda)\in \Psi^{-\infty,\tau,res}_{\ee}(X_s;E)
$$ 
a holomorphic family  in $\lambda$ for $\lambda\in V$.  Thus, composing on the right by $(\Id-R_3(\lambda))^{-1}$ in \eqref{ur.18}, we finally obtain that 
\begin{equation}
      \rho(D_{\ew}^2-\lambda)\rho Q_4(\lambda)= \Id \quad \mbox{with} \;Q_4(\lambda)= Q_{H}(\lambda)+ Q_{M}(\lambda),
\label{fe.1}\end{equation}
where $Q_H(\lambda)\in \Psi^{-2,\tau}_{\ee}(X_s;E)$ is holomorphic in $\lambda\in V$ and $f(\epsilon,\lambda)Q_{M}(\lambda)\in \Psi^{-\infty,\tau}_{\ee}(X_s;E)$ is holomorphic in $\lambda$ with 
$$
     N_{\mf}( Q_M(\lambda))= \rho^{-1}\frac{\Pi_{\ker D^2_{\w}}}{\lambda}\rho^{-1} \quad \mbox{and} \; N_{\ff}(Q_{M}(\lambda))=0.
$$
Conjugating equation \eqref{fe.1} by $\rho$ finally gives 
$$
(D_{\ew}^2-\lambda)\rho Q_4(\lambda)\rho= \Id, 
$$
giving a proof of Theorem~\ref{ur.8} for $\epsilon\in [0,\epsilon_0]$.  For $\epsilon\ge \epsilon_0$, we have a family of resolvent of elliptic operators on a compact manifold, so the result is standard.

\section{Wedge surgery heat calculus}

\subsection{Wedge surgery heat space}  \label{wshs.1}

Recall that $M$ is an oriented closed manifold with a co-oriented hypersurface $H \subseteq M.$ In this subsection, we will construct a space that carries the heat kernel of the wedge surgery Hodge Laplacian as a well-behaved distribution.\\

We begin with $M\times M\times[0,1)_{\eps}\times \bbR^+_{\tau}$, where $\tau=\sqrt t$. This space has two boundary hypersurfaces: $\{\eps=0\}$, which we denote $\bhs{hmf}$, and $\tau=0$, which we denote $\bhs{tb}$. Then we perform the following series of blow-ups.\\

\textbf{Step 1:} Blow up the interior lifts of first $H\times H\times\{\eps=0\}\times \bbR^+_{\tau}$ and then of \linebreak $H\times M\times\{\eps=0\}\times \bbR^+_{\tau}$ and $M\times H\times\{\eps=0\}\times \bbR^+_{\tau}$,  which yields a space canonically identified with $X^2_{b,s}\times \bbR^+_{\tau}$.   Call the faces created $\bhs{hbf}$, $\bhs{hlf}$, and $\bhs{hrf}$ respectively.\\

\textbf{Step 2:} The face $\bhs{hbf}$ is naturally identified with $\bhs{\bff}\times \bbR^+_{\tau}$,  where $\bhs{\bff}$ is the corresponding face in $X^2_{b,s}$.   Recall moreover that $\bhs{\bff}$ comes with a natural fibration $\phi_b:\bhs{\bff}\to Y\times Y$.  If $D_Y\subset Y\times Y$ denotes the diagonal, we blow up $\phi_b^{-1}(D_Y)\times \{0\}\subset \bhs{\bff}\times \bbR^+_{\tau}=\bhs{hbf}$  to obtain the space 
\begin{equation}
	[X_{b,s}^2\times \bbR^+_{\tau}; \phi_b^{-1}(D_Y) \times\{\tau=0\}],
\label{wsh.1}\end{equation}
which creates a new front face $\bhs{hff}$.  Convenient coordinates valid near $\bhs{hff}\cap\bhs{tb},$ away from other boundary hypersurfaces, include
\begin{equation}\label{coords:interiorhff}
	\lrpar{X:=\frac x\eps, \quad
	y, \quad
	z, \quad
	X':=\frac{x'}{\eps}, \quad
	u:=\frac{y'-y}{\eps}, \quad
	z', \quad
	\sigma:=\frac{\tau}{\eps}, \quad
	\eps},
\end{equation}
where $\eps$ is a boundary defining function for $\bhs{hff}.$
The most interesting portion of this space is a neighborhood of the triple junction $\bhs{hmf}\cap\bhs{hff}\cap\bhs{tb}$, in which, assuming $x'/x$ and $(y'-y)/x$ are bounded away from infinity, we have the coordinates
\begin{equation}\label{coords:step2}
	\lrpar{x, \quad
	\frac{x'}{x}, \quad
	\frac{\tau}{x}, \quad
	\frac{\eps}{x'}, \quad
	y, \quad
	\frac{y'-y}{x}, \quad
	z, \quad
	z'}.
\end{equation}
In these coordinates $x$ is a bdf for $\bhs{hff}$, $\tau/x$ is a bdf for $\bhs{tb}$, and $\eps/(x')$ is a bdf for $\bhs{hmf}$. Note that if $x'/x$ gets large we can make appropriate modifications to these coordinates. 

Now, the face $\bhs{hmf}$ is isomorphic to a version of the heat space in \cite{Mazzeo-Vertman}, where we have not yet blown up the diagonal at $\sqrt t=0$ but have additionally blown up $x=x'=0$ away from $y=y'$ and $\tau =0.$ This is an overblown version of the Mazzeo-Vertman heat space, analogous to the overblown 0-calculus of \cite{Mazzeo-Melrose:Zero} developed in \cite{Lauter}. The front face of $\bhs{hmf}$ is $\bhs{hmf}\cap\bhs{hff}$. \\

\textbf{Step 2':}  Relation between the $b$-heat space and the scattering heat space. \\

Using the coordinates 
\begin{equation}
  \widetilde{u}:= \frac{y-y'}{\sqrt{\rho^2+(\rho')^2}}= \frac{y-y'}{\sqrt{2+x^2+(x')^2}},  \quad \widetilde{\sigma}:= \frac{\tau}{\rho^2+(\rho')^2},
\label{co.1}\end{equation}
the face $\bhs{hff}$ fibers over $Y$ with typical fiber $Z^2\times [-\frac{\pi}2,\frac{\pi}2]_{ob}^2\times \overline{\bbR^h_{\widetilde{u}} \times \bbR^+_{\widetilde{\sigma}}}$, where $\overline{\bbR^h_{\widetilde{u}}  \times \bbR^+_{\widetilde{\sigma}}}$ denotes the radial compactification of $\bbR^h_{\widetilde{u}}  \times \bbR^+_{\widetilde{\sigma}}$.
We will see that in constructing the heat kernel each of these fibers will carry the heat kernel of a model problem, namely the product of the Euclidean heat kernel on $\bbR^h$ with the heat kernel corresponding to the metric \linebreak $ dX^2 + (1+X^2) g_Z$ on $\bbR_X \times Z.$ This is a scattering metric with two ends and the heat kernel construction has been carried out in \cite{Sher:ACHeat}. 

However, at the moment, the fibers of $\bhs{hff}$ looks more like a $b$-heat space rather than a scattering heat space.  In the remaining steps of the construction, we will make a series of blow-ups that will transform the fibers of $\bhs{hff}$ into a scattering heat space.  To see what blow-ups need to be performed, it is helpful to first see how, at the cost of changing the notion of time, the $b$-heat space can be seen as a blow-down of the scattering heat space.    

Thus, let $W$ be a compact manifold with boundary and $r$ a choice of boundary defining function.  In terms of the  $b$-double space $W^2_b$, the compactified $b$-heat space is given by 
$$
    HW_b:= [W^2_b\times \overline{\bbR^+_{\tau}}; D_b\times \{0\}],  \quad \tau= \sqrt{t},
$$
where $D_b\subset W^2_b$ is the interior lift of the diagonal.  Let us denote by $\zf$ and $\tb$ the faces corresponding to $t=\infty$ and $t=0$, by $\bff_0$ the face in $HW_b$ corresponding to the front face of the $b$-double space, by $\tf$ the face created by the blow-up of $D_b\times\{0\}$, by $\rb_0$ and $\lb_0$ the faces corresponding to the right and left faces of the $b$-double space.  A picture of the $b$-heat space is then given in 
  Figure~\ref{fig:scatteringheat}.  Now, at this stage, the fibers of $\bhs{hff}$ essentially correspond to $W^2_b\times \overline{\bbR^+_{\tau}}$, that is, to the $b$-heat space before we blow up the $b$-diagonal at $\tau=0$ to create the boundary hypersurface $\tf$.  For the convenience of the reader, the following table explains how the boundary faces of $\bhs{hff}$ correspond with those of $W^2_b\times \overline{\bbR^+_{\tau}}$:
\begin{equation}    
\begin{tabular}{|c||c|c|c|c|c|}
\hline
 & & & & & \\
$\bhs{hff}$ & $\bhs{tb}\cap\bhs{hff}$ & $\bhs{hrf}\cap\bhs{hff}$ & $\bhs{hlf}\cap\bhs{hff}$ & $\bhs{hmf}\cap\bhs{hff}$ & $\bhs{hbf}\cap\bhs{hff}$  \\
& & & & & \\
\hline
& & & & & \\
$W^2_b\times \overline{\bbR^+_{\tau}}$ & $\tb$ & $\lb_0$ & $\rb_0$ & $\bff_0$ & $\zf$ \\
& & & & & \\
\hline 
\end{tabular} \; 
\end{equation}

In terms of the $b$-heat space, the scattering heat space of  \cite{Sher:ACHeat} is then given by
\begin{equation}
  HW_{\sc}:= [HW_b; \tb\cap \rb_0, \tb\cap \lb_0, \tf\cap \bff_0,\tb\cap \bff_0],
\label{sch.1}\end{equation}
where we denote correspondingly by  $\rb$, $\lb$, $\sc$, and $\bff$ the new faces created by these blow-ups, \cf Figure~\ref{fig:scatteringheat} with \cite[Figure~1]{Sher:ACHeat}.

\begin{figure}
	\centering
	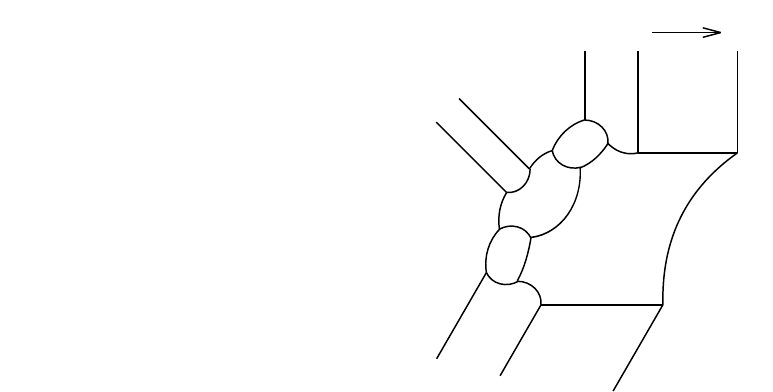
	\caption{The $b$-heat space, left, as a blow-down of the scattering heat space with a rescaled time, right.}
	\label{fig:scatteringheat}
\end{figure}

Notice that by the commutativity of nested blow-ups, many blow-ups commute.  For instance, we could have instead proceeded in the following order,
\begin{equation}
  HW_{\sc}=[W^2_b\times \overline{\bbR^+_{\tau}}; \tb\cap \rb_0, \tb\cap \lb_0, \tb\cap \bff_0, (D_b\times \overline{\bbR^+_{\tau}})\cap \bff, (D_b \times \overline{\bbR^+_{\tau}})\cap \tb].
\label{sch.2}\end{equation} 
From this latter point of view, the first three blow-ups are there to introduce a rescaled notion of time that fits with the original description given in  \cite{Sher:ACHeat}.  Indeed, after the first two blow-ups, $\tau$ needs to be replaced by
$$
     \frac{\tau (r^2+(r')^2)}{rr'}
$$
to have a boundary defining function for $\tb$.  After the third blow-up is performed, a boundary defining function for $\tb$ is given by 
\begin{equation}
  \frac{\tau \sqrt{r^2+(r')^2}}{rr'}.
\label{sch.3}\end{equation}
This may look at first like at strange rescaled time, but making the substitution $(\tau,r,r')\to (\widetilde{\sigma}, \ang{X}^{-1}, \ang{X'}^{-1})$ in \eqref{sch.3}  gives us $\tau/\eps$, which will be the natural rescaled time for the model problem at $\bhs{hff}$.  

To complete our construction of the wedge surgery heat space, the plan is then to perform blow-ups on the space \eqref{wsh.1} which on the face $\bhs{hff}$ will correspond to the blow-ups of \eqref{sch.2},  but with the following three differences:
\begin{enumerate}
\item On the fibers of $\bhs{hff}$, the scattering manifold we consider is $(Z\times \bbR^h_{\widetilde{u}})\times \overline{\bbR}_X$, where $\bbR^h_{\widetilde{u}}$ is seen as part of the cross-section though it is not compact;
\item We will not do the blow-ups corresponding to the first two blow-ups in \eqref{sch.2}, namely the blow-ups of $\tb\cap \rb_0$ and $\tb\cap \lb_0$;

\item We work with the overblown $b$-double space instead of the $b$-double space.
\end{enumerate}
For the first point, this ensures that the last two blow-ups in \eqref{sch.2} have the effect of replacing the variable $\widetilde{u}$ by the variable $u= \frac{y-y'}{\eps}$, which is better adapted to the model problem.  For the second point, we can do this since in \cite{Sher:ACHeat}, the scattering heat kernel is shown to vanish rapidly at the faces $\rb$ and $\lb$ created by these blow-ups, so they can be omitted in the description of the heat kernel.  The advantage is that there will be fewer blow-ups in the construction of the wedge surgery heat space.  The disadvantage is that the fibers of $\bhs{hff}$ will not quite correspond to a scattering heat space, not even a blow-down version since the blow-ups of $\tb\cap \rb_0$ and $\tb\cap \lb_0$ cannot be commuted at the end.  Nevertheless, it can be interpreted as a scattering heat space with fewer submanifolds blown up, that is, a somewhat blow-down version of the scattering heat space.  

We are now ready to complete the construction of the wedge surgery heat space.  \\

\textbf{Step 3:} Blow up $\bhs{hmf}\cap\bhs{tb}$ and call the new face $\bhs{stf}$.\\

This blow-up correspond to the blow-up of $\tb\cap\bff_0$ in \eqref{sch.2}.  Near $\bhs{tb}\cap\bhs{hff}\cap\bhs{stf}$, the coordinates \eqref{coords:step2} must be replaced by
\begin{equation}\label{coords:step3smallsigma}
	\lrpar{ x, \quad
	\frac{x'}{x}, \quad
	\sigma\frac{x'}{x}=\frac{\tau}{\eps}\frac{x'}{x}, \quad
	\frac{\eps}{x'}, \quad
	y, \quad
	\frac{y'-y}{x}, \quad
	z, \quad
	z'}.
\end{equation}
Here $\rho_{hff}=x$, $\rho_{tb}=\frac{\sigma x'}x$, and $\rho_{stf}=\eps/x'$. 
In the other regime, near $\bhs{hmf}\cap\bhs{hff}\cap\bhs{stf}$, we instead have
\begin{equation}\label{coords:step3largesigma}
	\lrpar{x, \quad
	\frac{x'}{x}, \quad
	\sigma^{-1}\frac{x}{x'}, \quad
	\frac{\tau}{x}, \quad
	y, \quad
	\frac{y'-y}{x}, \quad
	z, \quad
	z'},
\end{equation}
with $\rho_{hff}=x$, $\rho_{stf}=\tau/x$, and $\rho_{hmf}=\frac{x}{\sigma x'}$. Note also that near $\bhs{hmf}$ and $\bhs{stf}$ but away from $\bhs{hff}$, $\bhs{tb}$ and $\bhs{hbf}$, we have the simpler coordinates
\begin{equation}\label{coords:step3interior}
	\lrpar{ x, \quad
	x', \quad
	\tau, \quad
	\sigma^{-1}, \quad
	y, \quad 
	y', \quad
	z, \quad 
	z'}.
\end{equation}

In terms of \eqref{sch.2}, we have now created the face $\bff$ for the scattering heat space at $\bhs{hff}$ (note that at $\bhs{hff}\cap\bhs{stf}$ we have $\sigma$ nonzero, $\eps/x=0$, $\eps/x'=0$, and $(\eps/x)/(\eps/x')$ nonzero), and there has been no effect on the geometry of $\bhs{hmf}$; the zero-time boundary of $\bhs{hmf}$ is $\bhs{hmf}\cap\bhs{stf}$. We still need to deal with the diagonals.\\

\textbf{Step 4:} Blow up the intersection of the closure of the lift of $\diag_M\times[0,1)_{\eps}\times \bbR^+_{\tau}$ with $\bhs{stf}$. Call the new face $\bhs{sf}$. This blow-up has two purposes: it creates the scattering face $\sc$ in the notation of \eqref{sch.2} for the scattering heat space at $\bhs{hff}$ and creates the $t=0$ diagonal of the wedge heat space at $\bhs{hmf}$, completing the Mazzeo-Vertman over-blown heat space. 

Projective coordinates on the interior of $\bhs{sf}$, away from $\bhs{hff}$, can be obtained by modifying \eqref{coords:step3interior} for large $\sigma$,
\begin{equation}\label{coords:step4interior}
	\lrpar{\Theta_x:=\frac{x'-x}{\tau}, \quad
	x, \quad
	\tau, \quad
	\sigma^{-1}, \quad
	\Theta_y:=\frac{y'-y}{\tau}, \quad
	y, \quad
	\Theta_z:=\frac{z'-z}{\tau}, \quad
	z}.
\end{equation}
Similarly for small $\sigma$ we have
\begin{equation}\label{coords:step4smallsigmainterior}
	\lrpar{\Xi_x:=\frac{x'-x}{\eps}, \quad
	x, \quad
	\sigma, \quad 
	\eps, \quad
	\Xi_y:=\frac{y'-y}{\eps}, \quad
	y, \quad
	\Xi_z:=\frac{z'-z}{\eps}, \quad 
	z}.
\end{equation}
On the interior of $\bhs{sf}$ near $\bhs{hff}$, we may assume that $\frac12< \frac{x'}x<2$, so coordinates \eqref{coords:step3smallsigma} and \eqref{coords:step3largesigma} may be replaced by
\begin{equation}\label{coords:step4smallsigma}
	\lrpar{x, \quad
	\Xi_x, \quad
	\sigma, \quad
	\eta:=\frac{\eps}{\rho}, \quad
	y, \quad
	\Xi_y, \quad 
	\bar\Xi_z:=\frac{z'-z}{\eps/\rho}, \quad
	z}
\end{equation}
for small $\sigma$ and by

\begin{equation}\label{coords:step4largesigma}
	\lrpar{x, \quad
	\Theta_x, \quad
	\sigma^{-1}, \quad
	\frac{\tau}{\rho}, \quad
	y, \quad
	\Theta_y, \quad
	\bar\Theta_z:=\frac{z'-z}{\tau/\rho}, \quad 
	z}
\end{equation}
for large $\sigma$.
Note that the coordinates for large $\sigma$ restrict to $\bhs{hmf}$ (where $\sigma^{-1}=0$) to be good coordinates on the Mazzeo-Vertman heat space in the interior of the temporal front face.\\

\textbf{Step 5:} Blow up the intersection of the  interior lift of $\diag_M\times [0,1)_{\eps}\times \bbR^+_{\tau}$ with $\bhs{tb}$ to create the new boundary face $\bhs{tf}$. This creates the $t=0$ diagonal for positive $\eps$  and does not intersect $\bhs{hmf}$.  On $\bhs{hff}$, this blow-up has the effect of creating the boundary hypersurface corresponding to $\tf$ in the scattering heat space \eqref{sch.2}. 

The final space is the \textbf{wedge-surgery heat space}, denoted $HX_{\w,s}$. Specifically, we have
\[HX_{\w,s}=[X_{b,s}^2\times \bbR^+_{\tau}; \phi_b^{-1}(D_Y)\times\{0\};\bhs{hmf}\cap\bhs{tb};\bhs{stf}\cap \operatorname{DH}_M;\bhs{tb}\cap\operatorname{DH}_M],\]
where $\operatorname{DH}_M$ denotes the appropriate interior lift of $\diag_M\times [0,1)_{\eps}\times \bbR^+_{\tau}$.  
Coordinates valid in the interior of $\bhs{tf}$ away from other boundary hypersurfaces include
\begin{equation}\label{coords:step5part0}
	\lrpar{x, \quad 
	y, \quad
	z, \quad
	\Theta_x, \quad
	\Theta_y, \quad
	\Theta_z, \quad
	\eps, \quad
	\tau},
\end{equation}
in which $\tau = \sqrt t$ is a bdf for $\bhs{tf}$.

Near $\bhs{sf}\cap\bhs{tf}$ and before performing the last blow-up, we can use the coordinates \eqref{coords:step4smallsigmainterior}.  In these coordinates, the submanifold being blown up is $\{\sigma=\Xi_x=\Xi_y=\Xi_z=0\}$, which yields coordinates on the blow-up in the interior of $\bhs{tf}$ given by
\begin{equation}\label{coords:step5part1}
	\lrpar{x,\quad 
	\Theta_x, \quad
	\sigma, \quad
	\eps, \quad
	y, \quad
	\Theta_y, \quad
	z, \quad
	\Theta_z}.
\end{equation}
Similarly, the coordinates \eqref{coords:step4smallsigma} valid near $\bhs{hff}\cap\bhs{sf}$ must be replaced by
\begin{equation}\label{coords:step5part2}
	\lrpar{x, \quad
	\Theta_x, \quad
	\sigma, \quad
	\eta, \quad
	y, \quad
	\Theta_y, \quad
	\bar{\Theta}_z:=\frac{z'-z}{\tau/\rho}, \quad
	z}
\end{equation}
after the last blow-up.  
Finally, in the interior of $\bhs{tf}$, near $\bhs{hff}$, but away from $\bhs{sf}$, we must replace the coordinates \eqref{coords:interiorhff} by 
\begin{equation}\label{coords:step5part3}
	\lrpar{\sigma, \quad
	\Theta_x, \quad
	X, \quad
	\eps, \quad
	y, \quad
	\Theta_y, \quad
	z, \quad
	\widetilde{\Theta}_z:=\frac{z'-z}{\sigma}}.
\end{equation}

\textbf{Summary:} The heat space has a total of nine boundary hypersurfaces
\begin{equation*}
	\bhs{hmf}, \quad
	\bhs{tb}, \quad
	\bhs{hff}, \quad
	\bhs{hbf}, \quad
	\bhs{hlf}, \quad
	\bhs{hrf},\quad
	\bhs{stf}, \quad
	\bhs{sf}, \quad
	\bhs{tf}.
\end{equation*}

We informally summarize these faces in Figure \ref{fig:HeatTable} where we indicate, for instance, that $\bhs{hmf}$ sits over where $\eps$ vanishes but $x$ does not vanish and $x'$ does not vanish.

\begin{figure}%
  \caption{A schematic summary of the heat space boundary hypersurfaces.}%
  \label{fig:HeatTable}%
\begin{tabular}{|c|l|l|}
\hline Boundary & Vanishing & Blow-up codimension\\
\hline $
\bhs{hmf}
	$ & $
	\eps, !x, !x', !\tau
	$ &
	\\ \hline $
\bhs{tb}
	$&$
	\tau, !\eps, !(\zeta=\zeta')
	$ & 
	\\ \hline $
\bhs{hbf}
	$&$
	\eps, x, x', !(y=y',\tau)
	$ & 2
	\\ \hline $
\bhs{hlf}
	$&$
	\eps, x,  !x'
	$ & 1
	\\ \hline $
\bhs{hrf}
	$&$
	\eps, x', !x
	$ & 1
	\\ \hline $
\bhs{hff}
	$&$
	\eps, x, x', \tau, y=y'
	$ & h+3 
	\\ \hline $
\bhs{stf}
	$&$
	\eps, \tau, !(\zeta=\zeta'), !x, !x'
	$ & 1
	\\ \hline $
\bhs{sf}
	$&$
	\eps, \tau, \zeta=\zeta', !x, !x'
	$ & m+ 1
	\\ \hline $
\bhs{tf}
	$&$
	\zeta=\zeta', \tau, !\eps
	$ & m
	\\ \hline 
\end{tabular} \\
\end{figure}
\begin{figure}%
  \caption{A schematic summary of the reduced heat space boundary hypersurfaces.}%
  \label{fig:RedHeatTable}%
\begin{tabular}{|c|l|l|}
\hline Boundary & Vanishing & Blow-up codimension\\
\hline $
\bhs{hmf}
	$ & $
	\eps, !x, !x'
	$ &
	\\ \hline $
\bhs{tb}
	$&$
	\tau, !(\eps, x, x', y=y')
	$ & 
	\\ \hline $
\bhs{hbf}
	$&$
	\eps, x, x', !(y=y',\tau)
	$ & 2
	\\ \hline $
\bhs{hlf}
	$&$
	\eps, x,  !x'
	$ & 1
	\\ \hline $
\bhs{hrf}
	$&$
	\eps, x', !x
	$ & 1
	\\ \hline $
\bhs{hff}
	$&$
	\eps, x, x', \tau, y=y'
	$ & h+3 
	\\ \hline 
\end{tabular} \\
\end{figure}

The face $\bhs{hmf}$ is isomorphic to a blown-up version of the Mazzeo-Vertman heat space and the face $\bhs{hff}$ is isomorphic to a somewhat blown-down version of the scattering heat space of \cite{Albin:Heat, Sher:ACHeat, Guillarmou-Sher}. Specifically, the faces of the Mazzeo-Vertman heat space, denoted 
$\{\textrm{lf}, \textrm{rf}, \textrm{ff}, \textrm{tf}, \textrm{td} \}$ in \cite{Mazzeo-Vertman}, correspond to faces of $HX_{\w,s}$ by
\begin{equation*}
\begin{gathered}
	\textrm{lf}\sim\bhs{hmf}\cap\bhs{hlf}; \quad 
	\textrm{rf}\sim\bhs{hmf}\cap\bhs{hrf}; \quad
	\textrm{ff}\sim\bhs{hmf}\cap\bhs{hff}; \\
	\textrm{tf}\sim\bhs{hmf}\cap\bhs{stf}; \quad 
	\textrm{td}\sim\bhs{hmf}\cap\bhs{sf},
\end{gathered}
\end{equation*}
and the faces of the `ac heat space', denoted 
$\{\textrm{tf}, \textrm{tb}, \textrm{sc}, \textrm{bf}_0, \textrm{lb}_0, \textrm{lb}, \textrm{rb}_0, \textrm{rb}, \textrm{zf} \}$
in \eqref{sch.1}, correspond to faces of $HX_{\w,s}$ by
\begin{equation*}
\begin{gathered}
	\textrm{tf}\sim\bhs{hff}\cap\bhs{tf}, \quad 
	\textrm{tb}\sim \bhs{hff}\cap\bhs{tb}, \quad 
	\textrm{sc}\sim\bhs{hff}\cap\bhs{sf}, \quad
	\textrm{bf}\sim\bhs{hff}\cap\bhs{stf}, \\
	\textrm{bf}_0\sim\bhs{hff}\cap\bhs{hmf}, \quad 
	\textrm{lb}_0\sim\bhs{hff}\cap\bhs{hrf}, \quad 
	\textrm{rb}_0\sim\bhs{hff}\cap\bhs{hlf}, \quad
	\textrm{zf}\sim\bhs{hff}\cap\bhs{hbf},
\end{gathered}
\end{equation*}
except for $\textrm{lb}$ and $\textrm{rb}$ which do not have an analogue in $\bhs{hff}$ and are somehow absorbed in $\bhs{hff}\cap \bhs{tb}$.  
Note that lb$_0$ corresponds to $\bhs{hrf}$ and rb$_0$ to $\bhs{hlf}$, which seems strange at first. However, this is natural, because the blow-ups in $HX_{\w,s}$ are at $x=0$ and the blow-ups in the scattering heat space are at spatial infinity.

\subsection{Triple heat space} \label{sec:TripleHeat}

To simplify the constructions below, we will prove composition of two operators assuming that one of them lies over a `reduced heat space'. This corresponds to conormal sections on $HX_{\w,s}$ that vanish to infinite order at $\bhs{stf},$ $\bhs{sf},$ and $\bhs{tf}$ as well as $\bhs{tb}.$
Thus the reduced heat space is constructed following steps 1 and 2 of the construction of $HX_{\w,s},$ so is given by 
\begin{equation*}
	RHX_{\w,s} := [X^2_{b,s} \times \bbR^+_{\tau}; \phi_b^{-1}(D_Y)\times \{\tau=0\} ]
\end{equation*}
with blow-down map $\beta_{RH}: RHX_{\w,s}\to M^2\times [0,1]_{\eps}\times \bbR^+_{\tau}$ and boundary hypersurfaces
\begin{equation*}
	\bhs{hmf}, \quad
	\bhs{tb}, \quad
	\bhs{hbf}, \quad
	\bhs{hlf}, \quad
	\bhs{hrf}, \quad
	\bhs{hff}.
\end{equation*}
This is summarize in Figure~\ref{fig:RedHeatTable}.

Our aim in constructing the triple space is a geometric understanding of the composition of two heat operators, given by
\begin{equation}\label{eq:ConvolutionProduct}
	\cK_{C}(\eps, \zeta, \zeta'', t) = \int_0^t \int_M \cK_A(\eps, \zeta, \zeta',t-t')\cK_B(\eps, \zeta',\zeta'', t') \; \dvol_b'' dt'.
\end{equation}
For orientation, let us first focus on the time variables and consider
\begin{equation*}
	f(t) \; dt = \int_0^t g(t-t')h(t') \; dt' dt= \int_{t''+t' = t} (g(t'') \; dt'') (h(t') \; dt').
\end{equation*}
We prefer to work with $\tau = \sqrt t$ instead of $t,$ so this becomes
\begin{equation*}
	f(\tau) \; d\tau = \frac2\tau\int_{\sqrt{s^2+(s')^2} = \tau} (ss') 
	(g(s) \; ds) (h(s')\; ds').
\end{equation*}
Let $T^2 = \bbR^+_{s} \times \bbR^+_{s'},$ and consider the maps
\begin{equation*}
	\xymatrix{ & (s,s') \ar@{|->}[ld]^{\pi_L} \ar@{|->}[d]^{\pi_s} \ar@{|->}[rd]^{\pi_R} & \\
	s & \sqrt{s^2+(s')^2} & s' }
\end{equation*}
The map $\pi_s$ is not a b-fibration, but if we replace $T^2$ with $T^2_b = [T^2; \{ (0,0)\}]=\bbR^+_r\times[0,\pi/2]_{\theta}$ and these maps with their interior lifts, then we do get a diagram of $b$-fibrations
\begin{equation*}
	\xymatrix{ & (r,\theta) \ar@{|->}[ld]^{\pi_L} \ar@{|->}[d]^{\pi_s} \ar@{|->}[rd]^{\pi_R} & \\
	r\cos\theta & r & r\sin\theta. }
\end{equation*}
Indeed, since $\bbR^+$ is a manifold with boundary, this is the same as saying that they are b-maps and fibrations over the interior.
The composition formula is then
\begin{equation*}
	f(\tau) \; d\tau = \frac2\tau (\pi_s)_*
	\lrpar{ \pi_L^*(\tau g(\tau)\; d\tau) \pi_R^*(\tau h(\tau)\; d\tau)}.
\end{equation*}
$ $

Now, to construct the triple space, we 
begin with $M^3 \times [0,1]_{\eps} \times \bbR^+_{\tau} \times \bbR^+_{\tau'} $
and the three maps into $M^2  \times [0,1]_{\eps}\times \bbR^+,$
\begin{equation*}
	\xymatrix{
	& (\zeta, \zeta', \zeta'', \eps, \tau, \tau') \ar@{|->}[ld]_{\pi_{LM}} \ar@{|->}[d]^{\pi_{LR}} \ar@{|->}[dr]^{\pi_{MR}} &\\
	(\zeta, \zeta', \eps, \tau) & (\zeta, \zeta'', \eps, \sqrt{\tau^2+(\tau')^2}) & (\zeta', \zeta'', \eps, \tau'). }
\end{equation*}
Note that $\pi_{LM}$ and $\pi_{MR}$ are $b$-fibrations, but $\pi_{LR}$ is not.
We seek a space $R'HX_{\w,s}^3$ to which these maps lift to $b$-maps, with the lift of $\pi_{LM}$ mapping into $HX_{\w,s}$ and the lifts of $\pi_{LR}$ and $\pi_{MR}$ mapping into $RHX_{\w,s}.$ 

Denote the boundary hypersurfaces of $M^3 \times [0,1]_{\eps}\times \bbR^+_{\tau} \times \bbR^+_{\tau'} $ by
\begin{equation*}
	\bhs{\ell} = \{\tau =0 \}, \quad \bhs{r} = \{\tau'=0\}, \quad \bhs{hZ} = \{\eps =0\}.
\end{equation*}
We do not assign a name to $\{\eps =1\}$ as we will ignore this face.

We first blow-up the lift of $\{\tau=\tau'=0\}$ to obtain $M^3\times [0,1]_{\eps}\times T^2_b$ and denote the resulting boundary hypersurface by $\bhs{O}.$ 

We then pass to $X^3_{b,s} \times T^2_b$ by blowing up
\begin{equation*}
\begin{gathered}
	\{\eps=x=x'=x''=0\}, \quad
	\{\eps=x=x'=0\}, \quad
	\{\eps=x=x''=0\}, \quad
	\{\eps=x'=x''=0\}, \\
	\{\eps=x''=0\}, \quad
	\{\eps=x'=0\}, \quad
	\{\eps=x=0\}.
\end{gathered}
\end{equation*}
We label the resulting boundary hypersurfaces by 
\begin{equation*}
	\bhs{hT}, \quad
	\bhs{hF}, \quad
	\bhs{hC}, \quad
	\bhs{hS}, \quad
	\bhs{hN_1}, \quad
	\bhs{hN_2}, \quad
	\bhs{hN_3}
\end{equation*}
respectively.

For this new space, the three maps above lift to define b-fibrations
\begin{equation*}
	\xymatrix{
	& X^3_{b,s}\times T^2_b \ar[ld]^-{\pi_{LM,b}} \ar[d]^-{\pi_{LR,b}} \ar[rd]^-{\pi_{MR,b}} & \\
	X^2_{b,s}\times \bbR^+ &	X^2_{b,s}\times \bbR^+ & X^2_{b,s}\times \bbR^+.}
\end{equation*}

Next, in view of \cite[Lemma~2.5]{hmm}, we want to lift 
$\phi_b^{-1}(D_Y)\times \{\tau=0\}$ along these three maps, decompose it into p-submanifolds, and then blow these up.
Looking at $\pi_{LM,b}$ to start, note that the inverse image decomposes into four p-submanifolds corresponding to its intersections with 
$\bhs{hT}\cap \bhs{\ell},$ 
$\bhs{hT} \cap \bhs{O},$
$\bhs{hF}\cap \bhs{\ell},$ and
$\bhs{hF} \cap \bhs{O}.$
So we first blow-up
\begin{equation*}
	\bhs{hT} \cap \{y=y'=y''\} \cap \bhs{O}
\end{equation*}
and denote the resulting boundary hypersurface by $\bhs{hTT},$ and then we blow-up the interior lifts of 
\begin{equation*}
\begin{gathered}
	\bhs{hT}\cap \{y=y'\} \cap \bhs{O}, \quad
	\bhs{hT}\cap \{y=y''\} \cap \bhs{O}, \quad
	\bhs{hT}\cap \{y'=y''\} \cap \bhs{O}, \\
	\bhs{hT}\cap \{y=y'\} \cap \bhs{\ell}, \quad
	\bhs{hT}\cap \{y'=y''\} \cap \bhs{r},
\end{gathered}
\end{equation*}
and denote the resulting boundary hypersurfaces by $\bhs{hFT},$ $\bhs{hCT},$  $\bhs{hST},$ $\bhs{hT\ell},$ and $\bhs{hTr},$ respectively.
Next we blow-up the interior lifts of 
\begin{equation*}
\begin{gathered}
	\bhs{hF}\cap \{y=y'\} \cap \bhs{O}, \quad
	\bhs{hF}\cap \{y=y'\} \cap \bhs{\ell}, \\
	\bhs{hC}\cap \{y=y''\} \cap \bhs{O}, \\
	\bhs{hS}\cap \{y'=y''\} \cap \bhs{O}, \quad
	\bhs{hS}\cap \{y'=y''\} \cap \bhs{r},
\end{gathered}
\end{equation*}
and denote the resulting boundary hypersurfaces by
$\bhs{hFDO},$ $\bhs{hFD\ell},$ $\bhs{hCD},$ $\bhs{hSDO},$ and $\bhs{hSDr},$ respectively.

Denote the result of these blow-ups by $RHX^3_{\w,s}$ with blow-down map 
$$
\beta_{RH,3}: RHX^3_{\w,s}\to M^3\times [0,1]_{\eps}\times  \bbR^+_{s}\times \bbR^+_{s'}.
$$  
Thanks to \cite[Lemma~2.1 and 2.5]{hmm}, the maps above lift to $b$-submersions
\begin{equation}
	\xymatrix{
	& RHX_{\w,s}^3 \ar[ld]^{\beta^{R}_{LM}} \ar[d]^{\beta^{R}_{LR}} \ar[dr]^{\beta^{R}_{MR}} & \\
	RHX_{\w,s} & RHX_{\w,s} & RHX_{\w,s}.}
\label{bsub.1}\end{equation}
However, they are not b-fibrations since they are not $b$-normal. Indeed, in each case, some boundary hypersurface is mapped into a codimension 2 corner.  For instance, the map $\beta^R_{LR}$ maps $\bhs{hFT}$ into a codimension 2 corner. 

Finally, we want to blow-up the lifts along $\beta_{LM}$ of the submanifolds of $RHX_{\w,s}$ producing the faces $\bhs{stf},$ $\bhs{sf},$ and $\bhs{tf}.$
First, the inverse image of $\bhs{hmf}\cap \bhs{tb}$ decomposes into four p-submanifolds,
\begin{equation*}
	\bhs{hZ}\cap \bhs{\ell}, \quad \bhs{hZ}\cap \bhs{O}, \quad
	\bhs{hN_1} \cap \bhs{\ell}, \quad \bhs{hN_1} \cap \bhs{O},
\end{equation*}
whose blow-ups produce the boundary hypersurfaces denoted $\bhs{hZ\ell},$ $\bhs{hZO},$ $\bhs{hN_1\ell},$ and $\bhs{hN_1O},$ respectively.
The inverse image of $\operatorname{DH}_M\cap \bhs{stf}$ is the intersection of each of these four boundary hypersurfaces with the interior lift of 
\begin{equation*}
	\diag_{LM} = \diag_M \times M \times \bbR^+_{\tau}\times \bbR^+_{\tau'} \times [0,1]_{\eps}.
\end{equation*}
Blowing these up produces the boundary hypersurfaces $\bhs{hZ{\ell} sf},$ $\bhs{hZO sf},$ $\bhs{hN_1\ell sf},$ and $\bhs{hN_1 O sf}.$ Finally, the inverse image of $\operatorname{DH}_M\cap \bhs{tb}$ decomposes into the intersection of the interior lift of $\diag_{LM}$ with $\bhs{\ell}$ and with $\bhs{O};$ we blow both of these up and denote the resulting hypersurfaces by $\bhs{tf\ell}$ and $\bhs{tfO}.$  We denote the resulting space by $R'HX^3_{\w,s}.$  Its 32 boundary hypersurfaces are summarized in Figure \ref{fig:RedTripleHeat}.  The point of all these blow-ups is to obtain the following.

\begin{lemma}
The diagram \eqref{bsub.1} lifts to the diagram of $b$-submersions
\begin{equation*}
	\xymatrix{
	& R'HX_{\w,s}^3 \ar[ld]^{\beta_{LM}} \ar[d]^{\beta_{LR}} \ar[dr]^{\beta_{MR}} & \\
	HX_{\w,s} & RHX_{\w,s} & RHX_{\w,s}. }
\end{equation*}
\end{lemma}
\begin{proof}
Clearly, we have that $\beta_{LR}= \beta_{LR}^{R}\circ \beta'$ and $\beta_{MR}= \beta_{MR}^{R}\circ \beta'$ where  
$$
\beta': R'HX_{\w,s}^3\to RHX_{\w,s}^3
$$ 
is the blow-down map.  The lift of $\beta^R_{LM}$ is more delicate to construct and we need to use the proof of \cite[Lemma~2.5]{hmm}.  Indeed, since $\beta_{LM}^R$ is not a $b$-fibration, we cannot directly use the statement of \cite[Lemma~2.5]{hmm}.  However, for each of the blow-ups performed to obtain $HX_{\w,s}$ from $RHX_{\w,s}$, using among other things the fact that $\beta_{LM}^R$ is a $b$-submersion, we can nevertheless construct coordinates as in \cite[(2.17)]{hmm}, so that the proof of \cite[Lemma~2.5]{hmm} can be applied to define a lift $\beta_{LM}$ which is a $b$-submersion.  
\end{proof}

The $b$-submersions $\beta_{LM},$ $\beta_{LR},$ and $\beta_{MR}$ are not $b$-fibrations since they are not $b$-normal.  This can be seen directly from their exponent matrices:  
\begin{itemize}
\item [(a)] The $b$-maps $\beta_{LM}$ sends 
$\bhs{r}$ 
into the interior of $HX_{\w,s}$ and otherwise has exponent matrix with entries zero and one determined by
\begin{equation}
\begin{gathered}
	\beta_{LM}^*\bhs{hmf}= \{ \bhs{hZ}, \bhs{hN_1} \}, \quad
	\beta_{LM}^*\bhs{tb} = \{ \bhs{\ell}, \bhs{O}, \bhs{hCT}, \bhs{hST}, \bhs{hCD}, \bhs{hSDO} \}, \\
	\beta_{LM}^*\bhs{hbf} = \{ \bhs{hT}, \bhs{hF}, \bhs{hTr}, \bhs{hCT}, \bhs{hST} \}, \quad
	\beta_{LM}^*\bhs{hlf} = \{ \bhs{hC}, \bhs{hN_3}, \bhs{hCD} \}, \\
	\beta_{LM}^*\bhs{hrf} = \{  \bhs{hS}, \bhs{hN_2}, \bhs{hSDO}, \bhs{hSDr} \}, \\
	\beta_{LM}^*\bhs{hff} = \{ \bhs{hTT}, \bhs{hT\ell}, \bhs{hFT},  \bhs{hFD\ell}, \bhs{hFDO} \} \\
	\beta_{LM}^*\bhs{stf} = \{ \bhs{hZ\ell}, \bhs{hZO},  \bhs{hN_1\ell}, \bhs{hN_1O} \}, \\
	\beta_{LM}^*\bhs{sf} = \{ \bhs{hZ\ell sf}, \bhs{hZO sf},  \bhs{hN_1\ell sf}, \bhs{hN_1O sf} \}, \quad
	\beta_{LM}^*\bhs{tf} = \{ \bhs{tf\ell}, \bhs{tfO} \},
\end{gathered}
\label{exm.1}\end{equation}
with $\bhs{hCT},$ $\bhs{hST},$ $\bhs{hCD},$ and $\bhs{hSDO}$ repeated in the list;

\item [(b)] The $b$-map $\beta_{LR}$ maps $\bhs{\ell},$ $\bhs{r},$ and $\bhs{tf\ell}$ 
into the interior of $RHX_{\w,s}$ and otherwise has exponent matrix with entries zero and one determined by
\begin{equation}
\begin{gathered}
	\beta_{LR}^*\bhs{hmf}= \{ \bhs{hZ}, \bhs{hN_2}, \bhs{hZ\ell}, \bhs{hZO}, \bhs{hZ\ell sf}, \bhs{hZO sf}  \}, \\
	\beta_{LR}^*\bhs{tb} = \{ \bhs{O}, \bhs{hFT}, \bhs{hST}, \bhs{hFDO}, \bhs{hSDO}, 
		\bhs{hZO}, \bhs{hN_1O}, \bhs{hZO sf}, \bhs{hN_1O sf}, \bhs{tfO} \}, \\
	\beta_{LR}^*\bhs{hbf} = \{  \bhs{hT}, \bhs{hC}, \bhs{hT\ell}, \bhs{hTr}, \bhs{hFT}, \bhs{hST} \}, \\
	\beta_{LR}^*\bhs{hlf} = \{ \bhs{hF}, \bhs{hN_3}, \bhs{hFD\ell}, \bhs{hFDO} \}, \\
	\beta_{LR}^*\bhs{hrf} = \{  \bhs{hS}, \bhs{hN_1}, \bhs{hSDO}, \bhs{SDr}, \bhs{hN_1\ell}, \bhs{hN_1 O}, 
		\bhs{hN_1\ell sf}, \bhs{hN_1O sf} \}, \\
	\beta_{LR}^*\bhs{hff} = \{ \bhs{hTT}, \bhs{hCT}, \bhs{hCD}  \},
\end{gathered}
\label{exm.2}\end{equation}
with $\bhs{hFT},$ $\bhs{hST},$ $\bhs{FDO},$ $\bhs{hSDO},$ $\bhs{hZO},$ $\bhs{hN_1O},$ $\bhs{hZO sf},$ and $\bhs{hN_1O sf}$ repeated in this list;

\item [(c)] The $b$-map $\beta_{MR}$ maps $\bhs{\ell}$ and $\bhs{tf\ell}$ 
into the interior of $RHX_{\w,s}$ and otherwise has exponent matrix with entries zero and one determined by
\begin{equation}
\begin{gathered}
	\beta_{MR}^*\bhs{hmf}= \{ \bhs{hZ}, \bhs{hN_3}, \bhs{hZ\ell}, \bhs{hZO}, \bhs{hZ\ell sf}, \bhs{hZO sf}  \}, \\
	\beta_{MR}^*\bhs{tb} = \{ \bhs{r}, \bhs{O}, \bhs{hFT}, \bhs{hCT}, \bhs{hFDO}, \bhs{hCD}, 
		\bhs{hZO}, \bhs{hN_1O}, \bhs{hZOsf}, \bhs{hN_1Osf}, \bhs{tfO}  \}, \\
	\beta_{MR}^*\bhs{hbf} = \{ \bhs{hT}, \bhs{hS}, \bhs{hT\ell}, \bhs{hFT}, \bhs{hCT} \}, \\
	\beta_{MR}^*\bhs{hlf} = \{ \bhs{hF}, \bhs{hN_2}, \bhs{hFD\ell}, \bhs{hFDO}  \}, \\
	\beta_{MR}^*\bhs{hrf} = \{ \bhs{hC}, \bhs{hN_1}, \bhs{hCD}, \bhs{hN_1\ell}, \bhs{hN_1O}, 
		\bhs{hN_1\ell sf}, \bhs{hN_1O sf} \}, \\
	\beta_{MR}^*\bhs{hff} = \{ \bhs{hTT}, \bhs{hTr}, \bhs{hST}, \bhs{hSDO}, \bhs{hSDr} \},
\end{gathered}
\label{exm.3}\end{equation}
with $\bhs{hFT},$ $\bhs{hCT},$ $\bhs{hFDO},$ $\bhs{hCD},$ $\bhs{hZO},$ $\bhs{hN_1O},$ $\bhs{hZOsf},$ and $\bhs{hN_1Osf}$
  repeated in this list.
\end{itemize}

\begin{figure}%
  \caption{A schematic summary of the reduced heat triple space boundary hypersurfaces.}%
  \label{fig:RedTripleHeat}%
\begin{tabular}{|c|l|l|}
\hline Boundary & Vanishing & Blow-up codimension\\
\hline $
\bhs{\ell}
	$ & $
	\tau,  !\tau', !(\eps,!x,!x'), !(\zeta=\zeta')
	$ & \\ \hline $
\bhs{r}
	$ & $
	\tau', !\tau, !\eps
	$ & \\ \hline $
\bhs{hZ}
	$&$
	\eps, !x, !x', !x'', !\tau, !(\tau,\tau')
	$ & \\ \hline $
\bhs{O}
	$&$
	\tau, \tau', !(y=y'), !(y=y''), !(y'=y''), !(\zeta=\zeta')
	$ & 1 \\ \hline $
\bhs{hT}
	$&$
	\eps, x,x',x'', !(y=y'', \tau, \tau'),  !(y'=y'', \tau'), !(y=y', \tau)
	$ & 3 \\ \hline $
\bhs{hF}
	$&$
	\eps, x, x', !x'', !(y=y', \tau)
	$ & 2 \\ \hline $
\bhs{hC}
	$&$
	\eps, x,x'', !x', !(y=y'', \tau, \tau')
	$ & 2\\ \hline $
\bhs{hS}
	$&$
	\eps, x',x'', !x, !(y'=y'', \tau')
	$ & 2 \\ \hline $
\bhs{hN_1}
	$&$
	\eps, x'', !x, !x', !\tau
	$ & 1 \\ \hline $
\bhs{hN_2}
	$&$
	\eps, x', !x, !x''
	$ & 1 \\ \hline $
\bhs{hN_3}
	$&$
	\eps, x, !x', !x''
	$ & 1 \\ \hline $
\bhs{hTT}
	$&$
	\eps, x, x', x'', y=y'=y'', \tau, \tau'
	$ & 2h + 5 \\ \hline $
\bhs{hT\ell}
	$&$
	\eps, x, x', x'', y=y', \tau, !(\tau', y=y'=y'')
	$ & h + 4 \\ \hline $
\bhs{hTr}
	$&$
	\eps, x, x', x'', y'=y'', \tau', !(\tau, y=y'=y'')
	$ & h + 4 \\ \hline $
\bhs{hFT}
	$&$
	\eps, x, x', x'', y=y', \tau, \tau', !(y=y'=y'')
	$ & h + 5 \\ \hline $
\bhs{hCT}
	$&$
	\eps, x, x', x'',  y=y'', \tau, \tau', !(y=y'=y'')
	$ & h + 5 \\ \hline $
\bhs{hST}
	$&$
	\eps, x, x', x'', y'=y'', \tau, \tau', !(y=y'=y'')
	$ & h + 5 \\ \hline $
\bhs{hFD\ell}
	$&$
	\eps, x, x', !x'', y=y', \tau, !\tau'
	$ & h + 3 \\ \hline $
\bhs{hFDO}
	$&$
	\eps, x, x', !x'', y=y', \tau, \tau'
	$ & h + 4 \\ \hline $
\bhs{hCD}
	$&$
	\eps, x, x'', !x', y=y'', \tau, \tau'
	$ & h + 4 \\ \hline $
\bhs{hSDO}
	$&$
	\eps, x', x'', !x, y'=y'', \tau, \tau'
	$ & h + 4 \\ \hline $
\bhs{hSDr}
	$&$
	\eps, x', x'', !x, y'=y'', \tau', !\tau
	$ & h + 3 \\ \hline $
\bhs{hZ\ell}
	$&$
	\eps, \tau, !x, !x', !x'', !\tau', !(\zeta=\zeta')
	$ & 1 \\ \hline $
\bhs{hZO}
	$&$
	\eps, \tau, \tau', !x, !x', !x'', !(\zeta=\zeta')
	$ & 2 \\ \hline $
\bhs{hN_1\ell}
	$&$
	\eps, x'', \tau,  !x, !x', !\tau', !(\zeta=\zeta')
	$ & 2 \\ \hline $
\bhs{hN_1O}
	$&$
	\eps, x'', \tau, \tau', !x, !x', !(\zeta=\zeta')
	$ & 3 \\ \hline $
\bhs{hZ\ell sf}
	$&$
	\eps, \tau, \zeta=\zeta', !x, !x', !x'', !\tau'
	$ & m+1 \\ \hline $
\bhs{hZOsf}
	$&$
	\eps, \tau, \tau', \zeta=\zeta', !x, !x', !x''
	$ & m+2 \\ \hline $
\bhs{hN_1\ell sf}
	$&$
	\eps, x'', \tau, \zeta=\zeta', !x, !x', !\tau'
	$ & m+2 \\ \hline $
\bhs{hN_1Osf}
	$&$
	\eps, x'', \tau, \tau', \zeta=\zeta', !x, !x'
	$ & m+3 \\ \hline $
\bhs{tf\ell}
	$&$
	\tau, \zeta=\zeta', !\tau', !(\eps,!x,!x')
	$ & m \\ \hline $
\bhs{tfO}
	$&$
	\tau, \tau', \zeta=\zeta'. !(y=y'), !(y=y''), !(y'=y'')
	$ & m+1 \\ \hline 
\end{tabular} \\
\end{figure}
%

\subsection{Composition of wedge surgery heat operators} \label{sec:CompositionHeat}

Following the discussion around \eqref{eq:ConvolutionProduct}, we can understand the product of two wedge surgery heat operators using the triple space as
\begin{equation}\label{eq:HeatComposition}
	\cK_{A_1 \circ A_2} \; \beta_H^*(\tfrac12\tau\;d\tau)
	= (\beta_{LR})_*(\beta_{LM}^*(\cK_{A_1}\tau\;d\tau) \cdot \beta_{MR}^*(\cK_{A_2} \tau\;d\tau)).
\end{equation}

Let us fix a section $\mu$ of $\Omega(M)$ 
and rewrite this as
\begin{equation*}
	\kappa_{A_1 \circ A_2} \beta_{H,R}^*(\tfrac{\tau}{2\rho}\mu d\tau)
	= (\beta_{LR})_*\lrpar{ \beta_{LM}^*(\kappa_{A_1} \beta_{H,R}^*(\tfrac\tau\rho\mu d\tau))
	\cdot \beta_{MR}^*(\kappa_{A_2} \beta_{H,R}^*(\tfrac\tau\rho\mu d\tau)) }
\end{equation*}
where $\beta_{RH,L}, \beta_{RH,R}: RHX_{\w,s}\to X_s\times \bbR^+$ are the maps induced by projections on left and right.
Next, consider the corresponding section $\nu = d\eps \; \mu\;  \mu'\;  d\tau$ of $\Omega([0,1]_{\eps}\times M^2 \times \bbR^+_{\tau})$ and $\nu_3 = d\eps \; \mu\;  \mu'\;  \mu'' \; ds \; ds'$ of $\Omega([0,1]_{\eps}\times M^3 \times \bbR^+_{s}\times \bbR^+_{s'}).$
Multiplying  both sides by $\beta_{RH,L}^*(\mu \; d\eps)$ yields
\begin{equation}\label{eq:HeatCompositionDensities}
	\kappa_{A_1 \circ A_2} \beta_{RH,R}^*(\tfrac{\tau}{2\rho})\beta_{RH}^*\nu
	= (\beta_{LR})_*\lrpar{ \beta_{LM}^*(\kappa_{A_1} \beta_{H,R}^*(\tfrac\tau\rho)) 
	\cdot \beta_{MR}^*(\kappa_{A_2} \beta_{RH,R}^*(\tfrac\tau\rho)) \cdot \beta_{RH,3}^*\nu_3 }
\end{equation}
and it is this version of the formula that we will use to determine the asymptotics of $\kappa_{A_1 \circ A_2}.$\\

In our application we will be assuming that  $\kappa_{A_1}$  vanishes to infinite order at $\bhs{tb}$, $\bhs{stf}$ and that
	$\kappa_{A_2}$ vanishes to infinite order at $\bhs{tb}$,
and hence, from \eqref{exm.1} and \eqref{exm.2}, that
\begin{multline*}
	\beta_{LM}^*(\kappa_{A_1})\cdot \beta_{MR}^*(\kappa_{A_2}) \text{ vanishes to infinite order at } \\
	\cT\cB = \{ \bhs{\ell}, \bhs{r}, \bhs{O}, 
	\bhs{hFT}, \bhs{hCT}, \bhs{hST}, \bhs{hFDO}, \bhs{hCD}, \bhs{hSDO}, \\
	\bhs{hZ\ell}, \bhs{hZO},  \bhs{hN_1\ell}, \bhs{hN_1O},
	\bhs{hZOsf}, \bhs{hN_1Osf}, \bhs{tfO} \}.
\end{multline*}
Let $\rho_{\cT\cB}$ denote a product of boundary defining functions, one for each boundary hypersurface in $\cT\cB.$
We will simplify the analysis below by assuming infinite vanishing at $\cT\cB$ whenever convenient.
We need to determine the behavior of the densities under the blow-ups that produce the heat spaces. This is easy to determine as a blow-down map $\gamma:[N;F]\lra N$ satisfies $\gamma^*\Omega(N) = \rho_F^{\dim N-1-\dim F}\Omega([N;F]).$ We have included $\dim N-1-\dim F$ as `blow-up codimension' in the figures above.
Thus, for the reduced heat space we have
\begin{equation*}
\begin{multlined}
	\beta_{RH}^*\Omega([0,1]\times M^2 \times \bbR^+)
	= \rho_{hbf}^2 \rho_{hlf}\rho_{hrf}\rho_{hff}^{h+3} \; \Omega(RHX_{\w,s}) \\
	= \rho_{hmf}\rho_{tb} \rho_{hbf}^3 (\rho_{hlf}\rho_{hrf})^2\rho_{hff}^{h+4} \; \Omega_b(RHX_{\w,s})
	= J(RH)  \; \Omega_b(RHX_{\w,s})
\end{multlined}
\end{equation*}
%
and for the reduced heat triple space
\begin{multline*}
\begin{multlined}
	\rho_{\cT\cB}^{\infty}
	\beta_H^*\Omega([0,1]\times M^3 \times (\bbR^+)^2) 
	= 
	\rho_{\cT\cB}^{\infty}
	\rho_{hT}^3(\rho_{hC}\rho_{hF}\rho_{hS})^2
	\rho_{hN_1}\rho_{hN_2}\rho_{hN_3}
	\rho_{hTT}^{2h+5}(\rho_{hT\ell}\rho_{hTr})^{h+4} \\
	(\rho_{hFD\ell}\rho_{hSDr})^{h+3}
	\rho_{hZ\ell sf}^{m+1}
	\rho_{hN_1\ell sf}^{m+2}
	\rho_{tf\ell}^m \; 
	\Omega(RHX_{\w,s}^3) 
\end{multlined}\\
\begin{multlined}
	= 
	\rho_{\cT\cB}^{\infty}
	\rho_{hZ}
	\rho_{hT}^4(\rho_{hC}\rho_{hF}\rho_{hS})^3
	(\rho_{hN_1}\rho_{hN_2}\rho_{hN_3})^2
	\rho_{hTT}^{2h+6}(\rho_{hT\ell}\rho_{hTr})^{h+5} \\
	(\rho_{hFD\ell}\rho_{hSDr})^{h+4}
	\rho_{hZ\ell sf}^{m+2}
	\rho_{hN_1\ell sf}^{m+3}
	\rho_{tf\ell}^{m+1} \; 
	\Omega_b(RHX_{\w,s}^3)
\end{multlined}\\
=: J_3(RH) \Omega_b(RHX_{\w,s}^3).
\end{multline*}

Next we need to compute the lifts of $\tfrac\tau\rho$ along the lift of the projection onto the right factor of $M.$
For the reduced heat space, using $\sim$ to indicate equality up to a nowhere vanishing function, this is
\begin{equation*}
	\beta_{RH,R}^*(\tfrac\tau\rho) \sim \rho_{tb}\rho_{hbf}^{-1}\rho_{hrf}^{-1}
\end{equation*}
and for the full heat space,
\begin{equation*}
	\beta_{H,R}^*(\tfrac\tau\rho) \sim 
	\rho_{tb}\rho_{stf}\rho_{sf}\rho_{tf}
	\rho_{hbf}^{-1}\rho_{hrf}^{-1}.
\end{equation*}
Thus 
\begin{multline*}
	\rho_{\cT\cB}^{\infty} \cdot
	\beta_{LM}^*( \beta_{H,R}^*(\tfrac\tau\rho)) 
	\cdot \beta_{MR}^*(\beta_{RH,R}^*(\tfrac\tau\rho)) 
	\cdot \beta_{LR}^*(\beta_{RH,R}^*(\tfrac\tau\rho))^{-1}
	\sim 
	\\
	\rho_{\cT\cB}^{\infty}
	\rho_{hT}^{-1}\rho_{hF}^{-1}
	\rho_{hS}^{-1}\rho_{hN_2}^{-1}
	\rho_{hZ\ell sf}\rho_{hN_1 \ell sf} 
	\rho_{tf\ell}.
\end{multline*}

The lift of $J(RH)$ along $\beta_{LR}$ satisfies 
\begin{multline*}
	\rho_{\cT\cB}^{\infty}\cdot
	\beta_{LR}^*(J(RH))  \\=
	\rho_{\cT\cB}^{\infty}
	\rho_{hZ}\rho_{hN_2}\rho_{hZ\ell sf}
	(\rho_{hT}\rho_{hC}\rho_{hT\ell}\rho_{hTr})^3 
	(\rho_{hF}\rho_{hN_3}\rho_{hFD\ell})^2
	(\rho_{hS}\rho_{hN_1}\rho_{hSDr}\rho_{hN_1\ell sf})^2
	\rho_{hTT}^{h+4}.
\end{multline*}
and hence $\beta_{LR}^*(J(RH))^{-1}J_3(RH)$ is
\begin{equation*}
	\rho_{\cT\cB}^{\infty}
	\rho_{hT}
	\rho_{hF}\rho_{hS}
	\rho_{hN_2}
	(\rho_{TT}\rho_{hT\ell}\rho_{hTr}\rho_{hFD\ell}\rho_{hSDr})^{h+2}
	(\rho_{hZ\ell sf}\rho_{hN_1\ell sf}\rho_{tf\ell})^{m+1}
\end{equation*}

Thus all together there are $\nu_b \in \Omega_b(RHX_{\w,s})$ and $\nu_{3,b}\in\Omega_b(RHX_{\w,s}^3)$ such that
\begin{multline}\label{eq:PreFinalComp}
	\rho_{\cT\cB}^{\infty}\cdot
	\kappa_{A_1 \circ A_2} \nu_b  
	= (\beta_{LR})_*\Big(\beta_{LM}^*(\kappa_{A_1})
	\cdot \beta_{MR}^*(\kappa_{A_2}) \cdot \\
	\rho_{\cT\cB}^{\infty}
	(\rho_{TT}\rho_{hT\ell}\rho_{hTr}\rho_{hFD\ell}\rho_{hSDr})^{h+2}
	(\rho_{hZ\ell sf}\rho_{hN_1\ell sf}\rho_{tf\ell})^{m+2}
	\; \nu_{3,b}\Big).
\end{multline}

\begin{theorem}[Composition result] \label{thm:Composition}
Consider first the polyhomogeneous case.
Let
\begin{equation*}
\begin{gathered}
	\cK_{A_1} \in \cA^{\cE^1}_{\phg}(HX_{\w,s}; \Hom(E)\otimes \beta_{H,R}^*(\tfrac1\rho\Omega(M))) \\
	\cK_{A_2} \in \cA^{\cE^2}_{\phg}(RHX_{\w,s}; \Hom(E)\otimes \beta_{RH,R}^*(\tfrac1\rho\Omega(M))) 
\end{gathered}
\end{equation*}
where $E^1_{stf}=E^1_{tb} = E^2_{tb} = \emptyset$ and
$\Re E^1_{tf} + m + 2>0.$
Then 
\begin{equation*}
	\cK_{A_1} \circ \cK_{A_2}
	\in \cA^{\cE^3}_{\phg}(RHX_{\w,s}; \Hom(E)\otimes \beta_{RH,R}^*(\tfrac1\rho\Omega(M))) 
\end{equation*}
with index sets
\begin{equation*}
\begin{aligned}
	E^3_{tb} &= \emptyset, \\
	E^3_{hmf} & =
	(E^1_{hmf}+E^2_{hmf}) \bar\cup
	(E^1_{hrf} + E^2_{hlf}) \bar\cup
	(E^1_{sf} + E^2_{hmf} +m +2), \\
	E^3_{hbf} &=
	(E^1_{hbf} + E^2_{hbf}) \bar\cup
	(E^1_{hlf} + E^2_{hrf}) \bar\cup
	(E^1_{hff} + E^2_{hbf} + h+2) \bar\cup
	(E^1_{hbf} + E^2_{hff} + h+2), \\
	E^3_{hlf} &=
	(E^1_{hbf} + E^2_{hlf}) \bar\cup
	(E^1_{hlf} + E^2_{hmf}) \bar\cup
	(E^1_{hff} + E^2_{hlf} + h+2), \\
	E^3_{hrf} &=
	(E^1_{hrf} + E^2_{hbf}) \bar\cup
	(E^1_{hmf} + E^2_{hrf}) \bar\cup
	(E^1_{hrf} + E^2_{hff} + h+2) \bar\cup
	(E^1_{sf} + E^2_{hrf} + m+2), \\
	E^3_{hff} &=  E^1_{hff} + E^2_{hff} + h+2
\end{aligned}
\end{equation*}

More generally, this establishes the corresponding result for `hybrid' index sets (phg+bounded).
\end{theorem}

\begin{proof}
Let us abbreviate \eqref{eq:PreFinalComp} as
\begin{equation*}
	\kappa_{A_1 \circ A_2} \nu_b  
	= (\beta_{LR})_*(
	\wt \kappa
	\; \nu_{3,b}).
\end{equation*}
The mapping properties established above show that $\wt\kappa$ vanishes to infinite order at $\cT\cB$ and 
has the following index sets/bounds at the remaining boundary hypersurfaces of $RHX_{\w,s},$
\begin{equation*}
\begin{aligned}
	\wt E_{hZ} &= E^1_{hmf}+E^2_{hmf}, \\
	\wt E_{hT} &= E^1_{hbf} + E^2_{hbf}, \\
	\wt E_{hF} &= E^1_{hbf} + E^2_{hlf}, \\
	\wt E_{hC} &= E^1_{hlf} + E^2_{hrf}, \\
	\wt E_{hS} &= E^1_{hrf} + E^2_{hbf},\\
	\wt E_{hN_1} &= E^1_{hmf} + E^2_{hrf}, \\
	\wt E_{hN_2} &= E^1_{hrf} + E^2_{hlf},\\
	\wt E_{hN_3} &= E^1_{hlf} + E^2_{hmf}, \\
	\wt E_{hTT} &= E^1_{hff} + E^2_{hff} + h+2, \\
	\wt E_{hT\ell} &= E^1_{hff} + E^2_{hbf} + h+2, \\
	\wt E_{hTr} &= E^1_{hbf} + E^2_{hff} + h+2, \\
	\wt E_{hFD\ell} &= E^1_{hff} + E^2_{hlf} + h+2, \\
	\wt E_{hSDr} &= E^1_{hrf} + E^2_{hff} + h+2, \\
	\wt E_{hZ\ell sf} &= E^1_{sf} + E^2_{hmf} +m +2, \\
	\wt E_{hN_1\ell sf} &= E^1_{sf} + E^2_{hrf} + m+2, \\
	\wt E_{tf\ell} &= E^1_{tf} + m+2.
\end{aligned}
\end{equation*}
Thus as long as $\wt E_{tf\ell} = E^1_{tf} + m + 2$ is positive we can push forward.  To obtain the result, the idea is then to apply the pushforward theorem \cite[Theorem~2.3]{hmm}.  However, since $\beta_{LR}$ fails to be a $b$-fibration, this requires further justification.  What saves the situation is that all faces where $\beta_{LR}$ fails to be a $b$-fibration, namely where it fails to be $b$-normal, are faces contained in $\cT\cB$.  Since at these faces, $\wt \kappa$ vanishes rapidly, these faces can be safely ignored while applying the pushforward theorem.  On the other hand, among all the faces created to obtain $R'HX^3_{\w,s}$ from $RHX_{\w,s}^3$, only $\bhs{tf\ell}$, $\bhs{hZ\ell sf}$ and $\bhs{h N_1\ell sf}$ do not belong to $\cT\cB$.  For these three faces, it suffices to apply \cite[Lemma~2.7]{hmm} to see that $\beta_{LR}$ is a $b$-fibration away from the boundary hypersurfaces of $\cT\cB$.

\end{proof}

\section{Uniform construction of the heat kernel under a wedge surgery} \label{hk.0}

We are now prepared to solve the heat equation, in the strong sense of constructing the Schwartz kernel of the fundamental solution. Our construction follows the geometric microlocal analysis approach of Melrose \cite{MelroseAPS} in that we describe the various boundary hypersurfaces of the heat space, solve model heat problems there (\S\ref{sec:ModelHeatProbs}), and then put their integral kernels together to solve the heat equation (Theorem \ref{uhk.1}). Part of the construction requires summing a Volterra series, for which we have established a composition result in \S\ref{sec:CompositionHeat}.

Our construction shows that as $\eps \to 0$ the heat kernel converges, away from $H,$ to the Hodge Laplacian heat kernel of a space with a wedge singularity whose indicial roots are not constant, generalizing that of \cite{Mazzeo-Vertman}.

We will construct the heat kernel as a section of the density bundle $\Omega_{b,heat}$ on $HX_{\w,s}$ which is spanned over $C^{\infty}(HX_{\w,s})$ by the density
\begin{equation}
\nu_{b,heat}:=\beta_{H,R}^*(\nu_{b,s}),
\end{equation}
where $\nu_{b,s}$ is a nonvanishing section of the $b$-surgery density bundle $\Omega_{b,s}(X_s)$, equal to $\rho^{-1}\,dx\,dy\,dz$ in local coordinates.

\subsection{Model heat problems} \label{sec:ModelHeatProbs} $ $\\

{\bf The model problem at $\bhs{tf}.$} Let us start by examining the geometry of $\bhs{tf}.$ To obtain this boundary hypersurface, we blow up the intersection of $\bhs{tb}$ with the interior lift of $(\diag_M\times[0,1]_{\eps}\times \bbR^+_{\tau})$. This intersection can be identified with the single surgery space - its intersection with $\bhs{sf}$ is $\bhs{\sm}$ and its intersection with $\bhs{hff}$ is $\bhs{\bs}$. The face $\bhs{tf}$ can thus be identified with the radial compactification of the normal bundle to the lifted diagonal in the double edge surgery space, i.e., the edge surgery tangent bundle,
\begin{equation*}
	\xymatrix{
	\bhs{tf} \ar[rr]^{\cong} \ar[rd] & & \overline{{}^{e,s}TX}_s \ar[ld] \\
	& X_s. &}
\end{equation*}
For example, in the coordinates \eqref{coords:step5part0}, the projection down to $X_s$ is given by $(x,y,z,\eps)$. Over each point we have a copy of $\bbR^+_{\tau}\times\bbR^{m}_{\Theta}$.

We will now construct the fundamental solution of the heat equation at $\bhs{tf}$.  Away from $\bhs{hff}$ and $\bhs{sf}$, this is as  in compact case, \cf \cite[Chapter 7]{MelroseAPS} and \cite[\S7]{ARS1}.   Thus, let us focus on the most problematic region, that is, near the triple intersection $\bhs{tf}\cap\bhs{hff}\cap\bhs{sf}$.  There, we can use the coordinates \eqref{coords:step5part2}.  Writing the heat kernel in these coordinates gives
$$
    K_A= \widetilde{K}_A  d\Theta_x d\Theta_y d\overline{\Theta}_z.
$$
Set also $\widetilde{K}_{tf}:= \left.  \widetilde{K}_A \right|_{\bhs{tf}}$ assuming this restriction makes sense.  For $f$ supported in the coordinate chart of $(\rho,\theta,y,z)$ in $X_s$, the action of $A$ on $f$ is then given by 
\begin{equation}
  (Af)(\rho,\theta,y,z,\sigma)= \int_{\bbR^m}  \widetilde{K}_A f(\rho',\theta', y+\rho\eta\sigma\Theta_y, z+\sigma\eta  \overline{\Theta}_z) d\Theta_x d\Theta_y d\overline{\Theta}_z
  \label{hk.1}\end{equation}    
where 
$$
\rho'= \sqrt{(x')^2+\eps^2}= \rho\sqrt{(\sin\theta+\eta\sigma\Theta_x)^2+\eta^2}
$$
and
$$ 
 \theta'=\arctan \left( \frac{x'}{\eps} \right)= \arctan\left( \frac{\sin\theta+\eta\sigma \Theta_x}{\eta} \right).
$$
Hence, the restriction of \eqref{hk.1} to $\bhs{tf}$, that is, the restriction to $\sigma=0$, is given by 
\begin{equation}
 f(\rho,\theta,y,z)\int_{\bbR^m} \widetilde{K}_{tf}(\rho,\theta,y,z,\Theta_x,\Theta_y, \overline{\Theta}_z) d\Theta_x d\Theta_{y} d\overline{\Theta}_z.
\label{hk.2}\end{equation}

We now want to compute the action of $t(\pa_t + D^2_{\w,\eps})$ when composed on the left of $A$.  Considering first the action of $t\pa_t= \frac12 \sigma\pa_{\sigma}$ on $Af$, we can integrate by parts as in the compact case to obtain that 
\begin{equation}
  \left.  K_{t\pa_t \circ A} \right|_{\bhs{tf}}= -\frac12 (m+\cR)\widetilde{K}_{tf} d\Theta_x d\Theta_y d\overline{\Theta}_z,
\label{hk.3}\end{equation}
where $\cR= \Theta_x\cdot \pa_{\Theta_x}+ \Theta_y\cdot \pa_{\Theta_y}+ \overline{\Theta}_z\cdot \pa_{\overline{\Theta}_z}$ is the radial vector field in the coordinates 
$\Theta_x, \Theta_y, \overline{\Theta}_z$.  To compute the action of $t D^2_{\w,\eps}$, consider first a $\w,\eps$ vector field
\begin{equation}\label{eq:WVf}
	V = a \pa_x + b^i\pa_{y_i} + c^j\rho^{-1}\pa_{z_j}
\end{equation}
with $a$, $b^i$ and  $c^j$ smooth sections of $\End(E)$ on $X_s$.  We want to determine how $\tau V= \rho\eta\sigma V$ acts on $A$ from the left.  Now, $\rho V$ is an $e,s$ vector field, so keeping in mind that $\eta$ and $\overline{\Theta}_z$ depend on $x$ through $\rho$, we compute integrating by parts that 
\begin{equation}
   \tau VAf(\rho,\theta,y,z,\tau)= \int_{\bbR^m} (V'\widetilde{K}_A) f(\rho',\theta', y+\rho\eta\sigma\Theta_y, z+\sigma\eta\overline{\Theta}_z)d\Theta_x d\Theta_y d\overline{\Theta}_z
\label{hk.4}\end{equation}
with
\begin{equation}
V'= a(\rho\eta\sigma\pa_x -\pa_{\Theta_x}+ \eta\sigma\sin\theta(\overline{\Theta}_z\cdot\pa_{\overline{\Theta}_z}+v)) - b^i \pa_{\Theta_{y}^i}- c^j \pa_{\overline{\Theta}_{z}^j}.
\label{hk.5}\end{equation}
Hence, restricting \eqref{hk.4} to $\bhs{tf}$ gives that 
\begin{equation}
  \left. \tau VAf(\rho,\theta,y,z,\tau)\right|_{\sigma=0}= f(\rho,\theta,y,z) \int_{\bbR^m} (\left.V'\right|_{\sigma=0}\widetilde{K}_{tf}) d\Theta_x d\Theta_y d\overline{\Theta}_z
\label{hk.6}\end{equation}
with 
\begin{equation}
\left.  V'\right|_{\sigma=0}=  -a\pa_{\Theta_x} - b^i \pa_{\Theta_{y}^i}-  c^j \pa_{\overline{\Theta}_{z}^j}.
\label{hk.7}\end{equation}
Using these results, we thus see that the restriction of $t(\pa_t + D^2_{\w,\eps})K_A$ to $\bhs{tf}$ is given by 
\begin{equation}
     \left[ \left( -\frac12(m+\cR) + \Delta_{{}^{e,s}TX_s}\right) \widetilde{K}_{tf}\right]  d\Theta_x d\Theta_y d\overline{\Theta}_z,
\label{hk.8}\end{equation}
where $\Delta_{{}^{e,s}TX_s}$ is the Euclidian Laplacian in the fibers of  ${}^{e,s}TX_s$ induced by the principal symbol ${}^{e,s}\sigma_2(\rho^2D^2_{\w,\eps})$.  To solve the heat equation to first order on $\bhs{tf}$, we need \eqref{hk.7} to be zero and \eqref{hk.2} to be equal to $f$. As in the compact case, this is readily solved using the Fourier transform and gives
\begin{equation}
  L_{tf}:= \widetilde{K}_{tf}d\Theta_x d\Theta_y d\overline{\Theta}_z  = \frac{1}{(4\pi)^{\frac{m}2}}\exp \left(  -\frac{|\cdot|^2_{e,s}}4\right)\mu,
\label{hk.9}\end{equation}
where   $|\cdot |_{e,s}$ and the vertical density $\mu$ on the fibers of ${}^{e,s}TX_s$  are specified by the principal symbol ${}^{e,s}\sigma_2(\rho^2D^2_{\w,\eps})$.  Notice that writing the solution \eqref{hk.9} in terms of the density $\frac{dx'dy'dz'}{\rho'}$, we see that the heat kernel most have top order term of order $\sigma^{-m}$ at $\bhs{tf}$, $\eta^{-m}$ at $\bhs{sf}$ and $\rho^{-h}$ at $\bhs{hff}$.


\medskip


\textbf{The model problem at $\bhs{sf}$.}
This is another Euclidean face, corresponding to the face sc in the scattering heat space and to the zero-time diagonal in the wedge heat space. Unsurprisingly, the model here will be another Euclidean heat kernel.  Consider the coordinates \eqref{coords:step4smallsigma}. In these coordinates,
\begin{equation*}
\nu_{b,heat}=\frac{\eta^m\rho^h}{\sqrt{(\sin\theta+\eta\Xi_x)^2+\eta^2}}\,d\Xi_x\,d\Xi_y\,d\bar\Xi_z.
\end{equation*}
Clearly, the lift of $t\pa_t$ is simply $\frac12 \sigma\pa_{\sigma}$.  On the other hand, using integration by parts, we compute that 
\begin{equation}
   \tau VAf(\rho,\theta,y,z,\tau)= \int_{\bbR^m} (V'\widetilde{K}_A) f(\rho',\theta', y+\rho\eta\Xi_y, z+\eta\overline{\Xi}_z)d\Xi_x d\Xi_y d\overline{\Xi}_z
\label{hk.4b}\end{equation}
with this time
\begin{equation}
V'= -\sigma(a\pa_{\Xi_x} + b^i \pa_{\Xi_{y}^i} +c^j \pa_{\overline{\Xi}_{z}^j}).
\label{hk.5}\end{equation}

Hence, since $\eta=0$ at $\bhs{sf}$, we compute that 
\begin{equation}
N_{sf}(t\pa_t+t D_{\dR}^2)=\tfrac 12\sigma\pa_{\sigma}+\sigma^2\Delta_{sf},
\end{equation}
where $\Delta_{sf}$ is the Euclidean Laplacian on the interior of the fibers of $\bhs{sf}\to \bhs{sm}$ specified by the principal symbol ${}^{e,s}\sigma_2(\rho^2 D^2_{\w,\eps})$ restricted to $\left.{}^{e,s}T^*X_s\right|_{\bhs{sm}}$.  Notice in particular that $\Delta_{sf}$ does not depend on $\sigma$, so in each fiber this just the heat operator on a Euclidean space. Its heat kernel  is standard and given by
\begin{multline}\label{leadingmodel:sf}
L_{sf}=\frac{1}{(4\pi)^{m/2}\sigma^m}\exp\big(-\frac{\Xi_x^2+\Xi_y^2+\bar\Xi_z^2}{4\sigma^2}\big) d\Xi_x\,d\Xi_y\,d\bar\Xi_z\\ =\frac{\sqrt{(1+\eta\Xi_x)^2+\eta^2}}{(4\pi)^{m/2}\sigma^m\eta^mx^h}\exp\big(-\frac{\Xi_x^2+\Xi_y^2+\bar\Xi_z^2}{4\sigma^2}\big) \nu_{b,heat}.
\end{multline}
In terms of the coordinates \eqref{coords:step5part2}, notice that this correspond exactly to \eqref{hk.9}.  In fact, in these coordinates, the model heat kernel on $\bhs{sf}$ does not depend $\sigma$.  Moreover, it clearly agrees with the model at $\bhs{tf}$.


\medskip


\medskip {\bf The model problem at $\bhs{hff}.$} Next we consider $\bhs{hff}$. In the coordinates \eqref{coords:interiorhff}, we have
\begin{equation}
\nu_{b,heat}=\eps^h\langle X'\rangle^{-1}\,dX'\,du\,dz'.
\end{equation}
The vector fields lift as follows,
\begin{equation}
	\beta_{H,L}^*(\tau\pa_x) = \sigma \pa_{X}, \quad
	\beta_{H,L}^*(\tau \pa_y) = -\sigma\pa_{u}+\sigma\eps\pa_y, \quad
	\beta_{H,L}^*(\tau\tfrac1\rho \pa_z) = \tfrac{\sigma}{\langle X\rangle}\pa_z.
\label{vf.1}\end{equation}
Hence
\begin{multline*}
	N_{hff}(t\pa_t + t D_{\dR}^2) =
	\beta_{H,L}^*(t\pa_t + tD_{\dR}^2)\rest{\eps=0} \\
	= \tfrac12\sigma\pa_{\sigma} + \sigma^2
	\lrpar{ \tfrac1{\ang X^2} \Delta^{H/Y} + \Delta^{\bbR^h_u} - \pa_{X}^2 + \tfrac X{\ang X^2} \pa_X}
	\begin{pmatrix} \Id & 0 \\ 0 & \Id \end{pmatrix} \\
	+
	\frac{\sigma^2}{\ang X^2}
	\begin{pmatrix}
	\tfrac {X^2}{\ang X^2}\lrpar{(\wt N+1)^2-(\tfrac32)^2}-(\wt N-\tfrac12)
	& 
	-2\tfrac X{\ang X} d^Z \\
	-2\tfrac X{\ang X} \delta^Z 
	&
	\tfrac {X^2}{\ang X^2}\lrpar{(\wt N-1)^2-(\tfrac32)^2}+(\wt N+\tfrac12) 
	\end{pmatrix}.
\end{multline*}

To interpret this heat operator let us denote, for each $y \in Y,$
\begin{equation}
	\Sigma(Z_y):=[-\infty,\infty]_{X}\times Z_y; \quad g_{\Sigma(Z_y)}=dX^2+(1+X^2)g_{Z_y}.
\end{equation}
Note that $\Sigma(Z_y)$ is a scattering manifold and that $N_{hff}(t\pa_t+tD_{\dR}^2)$ restricts to the fiber over $y \in Y$ to be the heat operator on the lift of $F\rest{Z_y}$ to $\Sigma(Z_y) \times \bbR^h_u$ with the product metric $g_{\Sigma(Z_y)} + du^2$ and the lift of the bundle metric. This suggests the following model for the heat kernel:
\begin{multline}\label{leadingmodel:hff}
	L_{hff}:=H^{\Sigma(Z_y)}(\sigma,X,X',z,z')
	\frac{1}{(4\pi)^{h/2}}\sigma^{-h} \exp\lrpar{-\frac{|u|_{T_yY}^2}{4\sigma^2} } \langle X'\rangle^{v}\, dX'\,du\, dz'\\ 
	=\eps^{-h}H^{\Sigma(Z_y)}(\sigma,X,X',z,z')
	\frac{1}{(4\pi)^{h/2}}\sigma^{-h}\langle X'\rangle^{v+1}\exp\lrpar{ -\frac{|u|_{T_yY}^2}{4\sigma^2}} \nu_{b,heat},
\end{multline}
where $H^{\Sigma(Z_y)}(\sigma,X,X',z,z')$ is the corresponding heat kernel on $\Sigma(Z_y).$ We next describe the asymptotics of $L_{hff}$ near $\bhs{hff}$ and see that it matches the models at $\bhs{tf}$ and $\bhs{sf}.$\\

\textbf{Scattering heat kernel.} Consider $H^{\Sigma(Z_y)}(\sigma,X,X',z,z'),$ the heat kernel on the scattering manifold $\Sigma(Z_y),$ which has been described in \cite{Albin:Heat, Sher:ACHeat, Guillarmou-Sher}. Note that this is sometimes styled `asymptotically conic' in reference to the `big end' of a cone. 
Recall that $v = \dim Z.$  As in \S~\ref{ur.0}, we will suppose that Assumption~\ref{ur.1} holds.  This assumption is not as important as for the uniform construction of the resolvent, but we will make it anyway since it will lead to many helpful simplifications.  

\begin{theorem}[\cite{GH1,Sher:ACHeat,Guillarmou-Sher}]
 If Assumption~\ref{ur.1} holds, then the heat kernel $H^{\Sigma(Z_y)}$ is polyhomogeneous conormal on the heat space $H\Sigma(Z_y)_{\sc}$ of \eqref{sch.1} with index family $\cR$ such that
\begin{equation}
\inf \cR_{\zf}\ge 2, \quad \inf \cR_{\bfo}\ge 0, \quad \inf \cR_{\lb_0}\ge\nu_0+1, \quad \inf \cR_{\rb_0}\ge\nu_0+1, \quad \inf \cR_{\sc}\ge -v-1
\label{hkp.1b}\end{equation}
when we use right $b$-densities.
\label{hkp.1}\end{theorem}
\begin{proof}
It suffices to combine Theorem~\ref{no.12} with \cite[Theorem~8]{Sher:ACHeat}.  
\end{proof}

An important consequence of this theorem is  that  at $\bfo$ and $\sc$, the heat kernel $H^{\Sigma(Z_y)}$ is modelled on the heat kernel $H^{C(Z_y)}$ of the Hodge Laplacian $\Delta_{C(Z_y)}$ on the metric cone $C(Z_y)$ with cross-section $Z_y$.  Here is a more precise statement.

\begin{corollary}\label{hkp.2} The heat kernels $H^{\Sigma(Z_y)}$ and $H^{C(Z_y)}$ have the same leading order terms at the faces 
$\sc$ and $\bfo$. 
\end{corollary}

\begin{proof} The statement for 
sc follows immediately from \cite{Albin:Heat}.  At $\bfo$, notice first that we know from the proof of Theorem~\ref{no.12} that the resolvents $(\Delta_{\Sigma(Z_y)}+k^2)^{-1}$ and 
$(\Delta_{C(Z_y)}+k^2)^{-1}$ have the same leading order at $\bfo$.  We can then use \cite[Theorem 8]{Sher:ACHeat} to transition from resolvent to heat kernel. Consider the tubular neighborhood of infinity $[0,\delta)\times Z_y$ in $\Sigma(Z_y)$ where the scattering metric is $g=s^{-4}ds^2+s^{-2}h(s)$, with $s=1/r$. In this region, both $(\Delta_{\Sigma(Z_y)}^q-k^2)^{-1}$ and $(\Delta_{C(Z_y)}^q-k^2)^{-1}$ are well-defined, so we can consider 
\[\chi((\Delta_{\Sigma(Z_y)}^q-k^2)^{-1}-(\Delta_{C(Z_y)}^q-k^2)^{-1})\chi.\]
Here $\chi$ is a cutoff function which is supported on the tubular neighborhood and identically equal to one in a neighborhood of infinity. We apply the proof of \cite[Theorem 8]{Sher:ACHeat} to this difference of resolvents to obtain the corresponding difference of heat kernels. In particular, we see that the corresponding difference of heat kernels has leading order strictly greater than $0$. Since the exact conic heat kernel has leading term of order $0$ at bf$_0$, the proof is complete. \end{proof}

Finally, let us point out that the construction of the heat kernel is robust in its dependence on $y \in Y$ in that varying $y$ will produce a smooth family of integral kernels as long Assumption~\ref{ur.1} is maintained. Note that while each $H^{\Sigma(Z_y)}$ is polyhomogeneous, the resulting family of kernels will define an element of a calculus with bounds instead.  
\\

\textbf{Asymptotics of the model at $\bhs{hff}$.} We consider each coordinate system in which the model at $\bhs{hff}$ has nontrivial behavior, that is,  fails to vanish to infinite order at the intersection of $\bhs{hff}$ with another boundary hypersurface.  As explained at the end of $\textbf{Step 2'}$ of the proof of Theorem~\ref{ur.8} in \S~\ref{wshs.1}, the face $\bhs{hff}$ can be thought as a somewhat blow-down version of the scattering heat space of \cite{Sher:ACHeat}, since the infinite-order vanishing of $H^{\Sigma(Z_y)}$ at tb, lf, rf and bf allows us to work in this simplified space.

First, we can check directly that \eqref{leadingmodel:hff} agrees with the models at $\bhs{tf}$ and $\bhs{sf}$ described in \eqref{hk.9} and \eqref{leadingmodel:sf}.  On the other hand, Proposition~\ref{hkp.2} shows that $L_{hff}$ has the following asymptotic behavior at the intersection $\bhs{hff}\cap\bhs{hmf}$,  
\begin{multline}
	L_{hff} \sim 
	\eps^{-h}H^{C(Z_y)}(\sigma,X,X',z,z')
	\frac{1}{(4\pi)^{h/2}}\sigma^{-h}\langle X'\rangle^{v+1}
	\exp\big(-\frac{|u|_{T_yY}^2}{4\sigma^2})\nu_{b,heat} \\
	=\eps^{-h-v-1}H^{C(Z_y)}(\sigma,X,X',z,z')\frac{1}{(4\pi)^{h/2}}\sigma^{-h}(\rho')^{v+1}
	\exp\big(-\frac{|u|_{T_yY}^2}{4\sigma^2})\nu_{b,heat}.
\end{multline}
Using the parabolic scaling of the heat kernel on a cone,
\begin{equation*}
	H^{C(Z_y)}(\sigma,X,X',z,z')=\eps^{-v-1}H^{C(Z_y)}(\tau,x,x',z,z'),
\end{equation*}
as well as the fact that $\rho'=x'$ on $\bhs{hmf}$ near $\bhs{hff}\cap\bhs{hmf}$, this can be rewritten
\begin{multline}
	H^{C(Z_y)}(\tau,x,x',z,z')\frac{1}{(4\pi)^{h/2}}\tau^{-h}
		\exp(-\frac{(y'-y)^2}{4\tau^2})(x')^v\ dx'\,dy'\,dz'\\
	= \eps^{-h}\sigma^{-h}H^{C(Z_y)}(\tau,x,x',z,z')\frac{1}{(4\pi)^{h/2}}
		\exp(-\frac{u^2}{4\sigma^2})(x')^{v+1}\nu_{b,heat}|_{\bhs{hmf}},
\end{multline}
which, as described in \cite{Mazzeo-Vertman} is exactly the model that needs to be put on $\bhs{hmf}\cap\bhs{hff}$ to start the construction of the wedge heat kernel on $\bhs{hmf}$.

The compatibility of $L_{tf},$ $L_{sf}$ and $L_{hff}$ at the intersections of the boundary hypersurfaces on which they are defined guarantees the existence of an integral kernel, $G_0,$ that coincides with each of them. The analysis above shows that we can take $G_0$ to be an element of
\begin{equation}
	G_0 \in \sB^{\cG/\df g}_{\phg}\sA^{-m-1}_-(HX_{\w,s};
		\Hom(E) \otimes \beta_{H,R}^*(\tfrac1\rho\Omega(M)))
\label{hkp.2b}\end{equation}
where
$\df g = \infty$ at the boundary hypersurfaces 
$\bhs{tb}, 	\bhs{hff}, \bhs{stf}, \bhs{sf}, \bhs{tf},$ where we have
\begin{equation*}
\begin{gathered}
	\cG(\bhs{tb}) = \emptyset, \quad
	\cG(\bhs{hff}) = -h, \quad
	\cG(\bhs{stf}) = \emptyset, \\
	\cG(\bhs{sf}) = -m, \quad 
	\cG(\bhs{tf}) = -m,
\end{gathered}
\end{equation*}
while at $\bhs{hmf}$  and $\bhs{hbf}$,
$$
    \cG(\bhs{hmf})=0, \quad \df g(\bhs{hmf})>0,  \quad \cG(\bhs{hbf})=2, \quad \df g(\bhs{hbf})>2,
$$
and  at the remaining boundary hypersurfaces $\bhs{hlf}, \bhs{hrf}$, we have that $\cG = \emptyset,$ and
\begin{equation*}
	\df g(\bhs{hlf})=
	\df g(\bhs{hrf})= \nu_{\min} :=\nu_0+1 >1.
	\end{equation*}

By construction, the operator $G_0$ is such that 
\begin{equation*}
	(t\pa_t + tD^2_{\w,\eps})G_0 \in 
	\rho_{hbf}^{-2}\rho_{hlf}^{-2}\rho_{hff}\rho_{sf}\rho_{tf}\sB^{\cH/\df h}_{\phg}\sA^{-m-1}_-
		(HX_{\w,s};\Hom(E) \otimes \beta_{H,R}^*(\tfrac1\rho\Omega(M))),
\end{equation*}
where $\cH,$ $\df h$ are an index family and a weight function for $HX_{\w,s}$ equal to $\cG,$ $\df g$ except for 
\begin{equation}
   \cH(\bhs{hbf})=\emptyset,  \quad \df h(\bhs{hbf})>2.
\label{hkp.3}\end{equation}
Indeed, from the singularities of the coefficients in the lift of $(t\pa_t + tD^2_{\w,\eps})$, one would expect $\cH,\df h$ to be the same as $\cG,\df g$, even at $\bhs{hbf}$.  However, we know that  the leading term at $\bhs{hff}\cap\bhs{hbf}$  is annihilated  by the leading term of the lifted operator at $\bhs{hbf}$, that is, by $\ang{X}^{-1} \Delta_v \ang{X}^{-1}$.  Since the leading order of the lifted operator at $\bhs{hbf}$ does not depend on time, we can just choose $G_0$ so that $\left. G_0\right|_{\bhs{hbf}}$ is annihilated by 
$\ang{X}^{-1} \Delta_v \ang{X}^{-1}$, yielding the extra decay specified by  \eqref{hkp.3}.  Notice that a similar argument would also yield extra decay at $\bhs{hlf}$ if we knew that there was at least a partial polyhomogeneous expansion there.

\subsection{Improved parametrix}

The model problems for the subsequent terms in the expansion at $\bhs{sf}$ and $\bhs{tf}$ are easily solved using Fourier transform as in \cite[(7.58)-(7.63)]{MelroseAPS}, producing an improved parametrix $G_1$ satisfying
\begin{equation*}
\begin{gathered}
	G_1 \in 
	\sB^{\cG/\df g}_{\phg}\sA^{-m-1}_-
		(HX_{\w,s};\Hom(E) \otimes \beta_{H,R}^*(\tfrac1\rho\Omega(M))) \\
	\beta_L^*(t\pa_t + tD^2_{\w,\eps})G_1 \in 
	\rho_{hbf}^{-2}\rho_{hlf}^{-2}\rho_{hff}\rho_{sf}^{\infty}\rho_{tf}^{\infty}
	\sB^{\cH/\df h}_{\phg}\sA^{-m-1}_-
		(HX_{\w,s};\Hom(E) \otimes \beta_{H,R}^*(\tfrac1\rho\Omega(M))).
\end{gathered}
\end{equation*}
To improve the parametrix at $\bhs{hff},$ we can instead use the following lemma.

\begin{lemma}\label{lem:ImproveTff}
For each integer $k\ge -1$ and 
\begin{equation*}
	f \in \rho_{hff}^k\rho_{hbf}^{-2}\rho_{hlf}^{-2}\rho_{sf}^{\infty}\rho_{tf}^{\infty}
	\sB^{(\cH/\df h)}_{\phg}\sA^{-m-1}_-(RHX_{\w,s}; 
		\Hom(E) \otimes \beta_{H,R}^*(\tfrac1\rho\Omega(M))),
\end{equation*}
there exists 
\begin{equation*}
	u \in \rho_{hff}^{k+2}\rho_{sf}^{\infty}\rho_{tf}^{\infty}
	\sB^{(\cH/\df h)}_{\phg}\sA^{-m-1}_-
		(RHX_{\w,s}; \Hom(E) \otimes \beta_{H,R}^*(\tfrac1\rho\Omega(M)) ).
\end{equation*}
such that 
\begin{equation*}
	\left[f-(\pa_t + D^2_{\w,\eps})u\right]  \in \rho_{hff}^{k+1}\rho_{hbf}^{-2}\rho_{hlf}^{-2}\rho_{sf}^{\infty}\rho_{tf}^{\infty}
	\sB^{(\cH/\df h)}_{\phg}\sA^{-m-1}_-(RHX_{\w,s}; 
		\Hom(E) \otimes \beta_{H,R}^*(\tfrac1\rho\Omega(M))).
\end{equation*}
\end{lemma}

\begin{proof}
It suffices to take $u= G_1\circ f$ and use  Theorem~\ref{thm:Composition} to see that $u$ is the space given above.  Technically, we might have to take $\cH(\bhs{hmf})$ slightly bigger, but by taking $\df h (\bhs{hmf})>0$ smaller if needed, we may assume that $\cH(\bhs{hmf})$ remains unchanged.

\end{proof}

\begin{remark}
In the proof of Lemma \ref{lem:ImproveTff} we have used the composition result Theorem \ref{thm:Composition} which improves those in \cite{Albin:Heat, Sher:ACHeat} in that it includes the behavior at temporal infinity.
\end{remark}

Using this lemma to solve away successive terms at $\bhs{hff}$ we can construct, for any $\ell\geq 2,$ an improved parametrix
\begin{equation*}
\begin{gathered}
	G_\ell \in 
	\sB^{\cG/\df g}_{\phg}\sA^{-m-1}_-
		(HX_{\w,s};\Hom(E) \otimes \beta_{H,R}^*(\tfrac1\rho\Omega(M))),\\
	\beta_L^*(t\pa_t + tD^2_{\w,\eps})G_\ell \in 
	\rho_{hbf}^{-2}\rho_{hlf}^{-2}\rho_{hff}^{\ell}\rho_{sf}^{\infty}\rho_{tf}^{\infty}
	\sB^{\cH/\df h}_{\phg}\sA^{-m-1}_-
		(HX_{\w,s};\Hom(E) \otimes \beta_{H,R}^*(\tfrac1\rho\Omega(M))).
\end{gathered}
\end{equation*}
Asymptotically summing successive differences we can remove the error at $\bhs{hff}$ altogether,
\begin{equation*}
\begin{gathered}
	G_\infty \in 
	\sB^{\cH/\df h}_{\phg}\sA^{-m-1}_-
		(HX_{\w,s};\Hom(E) \otimes \beta_{H,R}^*(\tfrac1\rho\Omega(M))), \\
	\beta_L^*(t\pa_t + tD^2_{\w,\eps})G_\infty \in 
	\rho_{hbf}^{-2}\rho_{hlf}^{-2}\rho_{hff}^{\infty}\rho_{sf}^{\infty}\rho_{tf}^{\infty}
	\sB^{\cH/\df h}_{\phg}\sA^{-m-1}_-
		(HX_{\w,s};\Hom(E) \otimes \beta_{H,R}^*(\tfrac1\rho\Omega(M))).
\end{gathered}
\end{equation*}

Note that the error now vanishes to infinite order at all boundary hypersurfaces lying over $\{t=0\},$ so we can just as well view it as a distribution on a simpler space, $X^2_{b,s}\times \bbR^+,$
\begin{equation*}
	\beta_L^*(t\pa_t + tD^2_{\w,\eps})G_\infty \in 
	\rho_{hbf}^{-2}\rho_{hlf}^{-2} t^{\infty}
	\sB^{\cH/\df h}_{\phg}\sA^{-m-1}_-
		(X_{b,s}^2 \times \bbR^+;\Hom(E) \otimes \beta_{H,R}^*(\tfrac1\rho\Omega(M))).
\end{equation*}
This defines by convolution an operator on sections of $E$ over $X^2_s.$
In terms of the convolution product, $G_{\infty}$ satisfies
\begin{multline*}
	\beta_L^*(\pa_t + D^2_{\w,\eps})G_\infty = \Id - A \\
	\Mwith
	A \in 
	\rho_{hbf}^{-2}\rho_{hlf}^{-2}t^{\infty}
	\sB^{\cH/\df h}_{\phg}\sA^{-m-1}_-
		(X_{b,s}^2 \times \bbR^+;\Hom(E) \otimes \beta_{H,R}^*(\tfrac1\rho\Omega(M))).
\end{multline*}

\begin{lemma}
If $A \in \rho_{hbf}^{-2}\rho_{hlf}^{-2}t^{\infty} \sB^{\cH/\df h}_{\phg}\sA^{-m-1}_-(X_{b,s}^2 \times \bbR^+;\Hom(E) \otimes \beta_{H,R}^*(\tfrac1\rho\Omega(M)))$ then $\Id-A$ is invertible as an operator on $t^{\infty}\CI(X_s; E)$ with inverse $\Id - S$ for some $S$ in the same space as $A.$
\end{lemma}

\begin{proof}
Since $\df h(\bhs{hlf})-2+ \df h(\bhs{hrf})=2\nu_{\min}-2>0$, notice by Theorem~\ref{thm:Composition}  that taking the weight functions $\df h$ sufficiently small, we can assume that 
the space 
\begin{equation}
\rho_{hbf}^{-2}\rho_{hlf}^{-2}t^{\infty} \sB^{\cH/\df h}_{\phg}\sA^{-m-1}_-(X_{b,s}^2 \times \bbR^+;\Hom(E) \otimes \beta_{H,R}^*(\tfrac1\rho\Omega(M)))
\label{hkp.4}\end{equation}
is closed under composition.  It follows that the powers $A^k$ with respect to the convolution product are defined for all $k \in \bbN$.  Proceeding as in \cite[(7.66)]{MelroseAPS} and using the composition result for $b$-surgery operators of \cite[(105)]{mame1} , we thus see that for fixed $T>0$ and a choice of weighted $\cC^0$-norm of \eqref{hkp.4},  there exists positive constant $C$ and $K$ such that the weighted $\cC^0$ norm of $A^k$ is bounded by 
\begin{equation*}
	 K \frac{(Ct)^{k-1}}{(k-1)!} \quad \mbox{for} \; t\le T.
\end{equation*}
Thus the Volterra series inverting $\Id - A$ converges uniformly for $t \leq T$ and any $T \in \bbR.$   We can  apply a similar argument to the restriction of $A^k$ to $\bhs{hmf}$, and get at the same time control on $A^k-A^k\rest{\bhs{hmf}}$.  
We may also differentiate by a vector field tangent to the boundary hypersurfaces of $X^2_{b,s} \times \bbR^+$ and then apply the same argument, so that we can conclude that the Volterra series converges within the space \eqref{hkp.4}.
\end{proof}

We can now finally complete our uniform construction of the heat kernel.

\begin{theorem}\label{uhk.1}
Assume that Assumption~\ref{ur.1} holds. Let $\cG, \df g$ be the index set and weight function of \eqref{hkp.2b}.
The heat kernel of the twisted wedge surgery Hodge Laplacian is given by
\begin{equation*}
	e^{-tD_{\w,\eps}^2} = G_{\infty} (\Id - S) \in 
	\sB^{\cG/\df g}_{\phg}\sA^{-m-1}_-
		(HX_{\w,s};\Hom(E) \otimes \beta_{H,R}^*(\tfrac1\rho\Omega(M))).
\end{equation*}
In particular, the leading order terms at $\bhs{tf}$, $\bhs{sf}$ and $\bhs{hff}$ are given by \eqref{hk.9}, \eqref{leadingmodel:sf} and \eqref{leadingmodel:hff}, while the restriction at $\bhs{hmf}$ gives the wedge heat kernel of the Hodge Laplacian $D_{\w,0}^2$.
\end{theorem}

\begin{proof}
The operator $G_{\infty}(\Id-S)$ satisfies the wedge surgery heat equation with initial condition given by the (lift of the) identity since
\begin{equation}
	\beta_{H,L}(\pa_t + D_{\w,\eps}^2)(G_{\infty}(\Id-S)) = (\Id-A)(\Id-S) = \Id.
\label{hkp.6}\end{equation}
An application of the composition result in Theorem \ref{thm:Composition}
shows that the composition $G_{\infty} S$ is an element of 
\begin{equation}
	t^{\infty}
	\sB^{\cG/\df g}_{\phg}\sA^{-m-1}_-
		(X_{b,s}^2 \times \bbR^+;\Hom(E) \otimes \beta_{H,R}^*(\tfrac1\rho\Omega(M))).
\label{hkp.7}\end{equation}
In particular, it follows from \eqref{hkp.7} that $e^{-tD_{\w,\eps}^2}$ and $G_{\infty}$ have the same asymptotics at $\bhs{tf}$, $\bhs{sf}$ and $\bhs{hff}$.  Since by construction,
$G_{\infty}$ has the same leading terms as $G_0$ at these faces, this gives the claimed leading terms of $e^{-tD_{\w,\eps}^2}$ at these faces.  Clearly, the restriction of \eqref{hkp.6} to $\bhs{hmf}$ shows the  restriction of $e^{-tD_{\w,\eps}^2}$ is indeed the heat kernel of $D^2_{\w,0}$ with respect to its unique self-adjoint extension.  Indeed, recall that because of Assumption~\ref{ur.1} and Corollary~\ref{wss.6}, the operator $D^2_{\w,0}$ is essentially self-adjoint.
\end{proof}

\section{Wedge surgery and the trace of the heat kernel}  \label{wt.0}

In view of Mercer's theorem (see, e.g., \cite{Brislawn} for a discussion), the heat kernel is trace-class as its restriction to the diagonal is integrable and we have
\begin{equation*}
	\Tr(e^{-tD_{\w,\eps}^2}) = \int_M \cK_{e^{-tD_{\w,\eps}^2}}(\zeta, \zeta, \eps) \; d\zeta.
\end{equation*}
This trace is a function of $t$ and $\eps$ and the advantage of the integral expression is that we can make use of our description of the heat kernel from Theorem \ref{uhk.1} to determine the asymptotic behavior in $t$ and $\eps.$ To carry this out, our first step is to understand the lift $\wt \Delta_{HX}$ of the diagonal to the heat space. 

\subsection{The lifted diagonal}

The wedge surgery heat space $HX_{\w,s}$ is constructed from $M \times M  \times [0,1)_\eps \times  \bbR^+_{\tau}$ by a sequence of blow-ups and so comes equipped with a blow-down map,
\begin{equation*}
	HX_{\w,s} \xlra{\beta_{H}} M \times M \times [0,1)_\eps \times \bbR^+_{\tau}.
\end{equation*}
The interior lift of the diagonal is the closure in $HX_{\w,s}$ of the preimage of \linebreak $\diag_{M}\times (0,1)_{\eps}\times (0,\infty)_{\tau},$ and can be identified with a space constructed from $M \times [0,1)_{\eps}\times \bbR^+_{\tau} ,$
\begin{equation*}
	\wt \Delta_{HX} 
	= [M \times [0,1)_{\eps} \times \bbR^+_{\tau} ; 
	H \times \{0\}\times \bbR^+_{\tau};
	H \times \{0\}\times \{0\};
	M \times \{0\} \times \{0\}].
\end{equation*}
%
%
%
We denote the blow-down map to $X_s \times \bbR^+_{\tau} = [M\times [0,1)_{\eps} \times \bbR^+_{\tau} ; H \times \{0\}\times \bbR^+_{\tau} ]$ by
\begin{equation*}
	\beta_{\wt \Delta, (1)}: \wt \Delta_{HX} \lra X_s \times \bbR^+_{\tau}
\end{equation*}
and the boundary hypersurfaces of $\wt\Delta_{HX}$ by
\begin{equation*}
\begin{gathered}
	\bhs{tf}(\wt\Delta_{HX}) = \beta_{\wt \Delta, (1)}^{\sharp}(\{\tau=0\}), \quad
	\bhs{hmf}(\wt\Delta_{HX}) = \beta_{\wt \Delta, (1)}^{\sharp}(\bhs{\sm}(X_s)), \\
	\bhs{hbf}(\wt\Delta_{HX}) = \beta_{\wt \Delta, (1)}^{\sharp}(\bhs{\bs}(X_s)), \quad
	\bhs{hff}(\wt\Delta_{HX}) = \beta_{\wt \Delta, (1)}^{-1}(\bhs{\bs}\times \{ \tau=0\}), \\
	\bhs{sf}(\wt\Delta_{HX}) = \beta_{\wt \Delta, (1)}^{\sharp}(\{\eps=\tau=0\}).
\end{gathered}
\end{equation*}

This is precisely the same space that came up in \cite[\S7.2]{ARS1} when studying the behavior of the trace of the heat kernel under formation of fibered cusps. (The notation $\beta_{\wt \Delta, (1)}$ is used to be consistent with \cite[\S7.2]{ARS1}.)
As noted there, if we set
\begin{equation*}
	\sE\sT = [\bbR^+_{\tau} \times [0,1)_{\eps}; \{\tau=\eps=0\}],
\end{equation*}
then there is a natural lift of the map $M \times \bbR^+_{\tau} \times [0,1)_{\eps} \lra \bbR^+_{\tau} \times [0,1)_{\eps}$ to a b-fibration
\begin{equation*}
	\pi_{\eps,\tau}: \wt\Delta_{HX} \lra \sE\sT.
\end{equation*}

Let $\beta_{\eps,\tau}:\sE\sT \lra \bbR^+_{\tau} \times [0,1)_{\eps}$ denote the blow-down map, and denote the boundary hypersurfaces of $\sE\sT$ by
\begin{equation*}
	\bhs{tf}(\sE\sT) = \beta_{\eps,\tau}^\sharp(\{\tau=0\}), \quad
	\bhs{hff}(\sE\sT) = \beta_{\eps,\tau}^{-1}(0,0), \quad
	\bhs{af}(\sE\sT) = \beta_{\eps,\tau}^\sharp(\{\eps=0\}).
\end{equation*}
%

\subsection{The trace of the heat kernel}

The trace of the heat kernel is given by
\begin{equation*}
	\Tr(e^{-tD_{\w,\eps}^2}) = 
	(\beta_{\eps,\tau}\circ\pi_{\eps,\tau})_*(\tr\cK_{e^{-tD_{\w,\eps}^2}}|_{\wt\Delta_{HX}}).
\end{equation*}
Note that
\begin{multline*}
	e^{-tD_{\w,\eps}^2} \in 
	\sB^{\cG/\df g}_{\phg}\sA^{-m-1}_-
		(HX_{\w,s};\Hom(E) \otimes \beta_{H,R}^*(\tfrac1\rho\Omega(M))) \\
	\implies
	\tr\cK_{e^{-tD_{\w,\eps}^2}}|_{\wt\Delta_{HX}} \in
	\sB^{\wt {\cG}/\wt{\df g}}_{\phg}\sA^{-m-1}_-
		(\wt\Delta_{HX};\beta_{\wt\Delta, (1)}^*(\tfrac1\rho\Omega(M))) 
\end{multline*}
where $\wt \cG = \cG\rest{\wt\Delta_{HX}},$ $\wt{\df g} = \df g\rest{\wt\Delta_{HX}}$ are given by
\begin{equation*}
\begin{gathered}
\begin{multlined}
	\wt \cG(\bhs{tf}(\wt\Delta_{HX})) = \wt \cG(\bhs{sf}(\wt\Delta_{HX})) = -m, \quad
	\wt \cG(\bhs{hff}(\wt\Delta_{HX})) = -h, \\
	\wt{\df g}(\bhs{tf}(\wt\Delta_{HX})) = \wt{\df g}(\bhs{sf}(\wt\Delta_{HX})) = \wt{\df g}(\bhs{hff}(\wt\Delta_{HX})) = \infty,
\end{multlined}\\
	\wt \cG(\bhs{hmf}(\wt\Delta_{HX})) = 0 , \quad
	\wt{\df g}(\bhs{hmf}(\wt\Delta_{HX})) >0, \\
	\wt \cG(\bhs{hbf}(\wt\Delta_{HX})) = 2, \quad
	\wt{\df g}(\bhs{hbf}(\wt\Delta_{HX})) >2.
\end{gathered}
\end{equation*}

\begin{proposition}\label{prop:TrsEsT}
The trace of the heat kernel is a conormal function on $\sE\sT$ with a partial asymptotic expansion,
\begin{equation*}
	\beta_{\eps,\tau}^*\Tr(e^{-tD_{\w,\eps}^2}) \in 
	\sB^{\cJ/\df j}_{\phg}\sA^{-m-1}_-(\sE\sT),
\end{equation*}
where
\begin{equation*}
\begin{gathered}
	\cJ(\bhs{tf}(\sE\sT)) = -m, \quad
	\df j(\bhs{tf}(\sE\sT)) = \infty, \\
	\cJ(\bhs{hff}(\sE\sT)) = -h\bar\cup -m, \quad
	\df j(\bhs{hff}(\sE\sT)) = \infty, \\
	\cJ(\bhs{af}(\sE\sT)) = 0 \bar\cup 2 , \quad
	\df j(\bhs{af}(\sE\sT)) >0.
\end{gathered}
\end{equation*}
\end{proposition}

\begin{proof}
As in \cite[\S7.2]{ARS1}, let us write $\cK_{e^{-tD_{\w,\eps}^2}} = \wt \cK_{e^{-tD_{\w,\eps}^2}} \beta_{H,R}^*(\tfrac1\rho\mu_M)$ and multiply by $\beta_{\wt\Delta,(1)}^*(\mu_{\bbR^+_{\tau}\times [0,1)_{\eps}})$ to get
\begin{multline*}
	(\tr\cK_{e^{-tD_{\w,\eps}^2}}\rest{\wt\Delta_{HX}})\beta_{\wt\Delta,(1)}^*(\mu_{\bbR^+_{\tau}\times [0,1)_{\eps}}) \\
	= (\tr\wt \cK_{e^{-tD_{\w,\eps}^2}}\rest{\wt\Delta_{HX}})
	\beta_{\wt\Delta,(1)}^*(\tfrac1\rho\mu_{M \times \bbR^+_{\tau}\times [0,1)_{\eps}}) \\
	\in (\rho_{hbf}\rho_{hff})^{-1}(\rho_{hbf}\rho_{sf}\rho_{hff}^2)
	\sB^{\wt {\cH}/\wt{\df h}}_{\phg}\sA^{-m-1}_-
		(\wt\Delta_{HX};\Omega(\wt\Delta_{HX})).
\end{multline*}
If we push forward along $\pi_{\eps,\tau},$ then using \cite[Theorem 5]{me1} we get an element of
\begin{equation*}
	\rho_{hff}\sB^{\cJ/\df j}_{\phg}\sA^{-m-1}_-
		(\sE\sT;\Omega(\sE\sT)),
\end{equation*}
with $\cJ$ and $\df j$ as above. Finally we note that $\rho_{hff}\Omega(\sE\sT) = \beta_{\eps,\tau}^*\Omega(\bbR^+_{\tau}\times [0,1)_{\eps}),$ so we can cancel off the density we multiplied by, and obtain the result.
\end{proof}

\subsection{Symmetry for even metrics}

Recall that on a closed manifold, the small-time asymptotic expansion of the heat kernel
\begin{equation*}
	e^{-t\Delta} \sim t^{-m/2} \sum_{k\geq 0} a_{k} t^{k/2}
\end{equation*}
simplifies upon restricting to the diagonal,
\begin{equation*}
	e^{-t\Delta}\rest{\diag} \sim t^{-m/2} \sum_{k\geq 0} a_{2k} t^{k}.
\end{equation*}
This is shown in \cite[Chapter 7]{MelroseAPS} to follow from invariance of the model problems at $\bhs{tf}$ under reflection in the fibers. This argument applies to the wedge surgery heat kernel at both $\bhs{tf}$ and $\bhs{sf}$ so that we can replace the index sets at $\bhs{tf}(\sE\sT)$ and $\bhs{hff}(\sE\sT)$ from Proposition \ref{prop:TrsEsT} by
\begin{equation*}
	\cJ'(\bhs{tf}(\sE\sT)) = -m + 2\bbN_0, \quad
	\cJ'(\bhs{hff}(\sE\sT)) = -h\bar\cup (-m+2\bbN_0).
\end{equation*}
We now set out a more involved argument that allows us to further simplify the index set at $\bhs{hff}(\sE\sT).$

Analogously to \cite[\S7.3]{ARS1}, we consider a class of metrics for which the trace of the heat kernel simplifies. The class of metrics will depend on the choice of a tubular neighborhood $\sT = \mathrm{Tub}(H) = (-1,1)_x \times H,$ which we fix. Notice first that a product-type $\ew$-metric naturally induces on  $\beta^{-1}_s(\sT)\subset X_s$ a decomposition in terms of horizontal and vertical forms with respect to the fiber bundle $\beta_s^{-1}(\sT\times [0,1]_{\eps})\to Y$ induced $\phi: H\to Y$,
\begin{equation}
   \left. {}^{\ew} T^*X_s\right|_{\beta_s^{-1}(\sT\times [0,1]_{\eps})}= {}^{\ew}T^*_HX_s \oplus {}^{\ew}T^*_VX_s.            
\label{dc.1}\end{equation}

\begin{definition} An {\bf even $\ew$-metric} is a wedge surgery metric which in $\beta^{-1}_s(\sT\times[0,1]_{\eps})$ differs from a product-type $\ew$-metric by   elements of 
$$
\rho^2\CI(\beta^{-1}_s(\sT\times[0,1]_{\eps});{}^{\ew}T^*_HX_s\otimes {}^{\ew}T^*_HX_s)\quad \mbox{and} \quad \rho^2\CI(\beta^{-1}_s(\sT\times[0,1]_{\eps});{}^{\ew}T^*_VX_s\otimes {}^{\ew}T^*_VX_s)
$$  having only even powers of $\rho$ in their expansion at $\bhs{\bs}$.  
   If $g_{\ew}$ is a wedge surgery metric that differs from an even wedge surgery metric by $\rho^{\ell}$ times a smooth section of ${}^{\ew}T^*X_s \otimes {}^{\ew}T^*X_s,$ we say that $g_{\ew}$ is an {\bf even $\ew$ metric to order $\ell.$}
\label{em.1}\end{definition}

\begin{definition}
A bundle metric on a flat bundle $F$ is said to be even in $\rho$ if its Taylor expansion at $\bhs{\bs},$ in a collar neighborhood compatible with $\sT$ and a flat trivialization of $F,$ has only even powers of $\rho.$
\label{em.2}\end{definition}
Let 
\begin{equation}
\sB^{\cG/\df g}_{\phg,\even}\sA^{-m-1}_-
		(HX_{\w,s};\Hom(E) \otimes \beta_{H,R}^*(\tfrac1\rho\Omega(M)))
\label{em.3}\end{equation}
be the subspace of  $\sB^{\cG/\df g}_{\phg}\sA^{-m-1}_-
		(HX_{\w,s};\Hom(E) \otimes \beta_{H,R}^*(\tfrac1\rho\Omega(M)))$ consisting of elements $\kappa$ having an expansion at $\bhs{hff}$ of the form
\begin{equation}
   \rho_{hff}^h\kappa\sim \sum_{j=0}^{\infty}  \rho_{hff}^j\kappa_j
\label{em.4}\end{equation}
with $\kappa_j$ a conormal section on $\bhs{hff}$ such that in the  coordinates \eqref{coords:interiorhff}, we have that 
\begin{equation}
  \kappa_j(X, X',y, u, z,z')= (-1)^{j+N_{(H\times \bbR)/Y}}\kappa_j(X,X',y,-u,z,z'),
\label{em.5}\end{equation}
where $N_{(H\times \bbR)/Y}$ is the number operator giving the shift in vertical degree induced by $\kappa_j$  with respect to the fibered bundle 
$$
\phi\circ \pr_1:H\times \bbR\to Y,
$$ 
where $\pr_1:H\times \bbR\to H$ is the projection on the first factor. 
Similarly, we let 
\begin{equation}
\sB^{\cG/\df g}_{\phg,\odd}\sA^{-m-1}_-
		(HX_{\w,s};\Hom(E) \otimes \beta_{H,R}^*(\tfrac1\rho\Omega(M)))
\label{em.6}\end{equation}
be the subspace of $\sB^{\cG/\df g}_{\phg}\sA^{-m-1}_-
		(HX_{\w,s};\Hom(E) \otimes \beta_{H,R}^*(\tfrac1\rho\Omega(M)))$ consisting of elements $\kappa$ having an expansion at $\bhs{hff}$ of the form \eqref{em.4}
with 
\begin{equation}
  \kappa_j(X, X',y, u, z,z')= (-1)^{j+1+N_{(H\times \bbR)/Y}}\kappa_j(X,X',y,-u,z,z').
\label{em.7}\end{equation}

\begin{proposition}
If  $g_{\ew}$ and $g_F$ are even metrics, the heat kernel of $D^2_{\w,\eps}$ lies in \eqref{em.3}.  If $g_{\ew}$ and $g_F$ are only even to order $\ell$, then the heat kernel of 
$D^2_{\w,e}$ is in \eqref{em.3} up to a term in $\rho_{hff}^{\ell}\sB^{\cG/\df g}_{\phg}\sA^{-m-1}_-
		(HX_{\w,s};\Hom(E) \otimes \beta_{H,R}^*(\tfrac1\rho\Omega(M)))$.  		
\label{em.7}\end{proposition}
\begin{proof}
Since $G_{\infty}$ and $e^{-tD^2_{\w,\eps}}$ in Theorem~\ref{uhk.1} have the same expansion at $\bhs{hff}$, it suffices to show that $G_{\infty}$ can be constructed in \eqref{em.3}.  Clearly, the model $L_{hff}$ in \eqref{leadingmodel:hff}, which corresponds to $\kappa_0$ in the expansion \eqref{em.4}, satisfies \eqref{em.5} with $j=0$ and we can choose $G_0$ to be in \eqref{em.3}.   Now, it follows from \eqref{vf.1} and \cite[Proposition~10.1]{BGV} that the action of the exterior derivative $d_F$ from the left interchanges parity, that is, maps \eqref{em.3} into \eqref{em.6}.  For its adjoint, notice that since $g_{\ew}$ and $g_F$ are even, the operator $\sharp$ preserves parity, while the $*$-operator preserves or reverses parity depending on the parity of $v+1$.  Hence, we see that $\tau D_{\w,\eps}$ reverses parity, so that $\tau^2D^2_{\w,\eps}$ preserves parity.  Since the left action of $t\pa_t$ preserves parity, we see that the operator $t(\pa_t +D^2_{\w,\eps})$ preserves parity, which means that the construction of $G_{\infty}$ can be done in \eqref{em.3} as claimed.  Finally, if $g_{\w,\eps}$ and $g_F$ are only even to order the $\ell$, then clearly $G_{\infty}$ can be constructed in \eqref{em.3} up to a term in 
$\rho_{hff}^{\ell}\sB^{\cG/\df g}_{\phg}\sA^{-m-1}_-
		(HX_{\w,s};\Hom(E) \otimes \beta_{H,R}^*(\tfrac1\rho\Omega(M)))$.
\end{proof}
This has the following immediate consequence on the trace of the heat kernel.

\begin{corollary}\label{em.8}
If $m$ and $h$ are odd and $g_{\ew}, g_F$ are even metrics up to order $h+1$, then 
\begin{equation*}
	\beta_{\eps,\tau}^*\Tr(e^{-tD_{\w,\eps}^2}) 
\end{equation*}
does not have a term of order zero in its expansion at $\bhs{hff}$.
\end{corollary}

\subsection{Asymptotics of the determinant}

Recall that the zeta function of a Laplace-type operator on a closed manifold, $\Delta,$ is defined to be
\begin{equation*}
	\zeta(s) = \frac1{\Gamma(s)}\int_0^{\infty} \; t^s \; \Tr(e^{-t\Delta}-\Pi_{\ker\Delta})\; \frac{dt}t
\end{equation*}
for $\Re(s)\gg 0.$ The short-time asymptotic expansion of the trace of the heat kernel can be used to meromorphically continue this function to the entire complex plane with at worst simple poles and the determinant of $\Delta$ is defined to be $\exp(-\zeta'(0)).$

As is well-known (see, e.g., \cite[(10.3)]{ARS1}), if the short-time expansion of the heat kernel has the form
\begin{equation*}
	\Tr(e^{-t\Delta}) \sim t^{-m/2} \sum_{k\geq 0} a_{k/2}t^{k/2}
\end{equation*}
then the derivative of the zeta function at the origin is given by
\begin{equation}\label{eq:ZetaDerZero}
	\zeta'(0) = \sideset{^R}{_0^{\infty}}\int \Tr(e^{-t\Delta}) \frac{dt}t
	+ \gamma(a_{m/2}-\dim\ker\Delta)
\end{equation}
where $\gamma$ is the Euler-Mascheroni constant. If $m$ is odd then $a_{m/2}=0.$

For a Hodge Laplacian on a space with wedge singularities, the short-time asymptotics of the trace of the heat kernel involves not just powers of $t^{1/2},$ but also powers of $t^{1/2}$ multiplied by $\log t.$ Consequently, the meromorphic continuation of the integral of the trace of the heat kernel will have poles of order two and the zeta function is not {\em a priori} regular at $s=0.$ Dar \cite{Dar} showed that for spaces with conic singularities the linear combination of zeta functions occurring in the definition of analytic torsion is regular at the origin. Mazzeo-Vertman \cite[Proposition 4.3]{Mazzeo-Vertman} showed that on an odd dimensional space with wedge singularities, the individual zeta functions are regular at the origin. In particular it follows that that the derivative of the zeta functions at zero are given by \eqref{eq:ZetaDerZero} in this setting.\\

We have shown above that the trace of the wedge surgery heat kernel is partially polyhomogeneous on the space $\sE\sT.$ In particular, it satisfies
\begin{equation*}
\begin{gathered}
	\Tr(e^{-tD^2_{\ew}}) \sim \Tr(e^{-tD^2_{\w}})
	+ o(1) \Mat \bhs{af},  \\
	\Tr(e^{-tD^2_{\ew}}) \sim 
	\rho_{hff}^{-m}\sum_{k=0}^m A_k^{hff} \rho_{hff}^k
	+ \rho_{hff}^{-h}\sum_{\ell=0}^h \wt A^{hff}_{\ell} \rho_{hff}^{\ell} \log\rho_{hff}
	+ o(1) \Mat \bhs{hff}.
\end{gathered}
\end{equation*}
Moreover, if $g_{\ew}$ and $g_F$ are even metrics and $M$ and $Y$ are odd dimensional, Corollary~\ref{em.8} implies that $A_m^{hff}=0.$ 
In \cite[\S11]{ARS1}, we analyzed the asymptotics of the trace of the heat kernel undergoing analytic surgery to a fibered cusp metric. This analysis applies here with almost no change (e.g., here we have $\rho_{hff}^{-h}$ whereas there we had $\rho_{hff}^{-h-1}$). Applying that analysis here we have the following theorem.

\begin{theorem}
Denote the logarithm of the product of the positive small eigenvalues by $\log\zeta_{\mathrm{small}}.$   If Assumption~\ref{ur.1} holds, then
as $\eps\to0,$ 
\begin{multline*}
	\FP_{\eps=0} \left( \sideset{^R}{_0^{\infty}}\int \Tr(e^{-tD^2_{\ew}}) \frac{dt}t +\log\zeta_{\mathrm{small}}  \right)
	= 
	\sideset{^R}{_0^{\infty}}\int \Tr(e^{-tD^2_{\w}}) \frac{dt}t
	+\sideset{^R}{_0^{\infty}}\int A^{hff}_m \frac{d\sigma}\sigma 
	\\
	-\gamma(\dim \ker D^2_{\w} - \dim \ker D^2_{\ew})
	.
\end{multline*}
Thus if $m$ is odd then 
\begin{equation}
	\FP_{\eps=0} (\zeta_{\ew}'(0) + \log\zeta_{\mathrm{small}})
	= \zeta_{\w}'(0) 
	+\sideset{^R}{_0^{\infty}}\int A^{hff}_m \frac{d\sigma}\sigma.
\label{em.11}\end{equation}
If furthermore $h$ is odd and $g_{\w,\eps}$ and $g_{F}$ are even to order $h+1$, then
\begin{equation}
	\FP_{\eps=0} (\zeta_{\ew}'(0) + \log\zeta_{\mathrm{small}})
	= \zeta_{\w}'(0).
\label{em.12}\end{equation}
\label{em.10}\end{theorem}

\section{A Cheeger-M\"uller theorem for wedge metrics} \label{cm.0}

Let $M$ be an oriented closed manifold, $H\subset M$ a co-oriented hypersurface and \linebreak $c: (-1,1)_x\to M$ a choice of tubular neighborhood.  Let $\phi: H\to Y$ be a smooth fiber bundle with total space $H$ and base $Y$ a closed oriented manifold.  Let $\alpha: \pi_1(M)\to \GL(k,\bbR)$ be a unimodular representation and let $F\to M$ be the corresponding flat vector bundle.  Let $g_{\ew}$ and $g_{F}$ be a wedge-surgery metric and bundle metric that are even to order $h+1$ and  suppose that Assumption~\ref{ur.1} holds.  By Corollary~\ref{wss.6} and Corollary~\ref{wss.7}, we then know that the wedge Hodge Laplacian $\Delta_{\w,0}= \eth_{\w,0}^2$ is essentially self-adjoint with minimal extension given by
$$
      \cD_{\min}(\Delta_{\w,0})= r^2H^2_{\w}(\bhs{\sm}; \Lambda^*({}^{\w}T^*\bhs{\sm})\otimes F).  
$$ 

On the other hand , let $\widehat{M\setminus H}$ denote the stratified space obtained from $M$ by collapsing the fibers of $\phi:H\to Y$ onto the base $Y$.  Since $\mathrm H^*(H/Y;F)=0$, we know from \cite[Proposition~1]{hhm} that there is a canonical isomorphism between the intersection cohomology with upper middle perversity and with values in $F$ of  the stratified space $\widehat{M\setminus H}$  and the relative cohomology of the manifold with boundary $\bhs{\sm}$,
$$
    \IH^*_{\bm}(\widehat{M\setminus H};F)\cong \mathrm H^*(\bhs{\sm}, \pa \bhs{\sm};F).
$$
Intersection cohomology groups can also be identified with $L^2$-harmonic forms.
\begin{proposition}[Hunsicker-Mazzeo \cite{Hunsicker-Mazzeo}]
There is a canonical isomorphism between intersection cohomology and the space $   \cH^*_{g_{\w,0}}(M\setminus H;F)$ of $L^2$-harmonic forms taking values in $F$ of the wedge metric $g_{\w,0}$, so that 
$$
   \cH^*_{g_{\w,0}}(M\setminus H;F)\cong   \IH^*_{\bm}(\widehat{M\setminus H};F) \cong \mathrm H^*(\bhs{\sm}, \pa \bhs{\sm};F).
$$
\label{ihh.1}\end{proposition}
\begin{proof}
Recall first from \cite{Hunsicker-Mazzeo} that in general, taking the maximal or minimal extension of the exterior derivative leads to two notions of $L^2$-cohomology, namely the maximal $L^2$-cohomology groups $\mathrm H^*_{\max}(M\setminus H; g_{\w,0},F)$ and the minimal $L^2$-cohomology groups \linebreak  $\mathrm H^*_{\min}(M\setminus H; g_{\w,0},F)$.  In our case, since $\eth_{\w,0}$
is essentially self-adjoint, we know a fortiori, see for instance \cite[Proposition~4.6]{Hunsicker-Mazzeo}, that the exterior derivative has only one closed extension, so that in fact
$$
  \mathrm H^*_{\max}(M\setminus H; g_{\w,0},F)= \mathrm H^*_{\min}(M\setminus H; g_{\w,0},F).
$$
By \cite[Corollary~3.19]{Hunsicker-Mazzeo}, we also have that 
\begin{equation}
        \mathrm H^*_{\max}(M\setminus H; g_{\w,0},F)= \mathrm H^*_{\min}(M\setminus H; g_{\w,0},F)\cong  \IH^*_{\bm}(\widehat{M\setminus H};F).
 \label{ihh.2}\end{equation} 
Technically speaking Corollary~3.19 in \cite{Hunsicker-Mazzeo} is only stated in the case where $F$ is the trivial flat line bundle.  However, the arguments leading to the proof of \cite[Corollary~3.19]{Hunsicker-Mazzeo} work as well for any flat vector bundle provided one makes the appropriate notational changes.  

In particular, we deduce from \eqref{ihh.2} that the $L^2$-cohomology is finite dimensional.  By \cite[Theorem~2.4]{BL92}, the corresponding Hilbert complex is Fredholm and we have a Kodaira decomposition, so that finally 
$$
\cH^*_{g_{\w,0}}(M\setminus H;F)\cong   \IH^*_{\bm}(\widehat{M\setminus H};F)
$$
as claimed.
\end{proof}

Another consequence of Assumption~\ref{ur.1}, which follows from  the Leray-Serre spectral sequence of the fiber bundle $\phi:H\to Y$, is that $\mathrm H^*(H;F)=0$.  Thus, we see from the long exact sequence in cohomology associated to the pair $(M,H)$ that there is also an isomorphism 
$$
          \mathrm H^*(\bhs{\sm},\pa\bhs{\sm};F)\cong \mathrm H^*(M;F),  
$$ 
so that finally
$$
   \cH^*_{g_{\w,0}}(M\setminus H;F)\cong \mathrm H^*(M;F)\cong  \cH^*_{g_{\ew}}(M;F) \quad \mbox{for} \; \epsilon>0.
$$
As a consequence, we see that the projection $\Pi_{\sma}$ of Corollary~\ref{small.2} is just the projection onto the $L^2$-harmonic forms.  In particular, there are no positive small eigenvalues and Theorem~\ref{ur.8} ensures that there is a uniform spectral gap, namely that there is  $\delta>0$ such that 
$$
    \spec(\eth_{\dR,\epsilon}^2)\cap(0,\delta)=\emptyset \quad \forall \ \epsilon\in [0,1].
$$
This leads to the following result.

\begin{theorem}
Let $M$ be a closed odd dimensional oriented manifold with co-oriented hypersurface $H\subset M$ equipped with a fibered bundle $\phi: H\to Y$ where $Y$ is a odd dimensional oriented compact manifold.  Let $g_{\ew}$ be a wedge surgery metric even to order $\ell$.  Let also $\alpha:\pi_1(M)\to \GL(k,\bbR)$ be a unimodular representation and denote by $F$ the corresponding flat vector bundle.  Let $g_F$ a bundle metric even to order $\dim Y+1$ and suppose that Assumption~\ref{ur.1} holds for the wedge surgery de Rham operators $\eth_{\ew}$ associated to $g_{\ew}$ and $g_F$.  Then we have the equality
$$
  \bar T(\bhs{\sm},g_{\w,0},F,g_F,\mu^*)=  \tau(\bhs{\sm},\pa\bhs{\sm},\alpha,\mu^*) \tau(H,\alpha)
$$
where $\mu^j$ is any choice of basis of $\mathrm H^{j}(\bhs{\sm},\bhs{\sm};F)$.  
\label{cm.4}\end{theorem}

\begin{proof}
By the discussion above, there is no positive small eigenvalues.  Since  $\dim Y$ is odd, we therefore know  by Theorem~\ref{em.10} that 
\begin{equation}
  \FP_{\epsilon=0} \log T(M,g_{\ew},F,g_F)= \log T(M\setminus H, g_{\w,0},F, g_F).
\label{cm.5}\end{equation}
On the other hand, since  the projection $\Pi_{\sma}$ of Corollary~\ref{small.2} is an element of  \linebreak $\rho(\Psi^{-\infty,\delta}_{\ee}(X_s; E))\rho$   for some $\delta>0$, where $E= \Lambda^*({}^{\ew}T^*X_s)\otimes F$, we know that if $\omega^q_{\epsilon}$ is any orthonormal basis of harmonic forms for $g_{\ew}$, then
\begin{equation}
  \lim_{\epsilon\to 0} \log\left( \prod_{q=0}^n [\mu^q | \omega^q_{\epsilon}]^{(-1)^q} \right)= \log \left(  \prod_{q=0}^n [\mu^q | \omega_{0}^q]^{(-1)^q}\right).
\label{cm.6}\end{equation} 
Combining \eqref{cm.5} with \eqref{cm.6}, we thus see that
$$
    \FP_{\eps=0}\log\bar T(M, g_{\ew},F, g_F, \mu^*)= \log\bar T(M\setminus H, g_{\w,0},F,g_F,\mu^*).
$$
Applying \cite[Theorem~8.7]{ARS1}, we have also that 
$$
   \tau(M,\alpha, \mu^*)= \tau(\bhs{\sm},\pa \bhs{\sm},\alpha,\mu^*)\tau(H,\alpha),
$$
from which the result follows.  

\end{proof}

As a direct consequence of this result,  we obtain our main result, Theorem~\ref{int.6}.  

\begin{proof}[Proof of Theorem~\ref{int.6}]
Let $M:= \bar N\cup_{\pa \bar N} \bar N$ be the double of $\bar N$.  Then $M$ has a natural co-oriented hypersurface $H$ which is identified with the boundary of each of the two copies of $\bar N$ in $M$.   Since the wedge metric $g_{\w}$ is even, it is the restriction on each copy of $\bar N$ in $M$ of $g_{\w,0}$ for $g_{\ew}$ some even wedge surgery metric on $M$ . Similarly, $g_F$ can be seen as the restriction of an Euclidean metric $g_{\widetilde{F}}$ for the double $\widetilde{F}$ of $F$ on $M$.  Clearly,
$$
     2 \log\bar T(N,g_{\w},F,\mu^*)= \log\bar T(M\setminus H, g_{\w,0},\widetilde{F}, g_{\widetilde{F}}, \mu^* \sqcup \mu^* )
$$
and 
so the result follows by applying Theorem~\ref{cm.4} on $M$.
\end{proof}

When $\dim Y$ is even, the above argument does not work due to the fact that there is another term coming from the face $\bhs{hff}$, namely the second term on the right hand side of \eqref{em.11}, that can possibly contribute to the right hand side of \eqref{cm.5}. However, this extra contribution can be  interpreted as the analytic torsion of a wedge metric on an appropriate space.
Indeed, consider in this  case the oriented closed odd dimensional manifold $M'= \bbS^1\times H$ with flat vector bundle $F'=\pi_2^*(\left.  F\right|_{H})$, where $\pi_2: \bbS^1\times H\to H$ is the projection on the second factor.  Since $\mathrm H^*(H;F)=0$, so that $\mathrm H^*(\bbS^1\times H;F')=0$, we know from \cite[Proposition~1.13]{Muller1993} and the Cheeger-M\"uller theorem \cite{Muller1993} that the analytic torsion and the Reidemeister torsion for $F'$ on $\bbS^1\times H$ are both equal to one.  Thus, applying the strategy of the proof of Theorem~\ref{cm.4} to any exact wedge surgery metric $g_{\ew}'$ on $\bbS^1\times H$  and a bundle metric $g_{F'}$ which are respectively  identified with $g_{\ew}$ and $g_F$ in a collar neighborhood of $\{1\}\times H$ in $\bbS^1\times H$, where $\bbS^1\subset \bbC$ is thought of as the unit circle, we obtain that the extra contribution coming from $\bhs{hff}$ in \eqref{cm.5} is precisely equal to minus  the logarithm of the analytic torsion of $g_{\w,0}'$,
$$
   -\log T((\bbS^1\setminus\{1\})\times H, g_{\w,0}', F',g_{F'}).
$$
This leads to the following relative Cheeger-M\"uller theorem, essentially a reformulation of  the gluing formula of Lesch \cite[Theorem~1.2]{Lesch2013} in this special setting. 
\begin{corollary}
Consider the same setting as the one of Theorem~\ref{int.6} with the difference that we assume instead that $Y$ is even dimensional.  Then 
$$
\frac{\bar T(N,g_{\w},F,\mu^*)}{ T((\bbS^1\setminus \{1\})\times H, g_{\w,0}', F',g_{F'})^{\frac12}} = \tau(\bar N, \pa \bar N,\alpha, \mu^*)\tau (\pa \bar N,\alpha)^{\frac12}
$$
for any choice of basis $\mu^q$ of $H^q(\bar N,\pa \bar N;F)$ for $q=0,\ldots, \dim \bar N$.  
\label{cm.9}\end{corollary}

\bibliographystyle{amsalpha}
\bibliography{wedgesurgery}

\providecommand{\bysame}{\leavevmode\hbox to3em{\hrulefill}\thinspace}
\providecommand{\MR}{\relax\ifhmode\unskip\space\fi MR }
\providecommand{\MRhref}[2]{%
  \href{http://www.ams.org/mathscinet-getitem?mr=#1}{#2}
}
\providecommand{\href}[2]{#2}
\begin{thebibliography}{Mc{D}90}

\bibitem[AGR]{Albin-GellRedman}
Pierre Albin and Jess Gell-Redman, \emph{The index formula for families of
  {D}irac type operators on pseudomanifolds}, Available online at
  arXiv:1712.08513.

\bibitem[Alb07]{Albin:Heat}
Pierre Albin, \emph{A renormalized index theorem for some complete
  asymptotically regular metrics: the {G}auss-{B}onnet theorem}, Adv. Math.
  \textbf{213} (2007), no.~1, 1--52.

\bibitem[ALMP12]{ALMP:Witt}
Pierre Albin, {\'E}ric Leichtnam, Rafe Mazzeo, and Paolo Piazza, \emph{The
  signature package on {W}itt spaces}, Ann. Sci. \'Ec. Norm. Sup\'er. (4)
  \textbf{45} (2012), no.~2, 241--310.

\bibitem[ALMP18]{ALMP:Hodge}
Pierre Albin, Eric Leichtnam, Rafe Mazzeo, and Paolo Piazza, \emph{Hodge theory
  on {C}heeger spaces}, J. Reine Angew. Math. \textbf{744} (2018), 29--102.

\bibitem[ARS17]{ARS1}
Pierre Albin, Fr{\'e}d{\'e}ric Rochon, and David Sher, \emph{Resolvent, heat
  kernel, and torsion under degeneration to fibered cusps}, arXiv:1410.8406, to
  appear in Mem. Amer. Math. Soc., 2017.

\bibitem[ARS18]{ARS2}
Pierre Albin, Fr\'{e}d\'{e}ric Rochon, and David Sher, \emph{Analytic torsion
  and {R}-torsion of {W}itt representations on manifolds with cusps}, Duke
  Math. J. \textbf{167} (2018), no.~10, 1883--1950.

\bibitem[BGV04]{BGV}
Nicole Berline, Ezra Getzler, and Mich{\`e}le Vergne, \emph{Heat kernels and
  {D}irac operators}, Grundlehren Text Editions, Springer-Verlag, Berlin, 2004,
  Corrected reprint of the 1992 original.

\bibitem[BL92]{BL92}
J.~Br{\"u}ning and M.~Lesch, \emph{Hilbert complexes}, J. Funct. Anal.
  \textbf{108} (1992), no.~1, 88--132. \MR{1174159 (93k:58208)}

\bibitem[BMZ17]{BMZ}
Jean-Michel Bismut, Xiaonan Ma, and Weiping Zhang, \emph{Asymptotic torsion and
  {T}oeplitz operators}, J. Inst. Math. Jussieu \textbf{16} (2017), no.~2,
  223--349. \MR{3615411}

\bibitem[Bri88]{Brislawn}
Chris Brislawn, \emph{Kernels of trace class operators}, Proc. Amer. Math. Soc.
  \textbf{104} (1988), no.~4, 1181--1190.

\bibitem[BV13]{BV}
Nicolas Bergeron and Akshay Venkatesh, \emph{The asymptotic growth of torsion
  homology for arithmetic groups}, J. Inst. Math. Jussieu \textbf{12} (2013),
  no.~2, 391--447.

\bibitem[BZ92]{Bismut-Zhang}
Jean-Michel Bismut and Weiping Zhang, \emph{An extension of a theorem by
  {C}heeger and {M}{\"u}ller}, Ast\'erisque (1992), no.~205, 235, With an
  appendix by Francois Laudenbach. \MR{1185803 (93j:58138)}

\bibitem[Che79]{Cheeger1979}
Jeff Cheeger, \emph{Analytic torsion and the heat equation}, Ann. of Math. (2)
  \textbf{109} (1979), no.~2, 259--322.

\bibitem[CV19]{Calegari-Venkatesh}
Frank Calegari and Akshay Venkatesh, \emph{A torsion {J}acquet-{L}anglands
  correspondence}, Ast\'erisque (2019), no.~205, viii+226.

\bibitem[Dar87]{Dar}
Aparna Dar, \emph{Intersection {$R$}-torsion and analytic torsion for
  pseudomanifolds}, Math. Z. \textbf{194} (1987), no.~2, 193--216.

\bibitem[Fra35]{Franz1935}
W.~Franz, \emph{{\"{U}}ber die {T}orsion einer {\"{u}}berdeckung}, J. f{\"u}r
  die reine und angew. Math. \textbf{173} (1935), 245--253.

\bibitem[GH08]{GH1}
Colin Guilarmou and Andrew Hassell, \emph{The resolvent at low energy and
  {R}iesz transform for {S}chr{\"o}dinger operators on asymptotically conic
  manifolds, {P}art {I}}, Math. Ann. \textbf{341} (2008), no.~4, 859--896.

\bibitem[GH09]{GH2}
\bysame, \emph{The resolvent at low energy and {R}iesz transform for
  {S}chr{\"o}dinger operators on asymptotically conic manifolds, {P}art
  {I}{I}}, Ann. Inst. Fourier \textbf{59} (2009), no.~2, 1553--1610.

\bibitem[GKM13]{GKM}
J.B. Gil, T.~Krainer, and G.A. Mendoza, \emph{On the closure of elliptic wedge
  operators}, J. Geom. Anal. \textbf{23} (2013), no.~4, 2035--2062.
  \MR{3107690}

\bibitem[GM80]{GM1980}
Mark Goresky and Robert MacPherson, \emph{Intersection homology theory},
  Topology \textbf{19} (1980), no.~2, 135--162.

\bibitem[GM03]{Gil-Mendoza}
Juan~B. Gil and Gerardo~A. Mendoza, \emph{Adjoints of elliptic cone operators},
  Amer. J. Math. \textbf{125} (2003), no.~2, 357--408.

\bibitem[GS14]{Guillarmou-Sher}
Colin Guillarmou and David Sher, \emph{Low energy resolvent for the {H}odge
  {L}aplacian: {A}pplications to {R}iesz transform, {S}obolev estimates and
  analytic torsion}, Int. Math. Res. Not (2014), 1--75.

\bibitem[Has98]{Hassell}
Andrew Hassell, \emph{Analytic surgery and analytic torsion}, Comm. Anal. Geom.
  \textbf{6} (1998), no.~2, 255--289.

\bibitem[HHM04]{hhm}
Tam{\'a}s Hausel, Eugenie Hunsicker, and Rafe Mazzeo, \emph{Hodge cohomology of
  gravitational instantons}, Duke Math. J. \textbf{122} (2004), no.~3,
  485--548.

\bibitem[HLV18]{HLV}
L.~Hartmann, M.~Lesch, and B.~Vertman, \emph{On the domain of {D}irac and
  {L}aplace type operators on stratified spaces}, J. Spectr. Theory \textbf{8}
  (2018), no.~4, 1295--1348.

\bibitem[HM05]{Hunsicker-Mazzeo}
Eugenie Hunsicker and Rafe Mazzeo, \emph{Harmonic forms on manifolds with
  edges}, Int. Math. Res. Not. (2005), no.~52, 3229--3272. \MR{2186793
  (2006m:58027)}

\bibitem[HMM95]{hmm}
Andrew Hassell, Rafe Mazzeo, and Richard~B. Melrose, \emph{Analytic surgery and
  the accumulation of eigenvalues}, Comm. Anal. Geom. \textbf{3} (1995),
  no.~1-2, 115--222.

\bibitem[KM14]{Krainer-Mendoza}
Thomas Krainer and Gerardo Mendoza, \emph{Boundary value problems for elliptic
  wedge operators: the first order case}, J. {E}scher, {E}. {S}chrohe, {J}.
  {S}eiler, and {C}. {W}alker ({E}ds.), {E}lliptic and {P}arabolic {E}quations,
  vol. 119, Springer Proceedings in Mathematics {$\&$} Statistics, 2014.

\bibitem[Kot15]{Kottke}
Chris Kottke, \emph{A {C}allias-type index theorem with degenerate potentials},
  Comm. Partial Differential Equations \textbf{40} (2015), no.~2, 219--264.
  \MR{3277926}

\bibitem[Lau03]{Lauter}
Robert Lauter, \emph{Pseudodifferential analysis on conformally compact
  spaces}, Mem. Amer. Math. Soc. \textbf{163} (2003), no.~777, xvi+92.
  \MR{1965451}

\bibitem[Les13]{Lesch2013}
Matthias Lesch, \emph{A gluing formula for the analytic torsion on singular
  spaces}, Anal. PDE \textbf{6} (2013), no.~1, 221--256. \MR{3068545}

\bibitem[Lud18]{Ludwig}
Ursula Ludwig, \emph{An extension of a theorem by {C}heeger and {M}\"uller to
  spaces with isolated conical singularities}, C. R. Math. Acad. Sci. Paris
  \textbf{356} (2018), no.~3, 327--332. \MR{3767605}

\bibitem[Maz91]{maz91}
Rafe Mazzeo, \emph{Elliptic theory of differential edge operators. {I}}, Comm.
  Partial Differential Equations \textbf{16} (1991), no.~10, 1615--1664.

\bibitem[Mc{D}90]{McDonald}
Patrick Mc{D}onald, \emph{The {L}aplacian on spaces with cone-like
  singularities}, MIT Thesis, 1990.

\bibitem[Mel92]{me1}
Richard~B. Melrose, \emph{Calculus of conormal distributions on manifolds with
  corners}, Internat. Math. Res. Notices (1992), no.~3, 51--61.

\bibitem[Mel93]{MelroseAPS}
\bysame, \emph{The {A}tiyah-{P}atodi-{S}inger index theorem}, Research Notes in
  Mathematics, vol.~4, A K Peters Ltd., Wellesley, MA, 1993.

\bibitem[MM87]{Mazzeo-Melrose:Zero}
Rafe Mazzeo and Richard~B. Melrose, \emph{Meromorphic extension of the
  resolvent on complete spaces with with asymptotically negative curvature}, J.
  Funct. Anal. (1987), 260--310.

\bibitem[MM95]{mame1}
\bysame, \emph{Analytic surgery and the eta invariant}, Geom. Funct. Anal.
  \textbf{5} (1995), no.~1, 14--75.

\bibitem[MM13]{Marshall-Muller}
Simon Marshall and Werner M{\"u}ller, \emph{On the torsion in the cohomology of
  arithmetic hyperbolic 3-manifolds}, Duke Math. J. \textbf{162} (2013), no.~5,
  863--888.

\bibitem[MP13]{Muller-Pfaff:ATL2TorsionCmptLocSymSpaces}
Werner M{\"u}ller and Jonathan Pfaff, \emph{Analytic torsion and
  {$L^2$}-torsion of compact locally symmetric manifolds}, J. Differential
  Geom. \textbf{95} (2013), no.~1, 71--119.

\bibitem[MP14]{Muller-Pfaff:GrowthTorsionCohoArithGps}
\bysame, \emph{On the growth of torsion in the cohomology of arithmetic
  groups}, Math. Ann. \textbf{359} (2014), no.~1-2, 537--555.

\bibitem[M{\"u}l78]{Muller1978}
Werner M{\"u}ller, \emph{Analytic torsion and {$R$}-torsion of {R}iemannian
  manifolds}, Adv. in Math. \textbf{28} (1978), no.~3, 233--305.

\bibitem[M{\"u}l93]{Muller1993}
\bysame, \emph{Analytic torsion and {$R$}-torsion for unimodular
  representations}, J. Amer. Math. Soc. \textbf{6} (1993), no.~3, 721--753.

\bibitem[M{\"u}l12]{Muller:AsympRSTHyp3Mfds}
\bysame, \emph{The asymptotics of the {R}ay-{S}inger analytic torsion for
  hyperbolic {$3$}-manifolds}, Metric and {D}inferential Geometry. {T}he {J}eff
  {C}heeger {A}nniversary {V}olume, Progress in Math., vol. 297,
  Birkh{{\"a}}user, 2012, pp.~317--352.

\bibitem[MV12]{Mazzeo-Vertman}
Rafe Mazzeo and Boris Vertman, \emph{Analytic torsion on manifolds with edges},
  Adv. Math. \textbf{231} (2012), no.~2, 1000--1040.

\bibitem[Pfa14]{Pfaff:ExpGrowthHomTorTowerCongSubgpsBianchi}
Jonathan Pfaff, \emph{Exponential growth of homological torsion for towers of
  congruence subgroups of {B}ianchi groups}, Ann. Global Anal. Geom.
  \textbf{45} (2014), no.~4, 267--285.

\bibitem[Pfa17]{Pfaff:GluingFormATHypMfdsCusps}
\bysame, \emph{A gluing formula for the analytic torsion on hyperbolic
  manifolds with cusps}, J. Inst. Math. Jussieu \textbf{16} (2017), no.~4,
  673--743. \MR{3680342}

\bibitem[Rai19a]{Raimbault:ARHtorsion}
Jean Raimbault, \emph{Analytic, {R}eidemeister and homological torsion for
  congruence three--manifolds}, Ann. Fac. Sci. Toulouse Math. \textbf{28}
  (2019), no.~3, 417--469.

\bibitem[Rai19b]{Raimbault:Asymp}
\bysame, \emph{Asymptotics of analytic torsion for hyperbolic three-manifolds},
  Comment. Math. Helv. \textbf{94} (2019), no.~3, 459--531.

\bibitem[Rei35]{Reidemeister1935}
Kurt Reidemeister, \emph{Homotopieringe und {L}insenr{\"a}ume}, Abh. Math. Sem.
  Univ. Hamburg \textbf{11} (1935), no.~1, 102--109. \MR{3069647}

\bibitem[RS71]{Ray-Singer}
D.~B. Ray and I.~M. Singer, \emph{{$R$}-torsion and the {L}aplacian on
  {R}iemannian manifolds}, Advances in Math. \textbf{7} (1971), 145--210.

\bibitem[She13]{Sher:ACHeat}
David Sher, \emph{The heat kernel on an asymptotically conic manifold}, Anal.
  \& PDE \textbf{6} (2013), no.~7, 1755--1791.

\bibitem[SS20]{SS16}
R.~Seyyedali and G.~Székelyhidi, \emph{Extremal metrics on blowups along
  submanifolds}, J. Differential Geom. \textbf{114} (2020), no.~1, 171--192.

\bibitem[Ver19]{Vertman:CheegerMuller}
Boris Vertman, \emph{Cheeger-{M}\"{u}ller theorem on manifolds with cusps},
  Math. Z. \textbf{291} (2019), no.~3-4, 761--819.

\end{thebibliography}

\end{document}